\documentstyle[amssymb,amsfonts]{amsart}

\def\be#1{\begin{equation}\label{#1}}
\def\bas{\begin{align*}}
\def\eas{\end{align*}}
\def\bi{\begin{itemize}}
\def\ei{\end{itemize}}

\theoremstyle{plain}
   \newtheorem{theorem}[subsection]{Theorem}

   \newtheorem{proposition}[subsection]{Proposition}
   
   \newtheorem{lemma}[subsection]{Lemma}
   \newtheorem{corollary}[subsection]{Corollary}

\newenvironment{proof}{\noindent {\bf Proof} }{\endprf\par}
\def \endprf{\hfill  {\vrule height6pt width6pt depth0pt}\medskip}
\def\emph#1{{\it #1}}
\def\textbf#1{{\bf #1}}

\theoremstyle{remark}

\theoremstyle{definition}
   \newtheorem{definition}[subsection]{Definition}

\begin{document}

\author{James Wright}
\address{School of Mathematics and Maxwell Institute for Mathematical Sciences, University of
Edinburgh, JCMB, King's Buildings, Peter Guthrie Tait Road, Edinburgh EH9 3FD, Scotland}
\email{J.R.Wright@@ed.ac.uk}

\subjclass{42B05, 42B08, 42B20}

%\thanks{The author was supported in part by an EPSRC grant.}

\title[On a problem of Kahane]{On a problem of Kahane in
higher dimensions}

\maketitle

\centerline{\it Dedicated to Mike Christ on his 60th birthday}

\begin{abstract}
We characterise those real analytic mappings from ${\Bbb T}^k$ to ${\Bbb T}^d$ which
map absolutely convergent Fourier series on ${\Bbb T}^d$ to uniformly
convergent Fourier series via composition. We do this with respect to rectangular summation
on ${\Bbb T}^k$ (more precisely, unrestricted rectangular summation). We
also investigate uniform convergence with respect to square sums and highlight the differences
which arise.
\end{abstract}

%\maketitle

\section{Introduction}
The famous Beurling-Helson theorem \cite{BH} from the early 1950's
states that the only mappings of the
circle $\Phi : {\Bbb T} \to {\Bbb T}$ which preserve the space $A({\mathbb T})$
of absolutely convergent Fourier series are affine maps
 (see also \cite{L} and \cite{K}). More precisely, if
$\Phi : {\mathbb T} \to {\mathbb T}$ has the property that 
$$
f\circ \Phi \in A({\mathbb T}) \ \ {\rm whenever} \ \  f\in A({\mathbb T}), \ {\rm then}\ \
\Phi(e^{2\pi i t} ) \ = \ e^{2\pi i (c + k t)}
$$ 
for
some constant $c\in {\mathbb R}$ and integer $k\in {\mathbb Z}$. In a remarkable paper \cite{C},  P. Cohen
extended this to
mappings between any two locally compact abelian groups $G_1$ and $G_2$
with the same rigidity outcome;  only {\it piecewise affine} maps $\Phi : G_1 \to G_2$
transport $A(G_2)$ to $A(G_1)$ via composition with $\Phi$.

Kahane then asked
what happens when we relax the condition of absolute convergence to uniform convergence;
see for example, \cite{K-1}.
Are there nonlinear mappings $\Phi : {\mathbb T} \to {\mathbb T}$ (and if so, which ones) that  carry
the space $A({\mathbb T})$ to the space of uniformly convergent Fourier series $U({\mathbb T})$ via composition
with $\Phi$? In this direction, we have the following nice result of L. Alp\'ar \cite{A}.

{\bf Alp\'ar's Theorem} {\it For any real-analytic mapping $\Phi : {\mathbb T} \to {\mathbb T}$, we have
$f\circ \Phi \in U({\mathbb T})$ whenever $f\in A({\mathbb T})$. Furthermore this property
can fail for smooth $\Phi$; more precisely, there exists $\Phi \in C^{\infty}({\mathbb T})$ and
$f\in A({\mathbb T})$ such that $f\circ \Phi \notin U({\mathbb T})$.}

Alp\'ar's original proof was very complicated but along the way (as  people asked what happens between the real-analytic and $C^{\infty}$ categories
of mapping of the circle; see for example,  \cite{Kau}, \cite{Santos} and \cite{W}) simplifications were made.
Most notable is the work of R. Kaufman \cite{Kau} where some interesting connections were established
between Kahane's
problem and the theory of thin sets in harmonic analysis. 
 
In this paper, we investigate what happens with mappings $\Phi : {\mathbb T}^k \to {\mathbb T}^d$
between higher dimensional tori. Since ${\mathbb T}^k$ and ${\mathbb T}^d$ are two examples
of compact abelian groups, Cohen's result applies and we see that the only mappings $\Phi$
which preserve the space of absolutely convergent Fourier series are affine ones. As Kahane proposed, we
ask here what happens when we relax the condition of absolute convergence to uniform convergence. We
begin by stating a simple extension of Alp\'ar's Theorem.

\begin{proposition}\label{alpar-ext} Let $\Phi : {\mathbb T} \to {\mathbb T}^d$ be any real-analytic
mapping. Then $f\circ\Phi \in U({\mathbb T})$ whenever $f\in A({\mathbb T}^d)$. 
\end{proposition}

The proof of Proposition \ref{alpar-ext} is an adaptation of  Kaufman's proof of Alp\'ar's theorem.
We give the proof in the next section. Much more interesting is what happens when the domain ${\mathbb T}^k$
is higher dimensional; that is, when $k\ge 2$. In this paper, we will restrict ourselves to 
$k=2$. Already in this case, we will see most of the new phenomena and challenges. 
The situation for larger values of $k$ is similar but matters become much more technical. As we will
see, the analysis is already quite involved for $k=2$ but the statements of the main results
in this setting are simple. At the end of the paper we will discuss what happens when $k\ge 3$. 

When $k=2$, we want to understand when
$f\circ \Phi \in U({\mathbb T}^2)$ whenever $f\in A({\mathbb T}^d)$ but now we need to stipulate how we
are summing the Fourier series
$$
f\circ \Phi(x,y) \ = \ \sum_{(k,\ell)\in {\mathbb Z}^2} c_{k,\ell} \, e^{2\pi i [ k x + \ell y]};
$$
here $c_{k,\ell}$ denote the Fourier coefficients of $f\circ\Phi$. That is, the definition of the subspace
$U({\mathbb T}^2)$
of continuous functions $C({\mathbb T}^2)$ on the 2-torus ${\mathbb T}^2$ which have a uniformly convergent
Fourier series depends on how we sum the series (this of course is not needed when we consider the
subspace $A({\mathbb T}^2) \subset C({\mathbb T}^2)$ of absolutely convergent Fourier series). 

We will be mainly interested
in rectangular summation or what is more precisely referred to as {\it unrestricted rectangular summation}. This
is the space of continuous functions $g\in C({\mathbb T}^2)$ which have the property that the rectangular Fourier
partial sums
$$
S_{M,N}g(x,y) = \sum_{|k|\le M, |\ell|\le N} {\widehat g}(k,\ell) \, e^{2\pi i [ k x + \ell y]} =
\int\!\!\!\int_{{\mathbb T}^2} g(x-s,y-t) D_M (s) D_N (t) \, ds dt
$$
converges uniformly to $g$;  $\|S_{M,N} g - g \|_{L^{\infty}({\mathbb T}^2)} \to 0$ as $\min(M,N) \to \infty$.
Here $D_N(y) = \sin(2\pi(N+1/2)y)/\sin(\pi y)$ is the $N$th Dirichlet kernel. We will denote this subspace
as $U_{rect}({\mathbb T}^2)$ which in fact is a Banach space in its own right with respect to the norm
$$
\|g\|_{U_{rect}} \ := \ \sup_{M,N} \, \|S_{M,N} g \|_{L^{\infty}({\mathbb T}^2)}
$$
of uniform rectangular convergence. We will also consider the subspace $U_{sq}({\mathbb T}^2)$
of continuous functions $g$ whose Fourier sereis converges uniformly with respect to {\it square sums}, 
$\|S_{N,N} g - g \|_{L^{\infty}} \to 0$ as $N\to \infty$. This is also a Banach space with the norm
$$
\|g\|_{U_{sq}} \ := \ \sup_{N} \, \|S_{N,N} g \|_{L^{\infty}({\mathbb T}^2)}
$$
of uniform square convergence. Of course we could also consider spherical summation (of various orders)
$U_{sph}({\mathbb T}^2)$ but this problem has a different character and we will investigate this
setting elsewhere.

We will only consider real-analytic mappings $\Phi : {\mathbb T}^2 \to {\mathbb T}^d$ and in
this category, it turns out that it is sometimes the case and sometimes not the case that $f\circ \Phi \in U_{rect}({\mathbb T}^2)$
whenever $f\in A({\mathbb T}^d)$. The same holds for square sums $U_{sq}({\mathbb T}^2)$. 
One of our  main goals will be to give a simple characteristion of those real-analytic mappings 
$\Phi : {\mathbb T}^2 \to {\mathbb T}^d$ which have the property
$$
f\circ \Phi \in U_{rect}({\mathbb T}^2) \ \ {\rm whenever} \ \ f \in A({\mathbb T}^d). \eqno{(\Phi)_{rect}}
$$ 
We will denote by $(\Phi)_{sq}$ the analogous property but with respect to square sums.

It is easy to see that any continuous mapping $\Phi : {\mathbb T}^k \to {\mathbb T}^d$ can be described as
follows: if ${\vec t} = (t_1, \ldots, t_k) \in {\mathbb R}^k$ parameterises a point
$P({\vec t})  = (e^{2\pi i t_1}, \ldots, e^{2\pi i t_k}) \in {\mathbb T}^k$ on the $k$-torus, then 
$$
\Phi(P({\vec t})) \ = \ (e^{2\pi i [\phi_1({\vec t}) + {\vec L_1}\cdot {\vec t}]}, \ldots, 
e^{2\pi i [\phi_d({\vec t}) + {\vec L_d}\cdot {\vec t}]})
$$
for some $d$-tuple ${\overline{\phi}}\,({\vec t}) =  (\phi_1({\vec t}), \ldots, \phi_d({\vec t}))$
of 1-periodic functions on ${\mathbb R}^k$ and some $d$-tuple
${\overline{L}} = ({\vec L_1}, \ldots, {\vec L_d})$ of lattice points in ${\mathbb Z}^k$. Hence
every map 
$$
\Phi \ = \ \Phi_{{\overline{\phi}}, \ {\overline{ L}}} : \ {\mathbb T}^k \ \to \ {\mathbb T}^d
$$ 
has a periodic part ${\overline{\phi}}$ and a linear part ${\overline{L}}$. The mapping $\Phi$ is
said to be {\it affine} precisely when the periodic part ${\overline{\phi}} = (c_1, \ldots, c_d)$
is constant. 

Now let us restrict our attention to $k=2$ so that the periodic part ${\overline \phi} (s,t)$
is a $d$-tuple of analytic, 1-periodic functions of two variables. We put these $d$ periodic functions
together as a single periodic function $\psi(s,t, \omega) := \omega \cdot {\overline \phi} (s,t)$
where $\omega \in {\mathbb S}^{d-1}$ is a point on the $(d-1)$-dimensional unit sphere
${\mathbb S}^{d-1}$. Hence $\psi$ defines a real-analytic map on the compact manifold
${\mathbb T}^2 \times {\mathbb S}^{d-1}$. For any compact, analytic manifold $M$ we say
that an analytic function $\psi$ on ${\mathbb T}^2 \times M$ satisfies the {\it factorisation hypothesis}
(FH)  on $M$ if
at every point $P = (s,t,\omega) \in {\mathbb T}^2 \times M$, the 2nd order partial derivatives
$$
\psi_{ss} = \frac{{\partial}^2 \psi}{\partial s^2 }, \ 
\psi_{tt} = \frac{{\partial}^2 \psi}{\partial t^2}, \
\psi_{st} = \frac{{\partial}^2 \psi}{\partial s \partial t}
$$
of $\psi(s,t,\omega)$ satisfy the factorisation identities
$$
\psi_{st}(s,t, \omega) \ = \ K(s,t,\omega) \, \psi_{ss}(s,t,\omega)  \ \ \ {\rm and} \  \ \
\psi_{st}(s,t,\omega) \ = \ L(s,t, \omega) \, \psi_{tt}(s,t,\omega) 
$$
for analytic functions $K$ and $L$ in some neighbourhood of $P$. This condition can be compactly
expressed in terms of germs of analytic functions. 

\begin{definition} A real-analytic map $\psi : {\mathbb T}^2 \times M \to {\mathbb R}$
satsifies the factorisation hypothesis if at every point in ${\mathbb T}^2 \times M$, the
divisibility conditions
$$
\psi_{ss}, \ \psi_{tt} \ \ | \ \ \psi_{st} \eqno{(FH)}
$$
hold as germs of analytic functions.
\end{definition}

Our main result is the following.

 \begin{theorem}\label{main} Let $\Phi: {\mathbb T}^2 \to {\mathbb T}^d$ be a real-analytic
map. Suppose that ${\overline \phi}$ is the periodic part of $\Phi$ and set
$\psi(s,t,\omega) = \omega \cdot {\overline{\phi}}(s,t)$. Then $(\Phi)_{rect}$
holds if and only if $\psi$ satisfies $(FH)$ on $M = {\mathbb S}^{d-1}$.
\end{theorem}

The hypothesis (FH) on $M= {\mathbb S}^{d-1}$ 
with $\psi(s,t,\omega) = \omega \cdot {\overline{\phi}} (s,t)$ is satsified if
${\overline{\phi}} (s,t) = {\overline{f}}(s) + {\overline{g}}(t)$ 
or if ${\overline{\phi}} (s,t) = {\overline{f}}(k s + \ell t)$ where
${\overline{f}}, {\overline{g}}$
are both $d$-tuples of real-analytic, 1-periodic functions of one variable and $k, \ell$
are integers. Furthermore this factorisation hypothesis remains true under certain
small analytic perturbations; namely, 
if $\psi$ satisfies (FH), then for any real-analytic function
$K$, there is an $\epsilon>0$ such that $\zeta  = \psi + \epsilon K (\psi_{ss})^3 (\psi_{tt})^3$
also satsfies (FH).\footnote{We thank Sandy Davie for pointing this out to us.} On the other hand,
for any two real-analytic, 1-periodic, nonconstant functions $f$, $g$ of one variable, the
function $\psi(s,t) = f(s) g(t)$ will never satisfy (FH). 

Note that in Theorem \ref{main}, the necessary and sufficient condition for $(\Phi)_{rect}$
to hold only depends on the periodic part ${\overline{\phi}}$ and does {\it not} depend on the
linear part ${\overline{L}}$ of the mapping $\Phi : {\mathbb T}^2 \to {\mathbb T}^d$.  In particular, if (FH) fails for
$\psi(s,t,\omega) = \omega \cdot {\overline{\phi}}(s,t)$, then $(\Phi)_{rect}$ fails for all
$\Phi = \Phi_{{\overline{\phi}}, \, {\overline{L}}}$ with periodic part ${\overline{\phi}}$. The same
is not true for the square sum $(\Phi)_{sq}$ problem, at least when $d=1$. 

\begin{theorem}\label{square-sum} Let $\phi(s,t)$ by any real-analytic, 1-periodic function of two variables.
Then there exists infinitely many lattice points $L = (k,\ell) \in {\mathbb Z}^2$ such that $(\Phi)_{sq}$
holds where $\Phi = \Phi_{\phi, L} : {\mathbb T}^2 \to {\mathbb T}$. Furthermore there exist
real-analytic maps $\Phi = \Phi_{\phi, L} : {\mathbb T}^2 \to {\mathbb T}$ such that $(\Phi)_{sq}$
fails for some lattice point $L = (k,\ell)$.
\end{theorem}

Hence we see that there exist real-analytic, 1-periodic functions
$\phi(s,t)$ of two variables such that $(\Phi)_{sq}$ holds for some maps $\Phi = \Phi_{\phi, L}$
{\it and} $(\Phi)_{sq}$ fails for other maps $\Phi = \Phi_{\phi, L'}$, both maps having the same
periodic part. Thus the linear part $L = (k,\ell)$
of the map $\Phi(e^{2\pi i s}, e^{2\pi i t}) = e^{2\pi i [\phi(s,t) + ks +\ell t]}$ 
can determine whether or not the property $(\Phi)_{sq}$ holds. This is not the case for $(\Phi)_{rect}$.

\subsection{Outline of the paper}
In the next section, we give some preliminary reductions and we give the proof of Proposition \ref{alpar-ext}.
In Section \ref{consequences} we discuss some consequences of the factorisation hypothesis (FH). Sections \ref{robust}
and \ref{prelude-I}
develop a more robust one dimensional theory which will be needed in, as well as giving a prelude to, the proof of the main results.
We prove Theorem \ref{main} in Sections 6 and \ref{main-fail} and the proof of Theorem \ref{square-sum} is given
in Section 8.
In the final section we discuss the situation for mappings $\Phi : {\mathbb T}^k \to {\mathbb T}^d$ from
higher dimensional tori. 

\subsection{Notation}

Uniform bounds for oscillatory integrals lie at the heart of this paper. Keeping track of constants
and how they depend on the various parameters will be important for us. For the most part, constants $C$
appearing in inequalities $A \le C B$ between positive quantities $A$ and $B$ will be {\it absolute} or
{\it uniform} in that they can be taken to be independent of the parameters of the underlying problem. 
We will use $A \lesssim B$ to denote $A \le C B$.
When
we allow the constant $C$ to depend on a parameter (or parameters), say $\omega \in {\mathbb S}^{d-1}$,
we will write $A \le C_{\omega}B$ or $A \lesssim_{\omega} B$. For quantities $Q$ and $B$
where $B$ is nonnegative , we write $Q = O(B)$ (or $Q = O_{\omega}(B)$) to denote
$|Q| \le C B$ (or $|Q| \le C_{\omega} B$).
Also $P \sim Q$ or $P \sim_{\omega} Q$ will denote $C^{-1} |P| \le |Q| \le C |P|$ or
$C_{\omega}^{-1} |P| \le |Q| \le C_{\omega} |P|$.

We will use multi-index notation
$$
\frac{\partial^{k+\ell} \phi}{\partial s^k \, \partial t^{\ell}} (s,t) \ =: \ \partial^{k,\ell} \phi (s,t)
$$
for partial derivatives but also when
$k$ or $\ell = 0$, we adopt the shorthand notation $\partial^n_s \phi = \partial^{n,0}\phi$
and $\partial^m_t \phi = \partial^{0,m}\phi$. Finally we will also use $\phi_{ss}, \phi_{tt}$
and $\phi_{st}$ to denote the second order derivatives which play a special role.

\section{Preliminary reductions and the proof of Proposition \ref{alpar-ext}}\label{preliminary}

For the proof of Proposition \ref{alpar-ext}, or to establish $(\Phi)_{rect}$, or $(\Phi)_{sq}$, 
we need to show that whenever $f\in A({\mathbb T}^d)$, we have $f\circ \Phi \in U({\mathbb T})$,
$f\circ \Phi \in U_{rect}({\mathbb T}^2)$ or $f\circ \Phi \in U_{sq}({\mathbb T}^2)$, respectively. 
A simple application of the closed graph theorem shows that this is equivalent to proving that
$f \to f\circ \Phi$ is a continuous map from $A({\mathbb T}^d)$ to $U({\mathbb T})$, 
$U_{rect}({\mathbb T}^2)$ or $U_{sq}({\mathbb T}^2)$; that is, these
three problems are equivalent to establishing the apriori bounds
\begin{equation}\label{basic-est}
\|f\circ \Phi \|_{U({\mathbb T})}, \ \|f\circ \Phi\|_{U_{rect}({\mathbb T}^2)} \ \ {\rm or} \ \
\|f\circ \Phi \|_{U({\mathbb T}^2)} \ \le
\ C \, \|f\|_{A({\mathbb T}^d)}.
\end{equation}
Let $U_{*}$ denote either $U({\mathbb T}), U_{rect}({\mathbb T}^2)$ or $U_{sq}({\mathbb T}^2)$.
If 
$$
f({\vec{x}}) \ = \  \sum_{{\overline{n}}\in {\mathbb Z}^d} {\widehat f}({\overline{n}})  \
e^{2\pi i {\overline{n}}\cdot {\vec{x}}} \ \in \ A({\mathbb T}^d), \ \ {\rm then} \ \ 
f\circ \Phi  \ = \ \sum_{{\overline{n}}\in {\mathbb Z}^d} {\widehat f}({\overline{n}}) \ 
e^{2\pi i {\overline{n}} \cdot {\overline{g}}}
$$ 
where ${\overline{g}}({\vec{t}}) = {\overline{\phi}}({\vec t}) + {\overline{L}}  \cdot {\vec t}$
is the vector-valued function on ${\mathbb R}^k$ parameterising the mapping $\Phi : {\mathbb T}^k \to {\mathbb T}^d$.
By the triangle inequality, we have
$$
\|f \circ \Phi \|_{U_{*}} \ \le \ \sum_{{\overline{n}} \in {\mathbb Z}^d} |{\widehat f}({\overline{n}})| \
\| e^{2\pi i {\overline{n}} \cdot {\overline{g}}} \|_{U_{*}}
$$
and so \eqref{basic-est} holds if and only if the $U_{*}$ norms of the oscillations $e^{2\pi i {\overline{n}}\cdot {\overline{g}}}$
are uniformly bounded in ${\overline{n}}\in {\mathbb Z}^d$; that is,
\begin{equation}\label{main-est}
\sup_{{\overline{n}}\in {\mathbb Z}^d} \ \|e^{2\pi i {\overline{n}}\cdot {\overline{g}}} \|_{U_{*}} \ < \ \infty.
\end{equation}
A similar reduction can be made for the Beurling-Helson theorem but now the $U_{*}$ norms are
replaced with the Weiner norm $\|f\| = \sum_n |{\widehat{f}}(n)|$ on $A({\mathbb T})$.

\subsection{Reduction when $k=1$; for mappings $\Phi : {\mathbb T} \to {\mathbb T}^d$}

In the case $k=1$, we have  
${\overline{g}}(t) = {\overline{\phi}}(t) + {\overline{L}} t$ where ${\overline{\phi}}$ is a $d$-tuple
of real-analytic, 1-periodic real-valued functions of a single variable and ${\overline{L}} \in {\mathbb Z}^d$. Hence 
if $S_N f(x) = \sum_{|n|\le N} {\widehat f}(n) e^{2\pi i n x}$ denotes the classical $N$th Fourier
partial sum of a function $f\in C({\mathbb T})$, we have 
$$
\|e^{2\pi i {\overline{n}}\cdot {\overline{g}}} \|_{U({\mathbb T})} = 
\sup_N \| S_N (e^{2\pi i {\overline{n}} \cdot {\overline{g}}})\|_{L^{\infty}({\mathbb T})} =
\sup_{N,x} \ \Bigl| \int_{{\mathbb T}} e^{2\pi i {\overline{n}} \cdot [{\overline{\phi}}(x - t) + {\overline{L}} (x-t)]} D_N(t) \, dt\Bigr|.
$$ 
Since the Dirichlet kernel 
$$
D_N(t) \ = \
\frac{\sin(2\pi(N+1/2) t)}{\sin(\pi t)} = \frac{e^{2\pi i (N+ 1/2) t} - e^{-2\pi i (N+1/2) t}}{\sin(\pi t)} 
$$
\begin{equation}\label{Dirichlet} 
= \
 \frac{e^{2\pi i (N+ 1/2) t} - e^{-2\pi i (N+1/2) t}}{\pi t}  \ \ + \ \  O(1),
\end{equation}
for $t\in [-1/2, 1/2)$, 
we see that to establish \eqref{main-est} in the case of Proposition \ref{alpar-ext}, it suffices to obtain
(writing ${\overline{n}} = \lambda \omega$ for some $\omega \in {\mathbb S}^{d-1}$) the uniform bound
\begin{equation}\label{osc-est}
\sup_{\lambda, x, \rho \in {\mathbb R}, \omega \in {\mathbb S}^{d-1}} \
\Bigl| p.v. \int_{-\frac{1}{2}}^{\frac{1}{2}} e^{i [\lambda \omega \cdot {\overline{\phi}}(x-t) + \rho t]} \, \frac{dt}{t} \Bigr| 
\ < \ \infty.
\end{equation}

\subsection{Reduction when $k=2$; for mappings $\Phi: {\mathbb T}^2 \to {\mathbb T}^d$}\label{k=2}

In the case $k=2$, we have 
${\overline{g}}(s,t) = {\overline{\phi}}(s,t) + {\vec {L}_1} s + {\vec L_2} t$ where ${\overline{\phi}}$ is a $d$-tuple
of real-analytic, 1-periodic real-valued functions of two variables 
and ${\overline{L}} = ({\vec L_1}, {\vec L_2})$ is a pair of lattice points in ${\mathbb Z}^d$.
 Hence 
if $S_{M,N} f(x,y) = \sum_{|k|\le M, |\ell|\le N} {\widehat f}(k, \ell) e^{2\pi i k x + \ell y}$ denotes an
rectangular Fourier
partial sum of a function $f\in C({\mathbb T}^2)$, we have 
$\|e^{2\pi i {\overline{n}}\cdot {\overline{g}}} \|_{U_{rect}({\mathbb T}^2)} =$
$$ 
\sup_{M,N} \| S_{M,N} (e^{2\pi i {\overline{n}} \cdot {\overline{g}}})\|_{L^{\infty}({\mathbb T})} =
\sup_{M,N,x,y} \ \Bigl| \int\!\!\!\int_{{\mathbb T}^2} e^{2\pi i {\overline{n}} 
\cdot [{\overline{\phi}}(x - s, y - t) + {\overline{L}}\cdot (s,t)]} D_M(s) D_N(t) \, ds dt\Bigr|.
$$ 
From the formula and properties of the Dirichlet kernel displayed in \eqref{Dirichlet} we see that
the integral above underpinning the norms $\|e^{2\pi i {\overline{n}}\cdot {\overline{g}}} \|_{U_{rect}({\mathbb T}^2)}$
splits into a sum of eight integrals, four of these have the form
\begin{equation}\label{form-double}
 \int\!\!\!\int_{{\mathbb T}^2} e^{2\pi i [{\overline{n}} 
\cdot {\overline{\phi}}(x - s, y - t) + A_{\pm} s + B_{\pm}  t]} \, \frac{ds}{s} \frac{dt}{t}
\end{equation}
where $A_{\pm} = {\overline{n}}\cdot {\vec L_1} \pm (M+1/2)$ and 
$B_{\pm} = {\overline{n}} \cdot {\vec L_2} \pm (N+1/2)$. The other four integrals
have the form 
$$
\int\!\!\!\int_{{\mathbb T}^2} e^{2\pi i [{\overline{n}} 
\cdot {\overline{\phi}}(x - s, y - t) + A_{\pm} s + C  t]} \, \frac{ds}{s} dt, \ \ {\rm and} \ \
 \int\!\!\!\int_{{\mathbb T}^2} e^{2\pi i [{\overline{n}} 
\cdot {\overline{\phi}}(x - s, y - t) + D s + B_{\pm}  t]} \, \frac{dt}{t} ds,
$$
each multiplied by an $O(1)$ term. Here $C = {\overline{n}} \cdot {\vec L_2}$ and $D = {\overline{n}}\cdot {\vec L_1}$.
These last four integrals will be uniformly bounded if we can establish a more robust estimate than
\eqref{osc-est}; namely,
\begin{equation}\label{osc-est-II}
S(\omega) \ := \ \sup_{\lambda, \rho \in {\mathbb R}, (x,\tau) \in {\mathbb T}^2} \
\Bigl| p.v. \int_{-\frac{1}{2}}^{\frac{1}{2}} e^{i [\lambda \omega \cdot {\overline{\phi}}(x-s, \tau)
+ \rho s} \, \frac{ds}{s} \Bigr| 
\end{equation}
is a bounded function on ${\mathbb S}^{d-1}$, 
together with an similar bound for the oscillatory integral with the roles of the variables $s$ and $t$
interchanged. Indeed we will show that $\sup_{\omega\in {\mathbb S}^{d=1}} S(\omega) < \infty$ 
holds when the factorisation hypothesis
(FH) holds on ${\mathbb S}^{d-1}$ for $\psi(s,t,\omega) = \omega \cdot {\overline{\phi}}(s,t)$.
Furthermore, even when (FH) does not hold, we still have an estimate for the oscillatory integral
in \eqref{osc-est-II} which is uniform in the parameters $\lambda, x, \tau$ and $\rho$ but not
necessarily in $\omega$; that is, $S(\omega) < \infty$ for all $\omega \in {\mathbb S}^{d-1}$.

\subsection{The work of C. Fefferman} 

In \cite{CF}, Fefferman constructed a continuous function $f \in C({\mathbb T}^2)$ such that
the Fourier series of $f$ with respect to unrestricted rectangular summation diverges everywhere;
more precisely, 
the limit of the rectangular Fourier partial sums $S_{M,N}f(x,y)$ of $f$ does not exist for any
point $(x,y) \in {\mathbb T}^2$. At the heart of the construction is the observation that
if $g_{\lambda}(s,t) = e^{2\pi i\lambda s t}$, then for large $\lambda> 0$ and every $(x,y) \in {\mathbb T}^2$
with $x,y \not= 0$,
$$
\sup_{M,N} |S_{M,N} g_{\lambda}(x,y)| \ = \ \sup_{M,N}
\Bigl| \int\!\!\!\int_{{\mathbb T}^2} e^{2\pi i \lambda (x-s)(y-t)} D_M(s) D_N(t) \, ds dt\Bigr| \ \ge \ 
C \log \lambda
$$
for some $C>0$. To see this, the integral above splits
into four integrals plus some $O(1)$ terms (as in \eqref{form-double} and \eqref{osc-est-II}). The
four integrals have the form
$$
\int\!\!\!\int_{{\mathbb T}^2} e^{2\pi i [\lambda s t +  E_{\pm} s + F_{\pm}  t]} \, \frac{ds}{s} \frac{dt}{t}
$$
where $E_{\pm} = - \lambda y \pm (M+1/2)$ and $F_{\pm} = - \lambda x \pm (N+1/2)$. Depending on
the signs of $\lambda y$ and $\lambda x$, we can choose positive integers $M,N$ such that for one out
of the four choices of $(\pm, \pm)$ we have $E_{\pm}, F_{\pm} = O(1)$ and for the other
three choices, either $|E_{\pm}| \sim |\lambda|$ or $|F_{\pm}| \sim |\lambda|$ (or both). For the integral
with both $E_{\pm}, F_{\pm} = O(1)$, we have
\begin{equation}\label{CF}
 \int\!\!\!\int_{{\mathbb T}^2} e^{2\pi i [\lambda s t +  E_{\pm} s + F_{\pm}  t]} \, \frac{ds}{s} \frac{dt}{t} \ = \
\int_0^{1/2}\!\!\!\int_0^{1/2} \sin(\lambda s t) \frac{ds}{s} \frac{dt}{t} \ + \ O(1).
\end{equation}
Changing variables $\sigma = \lambda s t$ in the $s$ integral 
and interchanging the order of integration shows that the sine integral above is equal to
$c \, \log \lambda + O(1)$ as $\lambda \to \infty$. The other three
integrals are easily seen to be $O(1)$ for large $\lambda$.

Although the function $g_{\lambda}(s,t) = e^{2 \pi i\lambda s t}$ can never be of the form
$e^{i \lambda \omega\cdot {\overline{\phi}}(s,t)}$ since the phase $s \, t$ is not a periodic
function of $(s,t)$, the above observation of Fefferman does give a strong indication that
the estimates \eqref{main-est} may fail for the norms $U_{rect}({\mathbb T}^2)$. 

Let us analyse a little further the relationship between the hyperbolic phase $\lambda s t$
appearing in Fefferman's work and the phase 
$\lambda \omega \cdot {\overline{\phi}}(x-s,y-t) =: \lambda \psi(x-s,y-t, \omega)$
arising from the norms $U_{rect}({\mathbb T}^2)$. Expanding $\psi$ by its Taylor series at $(x,y)$, we have
$$
\psi(x+s,y+t, \omega) \ = \ \sum_{k,\ell=0}^{\infty} \frac{1}{k!\ell!} \partial^{k,\ell} \psi(x,y, \omega) \,  s^k t^{\ell}.
$$

We see that the hyperbolic oscillation $s  t$ arises when 
$\psi_{st}(x,y,\omega) =  \partial^{1,1} \psi (x,y,\omega)$ is nonzero. The only way that this term in
the Taylor expansion vanishes for all points $(x,y,\omega) \in {\mathbb T}^2 \times {\mathbb S}^{d-1}$
is when $\psi(s,t, \omega) = \omega \cdot {\overline{f}}(s) + \omega\cdot {\overline{h}}(t)$ for a pair of
$d$-tuple real-analytic functions ${\overline{f}}$ and ${\overline{h}}$ implying that the norm
$$
\|e^{2\pi i {\overline{n}}\cdot {\overline{g}}} \|_{U_{rect}({\mathbb T}^2)} \ = \ 
\|e^{2\pi i {\overline{n}}\cdot {\overline{g_1}}} \|_{U({\mathbb T})} \cdot 
\|e^{2\pi i {\overline{n}}\cdot {\overline{g_2}}} \|_{U({\mathbb T})}
$$ 
factors into a product of two norms on $U({\mathbb T})$ whose uniform boundedness is the content
of Proposition \ref{alpar-ext}. Here ${\overline{g_1}}(s) = {\overline{f}}(s) + {\vec  L_1} s$
and ${\overline{g_2}}(t) = {\overline{h}}(t) + {\vec L_2} t$. The  example ${\overline{\phi}}(s,t) =
{\overline{f}}(s) + {\overline{h}}(t)$ satisfies the factorisation hypothsis (FH) on ${\mathbb S}^{d-1}$
and so the fact that the associated mapping $\Phi : {\mathbb T}^2 \to {\mathbb T}^d$ transports
$A({\mathbb T}^d)$ to $U_{rect}({\mathbb T}^2)$ is covered by Theorem \ref{main}. Of course
this also follows from two applications of Proposition \ref{alpar-ext}.

Even if the hyperbolic term $\psi_{st} \, s  t $ arises, the other two quadratic terms
$\psi_{ss} s^2$ and $\psi_{tt} t^2$ can conspire against it to ameliorate its effect. For instance,
while the oscillatory integral
$$
p.v. \int\!\!\!\int_{{\mathbb T}^2} e^{2\pi i \lambda s t} \, \frac{ds}{s} \frac{dt}{t} \ = \ c \, \log(\lambda) + O(1)
$$
for large $\lambda > 0$ and some $c \not= 0$, we have the uniform bound
$$
\sup_{\lambda \in {\mathbb R}} \ \Bigl| p.v.
\int\!\!\!\int_{{\mathbb T}^2} e^{2\pi i \lambda (a s + b t)^2} \, \frac{ds}{s} \frac{dt}{t}  \Bigr| \ < \ \infty
$$
 for any choice of $a, b$. Such conspiracy can happen for mappings $\Phi : {\mathbb T}^2 \to {\mathbb T}^d$
where the periodic part ${\overline{\phi}}(s,t) = {\overline{f}}(k s + \ell t)$ and ${\overline{f}}$ is
a $d$-tuple of real-analytic, periodic functions of one variable and $k, \ell \in {\mathbb Z}$. Of course
this family of examples satisfies the factorisation hypothesis (FH) and so Theorem \ref{main} implies
that $\Phi$ transports $A({\mathbb T}^d)$ to $U_{rect}({\mathbb T}^2)$ via composition. It is no longer
the case that the positive result for this family of mappings follows from Proposition \ref{alpar-ext}. In fact, 
its study was the inspiration to find the argument to establish the part of Theorem \ref{main}
when the factorisation hypothesis (FH) holds.

\subsection{Proof of Proposition \ref{alpar-ext}}

Let us set $\psi(t, \omega) := \omega \cdot {\overline{\phi}}(t)$ which is a real-valued, real-analytic function
on the compact manifold ${\mathbb T} \times {\mathbb S}^{d-1}$. First suppose there exists a point
$(t_0, \omega_0) \in {\mathbb T} \times {\mathbb S}^{d-1}$ such that $\partial^k_t \psi (t_0, \omega_0) = 0$
for all $k\ge 2$. Hence $\psi(t, \omega_0) = \psi(t_0, \omega_0) + \partial_t \psi (t_0, \omega_0) (t - t_0)$
for $t$ near $t_0$ and hence for all $t \in {\mathbb T}$ since both sides of the equality are analytic functions
on the connected manifold ${\mathbb T}$. But $\psi(t,\omega_0)$ is periodic in $t$ and this forces
$\partial_t \psi(t_0, \omega_0) = 0$ since otherwise the function 
$g(t) := \psi(t_0,\omega_0) +  \partial_t \psi(t_0, \omega_0) (t - t_0)$ is not periodic in $t$. This implies
that $\omega_0 \cdot {\overline{\phi}}(t) \equiv constant$ and so one of the functions $\phi_j(t)$
in the $d$-tuple of analytic functions ${\overline{\phi}}(t) = (\phi_1(t), \ldots, \phi_d(t))$ can be written
as a linear combinatation of the other $d-1$ functions which (after a simple re-writing of the oscillatory integral
appearing in \eqref{osc-est}) reduces matters to establishing \eqref{osc-est} with $d$ replaced by $d-1$.

Therefore by a simple induction on dimension argument, we  may assume without loss of generality that for every point 
$(t,\omega) \in {\mathbb T} \times {\mathbb S}^{d-1}$, there is an integer $k = k(t,\omega) \ge 2$ such
that $\partial^k_t \psi(t, \omega) \not= 0$. Furthermore there exists an $\epsilon = \epsilon(t,\omega) >0$
such that 
\begin{equation}\label{derivative-bound}
\bigl| \partial^k_t \psi (t', \omega') \bigr | \ \ge \ \frac{1}{2} \bigl| \partial^k_t \psi(t, \omega) \bigr|
\end{equation}
for every $(t', \omega') \in B_{2\epsilon}(t,\omega)$ in a ball of radius $2\epsilon$ centred at $(t,\omega)$.
We cover ${\mathbb T} \times {\mathbb S}^{d-1}$ with balls $B_{\epsilon}$ of radius $
\epsilon(t,\omega) >0$ and use the compactness of ${\mathbb T}\times {\mathbb S}^{d-1}$
to extract a finite subcover 
$\{ B_j \}_{j=0}^M$ of  ${\mathbb T} \times{\mathbb S}^{d-1}$ and reduce matters to establishing
a uniform bound for the oscillatory integral in \eqref{osc-est} for parameters $\lambda, \rho \in {\mathbb R}$
and $(x,\omega) \in B_j$ for some $0\le j \le M$. For such parameters, we write the integral in \eqref{osc-est}
as
$$
p.v. \int_{-\frac{1}{2}}^{\frac{1}{2}} e^{i [\lambda \omega \cdot {\overline{\phi}}(x-t) + \rho t]} \, \frac{dt}{t} \ = \
p.v. \int_{|t|< \epsilon_j} e^{i [\lambda \omega \cdot {\overline{\phi}}(x-t) + \rho t]} \, \frac{dt}{t} \ + \ 
O(\log(1/\epsilon_j).
$$
For convenience we take $j=0$.
The ball $B_0$ has radius $\epsilon_0$ with centre $(x_0, \omega_0)$, say, so that the paramaters 
$(x,\omega)$ lie in the ball $B_{\epsilon_0}(x_0, \omega_0)$. Furthermore by \eqref{derivative-bound},
the phase of the oscillatory integral
$\varphi(t) := \lambda \, \omega \cdot {\overline{\phi}}(x-t) + \rho t$ satisfies
$$
\bigl|\varphi^{(k_0)}(t)\bigr| \ \ge \  |\lambda| \, \bigl| \partial^{k_0}_t \psi(x_0, \omega_0) \bigr| / 2 \ =: \ 
c_0 |\lambda|
 $$
for all $|t| < \epsilon_0$. Here $k_0 = k_0(x_0, \omega_0) \ge 2$ is the integer which makes $c_0$
above nonzero. By van der Corput's lemma (see page 332 in \cite{S}), together with an integration by parts
argument, we see that
$$
\Bigl| \int_{|\lambda|^{-1/k_0} \le |t| \le \epsilon_0} e^{i [\lambda \omega \cdot {\overline{\phi}}(x-t)  + \rho t]}
\frac{dt}{t} \Bigr| \  \le \ C(k_0, c_0).
$$
We remark that ensuring $k_0\ge 2$ has two advantages. First the linear term $\rho \, t$ in the
phase $\varphi(t)$ is killed off when we compute the $k_0$ derivative (thus making it possible to obtain
a bound from below on $\partial^{k_0}_t  \varphi (t)$. And second, when employing van der Corput's lemma we
do not need to ensure the additional condition
that $\partial_t \varphi (t)$ is monotone (a task which could be difficult to achieve) which is needed when $k_0 = 1$.

In the complimentary range $|t| \le |\lambda|^{-1/k_0}$, we compare the phase $\varphi(t)$
with its $k_0$ Taylor polynomial 
$$
P^{k_0}_{x,\omega}(t) = \sum_{\ell=0}^{k_0 -1} \frac{1}{\ell!} \varphi^{(\ell)}(0) \,  t^{\ell} \ 
= \ \varphi(0) + \varphi'(0) t + \lambda
\sum_{\ell=2}^{k_0-1} \frac{(-1)^{\ell}}{\ell!} (\omega \cdot {\overline{\phi}})^{(\ell)}(x) \, t^{\ell}
$$
so that
$$
|\varphi(t) - P^{k_0}_{x,\omega}(t)| \ \le \  \|{\overline{\phi}}\|_{C^{k_0}({\mathbb T})} \, |\lambda| |t|^{k_0}  
$$
and therefore
$$
 \Bigl| \int_{ |t| \le |\lambda|^{-1/k_0}} e^{i [\lambda \omega \cdot {\overline{\phi}}(x-t)  + \rho t]}
\frac{dt}{t} - \int_{|t| \le |\lambda|^{-1/k_0}} e^{i P^{k_0}_{x,\omega}(t)}
\frac{dt}{t} \Bigr| \ \le \ C |\lambda| \int_0^{|\lambda|^{-1/k_0}} t^{k_0 - 1} dt \ \le \ C 
$$
where $C = C(k_0, {\overline{\phi}})$. Hence we are left with estimating the oscillatory integral
\begin{equation}\label{key-reduction-1d}
\int_{|t| \le |\lambda|^{-1/k_0}} e^{i P^{k_0}_{x,\omega}(t)} \ \frac{dt}{t}
\end{equation}
where $P^{k_0}_{x,\omega}$ is a polynomial of degree at most $k_0$. Here we can appeal to a result of
Stein and Wainger \cite{SW} which states that for any $d\ge 0$, there is a universal constant
$C_d$ such that for any real polynomial $P(t) = a_d t^d + \cdots + a_1 t + a_0$ of degree at most $d$,  and for
any $0\le a < b$,
\begin{equation}\label{SW}
\Bigl| \int_{a\le |t| \le b} e^{i [a_d t^d + \cdots + a_1 t + a_0]} \ \frac{dt}{t} \Bigr| \ \le \ C_d.
\end{equation}
In particular the bound is uniform in the coefficients $\{a_0, \ldots, a_d\}$ and the
interval $a\le|t|\le b$ of integration. This completes the proof \eqref{osc-est} and
therefore Proposition \ref{alpar-ext}.

\subsection{Developing a perspective}\label{perspective}  Before we move on to truly higher dimensional variants of
Kahane's problem, we pause here to make a few comments on the argument presented above
and indicate some of the issues which lie ahead when we turn to mappings from ${\mathbb T}^2$ to ${\mathbb T}^d$.

First of all, matters are quickly reduced to proving a uniform bound for a certain integral, namely \eqref{osc-est}.
This integral contains an oscillatory factor $e^{i \varphi(t)}$ with a real-valued phase
$\varphi(t) = \lambda \, \omega \cdot {\overline{\phi}}(x-t) + \rho t$ (depending on a number 
of parameters $\lambda, \omega, x$ and $\rho$) which is integrated against a Calder\'on-Zygmund
kernel, the classical Hilbert transform kernel $1/t$. The task is to find enough cancellation in the integral to
overcome the logarithmic divergence of the integral $\int 1/|t|$ and do so obtaining a bound which is
uniform in all the parameters. Cancellation arises in two ways: $(1)$ from rapid oscillation of $e^{i\varphi(t)}$
which is detected by some derivative of the phase $\varphi$ being large and (2) from
the singular kernel $1/t$ which is detected by the fact that the Fourier transform of the principal-valued
distribution
$p.v. 1/t$,
\begin{equation}\label{FT}
\bigl( {\widehat{p.v.1/t}} \bigr) (\xi) \ = \ p.v. \int e^{-2\pi i \xi t} \, \frac{dt}{t}
\end{equation}
defines a {\it bounded} function in $\xi \in {\mathbb R}$. Equivalently, convolution with $p.v 1/t$,
the classical Hilbert transform, is bounded on $L^2({\mathbb R})$. 

These two types of cancellation often compliment each other perfectly giving us a successful
bound on an oscillatory integral. For instance, for the oscillatory integral in \eqref{osc-est},
we now give a slightly altered argument, using only the types of cancellation described above
(and in particular, avoiding the robust result of Stein and Wainger \eqref{SW}). We begin as
before and first take care of the
degenerate case to reduce ourselves to the situation where there is a nonzero derivative
$\partial^k_t$ of $\psi$ at every point $(x,\omega) \in {\mathbb T} \times {\mathbb S}^{d-1}$.

Now suppose
we chose the integer $k = k(x,\omega) \ge 2 $ to be {\it minimal} with respect to the property
that $\partial^{k}_t \psi (x,\omega) \not= 0$. Then the $k$th
Taylor polynomial 
$$
P^k_{x,\omega} (t) \ = \ \sum_{\ell=0}^{k-1} \frac{1}{\ell!} \varphi^{(\ell)}(0) \, t^{\ell} \ = \
\varphi(0) \ +  \  \varphi'(0) \, t
$$
of $\varphi(t) = \lambda \omega \cdot {\overline{\phi}}(x-t) + \rho \, t$
is a linear polynomial. For this fixed point $(x,\omega)$ we could choose $\delta = \delta(x,\omega) > 0$
such that $|\partial^k_t \varphi (t) | \ge (1/2) |\lambda| |\partial^k_t \psi (x,\omega)|$ for all $|t| < \delta$.
Hence as above, using the fact that the $k$th derivative of the phase $\varphi$ is large, we have the bound
$$
\Bigl| \int_{|\lambda|^{-1/k} \le |t| \le 1/2} e^{i [\lambda \omega \cdot {\overline{\phi}}(x-t)  + \rho t]}
\frac{dt}{t} \Bigr| \  \le \ C(k, c) \ + \ O(\log(1/\delta))
$$
where $c = |\partial^k_t \varphi(x,\omega)| > 0$. For the complimentary integral, we have
$$
 \int_{ |t| \le |\lambda|^{-1/k}} e^{i [\lambda \omega \cdot {\overline{\phi}}(x-t)  + \rho t]}
\frac{dt}{t} \ =
$$
$$
 \int_{ |t| \le |\lambda|^{-1/k}} \Bigl[ e^{i [\lambda \omega \cdot {\overline{\phi}}(x-t)  + \rho t]}
-  e^{i P^{k}_{x,\omega}(t)} \Bigr]
\frac{dt}{t} \ + \ A \, \int_{|t| \le |\lambda|^{-1/k}} e^{i \xi t} \frac{dt}{t} \ =: \ I + II
$$
where $A = e^{i \varphi(0)} = e^{i \lambda \omega \cdot {\overline{\phi}}(x)}$ and
$\xi = \varphi'(0)$.
The first term satisfies $|I| \le 2 \|{\overline{\phi}}\|_{C^k({\mathbb T})}$ as before and $II$ is uniformly bounded
by a slightly more robust version of \eqref{FT}. So we have managed to estimate
\eqref{osc-est} without appealing to the result \eqref{SW}  of Stein and Wainger. 

Although we obtain in this way a bound which is uniform in the parameters $\lambda, \rho \in {\mathbb R}$, the
bound depends on the point $(x,\omega) \in {\mathbb T} \times {\mathbb S}^{d-1}$. In fact the bound
depends on $k, \delta^{-1}$ and $c^{-1}$ which all depend on $(x,\omega)$. In particular, if we slightly perturb
$(x,\omega)$ to $(x',\omega')$, the associated quantities $\delta$ and $c$ could vanish; we have made
no provisions to prevent this from happening. This is why in the original argument, we employed a compactness
argument to stablise the integer $k$ and most importantly, we stablised the derivative bound of the phase
from below. This is the content of \eqref{derivative-bound}. {\it However} we are unable to stablise
the minimality property of $k$ used in the above, alternative argument. The minimality property of $k$
is highly unstable.
If we perturb slightly the point
$(x,\omega)$ to $(x',\omega')$, we can still achieve a bound from  below on 
$\partial^{k}_t (\omega' \cdot {\overline{\phi}})(x'-t)$
with $k = k(x,\omega)$ BUT it is no longer the case that all lower order derivatives 
$\partial^{\ell} (\omega' \cdot {\overline{\phi}}) (x')$ vanish for $2\le \ell \le  k-1$. Generically,
these derivatives will become nonzero.

For these reasons we are led to employ the result of Stein and Wainger in \eqref{SW}. One way
to view \eqref{SW} is to observe that the $L^2$ boundedness of convolution with the distribution $p.v. \, 1/t$
(this is the content of \eqref{FT}) remains true if we multiply the Hilbert transform kernel $1/t$ by
any polynomial oscillation; that is, convolution with $p.v.  e^{i P(t)}/t$ is bounded on $L^2$
and moreover, the $L^2$ operator norm is independent of the coefficients of $P$, depending only
on the degree of $P$. This is precisely the content of \eqref{SW}.\footnote{The bound \eqref{SW} in fact
says a little more -- this $L^2$ boundedness remains true after truncation of the kernel $e^{iP(t)}/t$
to an interval $a \le |t| \le b$ with bounds uniform in the truncation.} This observation has been extended
to higher dimensions 
%(in the one parameter theory) 
by Stein in \cite{S-B} and by Ricci and Stein \cite{RS}
to the nontranslation-invariant setting.
%(but still within the one parameter context). 
This extension of
Calder\'on-Zygmund theory, the stability of $L^p$ mapping properties of classical Calder\'on-Zygmund
singular integral operators by multiplication of the CZ kernel by a general polynomial oscillation is intimately connected
to the theory of Singular Radon Transforms, see \cite{RS}. Hence there is an interesting connection
between Kahane's problem and the theory of Singular Radon Transforms -- these matters are discussed
in a nice survey paper by Wainger, \cite{W}. The connections remain when we move to higher dimensions
but in our case, we will need to move from one parameter to multiparameter Calder\'on-Zygmund theory.

When we move to higher dimensions ($k=2$), the $U_{rect}({\mathbb T}^2)$ and $U_{sq}({\mathbb T}^2)$
norms involve a product of Dirichlet kernels $D_M(s) D_N(t)$ which will lead to two dimensional variants
of integrals in \eqref{osc-est} but now we will be integrating an oscillation against a product-type 
Calder\'on-Zygmund kernel which is in fact the quintessential  product-type CZ kernel, the double
Hilbert transform kernel $1/s t$. We will be confronted with oscillatory integrals of the form
$$
p.v. \int\!\!\!\int_{{\mathbb T}^2} e^{i\psi(s,t)} \, \frac{ds}{s} \frac{dt}{t}
$$
where $\psi$ is some real-valued, real-analytic phase depending a several parameters and we will need
to obtain bounds for this integral which are uniform with respect to these parameters. Again one of
these parameters will be a point $\omega \in {\mathbb S}^{d-1}$ on the unit sphere and so
$\psi(s,t) = \psi_{\omega}(s,t)$.  Here the nondegenerate
case is when for all points $(s,t,\omega) \in {\mathbb T}^2 \times {\mathbb S}^{d-1}$, there
are integers $m = m(s,t,\omega)\ge 2$ and $n = n(s,t,\omega) \ge 2$ such that both derivatives
$\partial^{(m)}_t \psi_{\omega}(s,t)$ and $\partial^{(n)}_s \psi_{\omega}(s,t)$ are nonzero. This will be
much more difficult to achieve (and in fact we will not be able to get ourselves exactly in this situation). 

Once in the nondegenerate case (or some variant), we can use a compactness argument to stablise
the integers $m$ and $n$ and corresponding derivative bounds from below. The minimality of 
$k = k(x,\omega)$ from the discussion above is replaced by the so-called {\it Newton diagram}
of the phase function $\psi$, based at some point $(x,y,\omega) \in {\mathbb T}^2 \times {\mathbb S}^{d-1}$. 
Newton diagrams are highly unstable and can dramatically change under small perturbations of the
phase function. Proceeding in some analgous way to the proof of Proposition \ref{alpar-ext} given above,
we will be led to enquire whether multiparameter Calder\'on-Zygmund theory (at least the part of the
theory concerned with $L^2$ mapping properties of convolution with product-type Calder\'on-Zygmund
kernels) is stable under multiplication of the kernel by general polynomial oscillations. Fortunately this
has already been considered in the context of singular Radon transforms (see \cite{CWW} and \cite{Street})
but unfortunately the answer is NO in general. The theory is not stable under multiplication by
general polynomial oscillations and part of our endeavour here is to understand to what extent (or
for which polynomial oscillations) it is stable. This explains in part why there will be many analytic
mappings $\Phi : {\mathbb T}^2 \to {\mathbb T}^d$ which do NOT carry $A({\mathbb T}^d)$ to
$U_{rect}({\mathbb T}^2)$.

\section{Consequences of  the factorisation hypothesis (FH)}\label{consequences}

Let $M$ be an real-analytic manifold. Suppose
that $\psi$ is an analytic function on ${\mathbb T}^2 \times M$ which satisfies the
factorisation hypothesis (FH)  on $M$; that is  every point of ${\mathbb T}^2 \times M$ has
a neighbourhood $U$ on which
$$
\partial^{1,1} \psi(s,t,\omega) = L(s,t,\omega) \partial^2_t \psi(s,t,\omega) \ \ {\rm and} \ \
\partial^{1,1} \psi(s,t,\omega) = K(s,t,\omega) \partial^2_s \psi(s,t,\omega) 
$$
for some real-analytic functions $L$ and $K$ on $U$. From these two relations on $U$, it
is a simple matter to see that every partial derivative $\partial^{k,\ell}\psi (s,t)$ with $k, \ell\ge 1$,
can be written as some linear combination of the pure derivatives $\partial^{j}_t \psi$ and $\partial^j_s \psi$
with coefficients depending on $L$ and $K$.

\begin{lemma}\label{relations} For every $k,\ell\ge 1$, we have
$$
\partial^{k,\ell}\psi \  \ = \ \ \sum_{j=2}^{k+\ell} Q^{k,\ell}_j \ \partial^j_t \psi
$$
on $U$. 
Furthermore we have the following relationships.
\begin{align}\label{ell=1}
 {\underline{j=2}} \ \ \ \ \ \ &Q^{k,1}_2 \ = \  \partial_s Q^{k-1,1}_2 + \sum_{j=2}^k Q^{k-1,1}_j \partial^{j-1}_t L \nonumber\\
 {\underline{3\le j \le k}} \ \ \ \ \ \ &Q^{k,1}_j \ = \ \partial_s Q^{k-1,1}_j + \sum_{r=j-1}^k {{r-1}\choose{j-2}} Q^{k-1,1}_r
\partial^{r+1-j}_t L \\
{\underline{j=k+1}} \ \ \ \ \ \ &Q^{k,1}_{k+1} \ = \ L^k . \nonumber
\end{align}
When $k=1$, the three formulae above collapse to a single formula; namely  $Q^{1,1}_2 = L$. 
When $\ell\ge 2$, we have
\begin{align}\label{ell = 2}
{\underline{j=2}} \ \ \ \ \ \ \ \ \ \ \ \ &Q^{k,\ell}_2 \ \ = \ \  \partial_t Q^{k,\ell-1}_2  \nonumber\\
{\underline{3\le j \le k+\ell - 1}} \ \ \ \ \ \ \ \ \ \ \ \  &Q^{k,\ell}_j \ \ = \ \ \partial_t Q^{k,\ell-1}_j +  Q^{k,\ell-1}_{j-1}\\
{\underline{j=k+\ell}} \ \ \ \ \ \ \ \ \ \ \ \  &Q^{k,\ell}_{k+\ell} \ \ = \ \ L^k . \nonumber
\end{align}
We also have $\partial^{k,\ell} \psi = \sum_{j=2}^{k+\ell} R^{k,\ell}_j \partial^j_s \psi$ where the coefficients
$R^{k,\ell}_j$ satisfy similar relations.
\end{lemma}

\begin{proof} First we observe that by differentiating $j-1$ times the factorisation hypothesis
$\partial^{1,1} \psi = L \partial^2_t \psi$ with respect to $t$, we have the useful formula
\begin{equation}\label{useful-identity}
\partial^{1,j} \psi \ = \ \sum_{r=2}^{j+1} {{j-1}\choose{r-2}} (\partial^{j+1-r}_t L) \, \partial^r_t \psi.
\end{equation}
This can be verified by a simple induction argument on $j$.

Next we will establish \eqref{ell=1} and we will do so by induction on $k$.
The case $k=1$ follows immediately from the factorisation hypothsis $\partial^{1,1} \psi = L \partial_t \psi$
and so $Q^{1,1}_2 = L$. Suppose now that \eqref{ell=1} holds for some $k-1 \ge 1$. In particular 
$\partial^{k-1,1}\psi = \sum_{j=2}^k Q^{k-1,1}_j \partial^j_t \psi$ and by differentiating this relation with
respect to $s$, we see that
$$
\partial^{k,1} \psi \ = \ \sum_{j=2}^k \Bigl[ \partial_s Q^{k-1,1}_j \, \partial^j_t \psi + Q^{k-1,1}_j \partial^{1,j} \psi\Bigr]
$$
Applying the useful formula \eqref{useful-identity} to the second term, we have
$$
\sum_{j=2}^k  Q^{k-1,1}_j \partial^{1,j} \psi \ = \ \sum_{j=2}^k Q^{k-1,1}_j \sum_{r=2}^{j+1} {{j-1}\choose{r-2}}
\partial^{j+1-r}_t L \, \partial^r_t \psi
$$
$$
= \ \partial^2_t \psi \sum_{j=2}^k Q^{k-1,1}_j \partial^{j-1}_t L \ + \ \sum_{r=3}^{k+1} \Bigl[
\sum_{j=r-1}^k {{j-1}\choose{r-2}} Q^{k-1,1}_j \partial^{j+1-r}_t L \Bigr] \, \partial^r_t \psi
$$
and therefore $\partial^{k,1} \psi = \sum_{j=2}^{k+1} Q^{k,1}_j \partial^j_t \psi$ where
the $\{Q^{k,1}_j\}_{j=1}^{k+1}$ satisfy the three relations given in \eqref{ell=1}. 
Note that in particular, the above gives $Q^{k,1}_{k+1} = Q^{k-1,1}_k L$ and  iterating this $k$ times, 
we obtain $Q^{k,1}_{k+1} = L^k$.

Let us now
turn to \eqref{ell = 2} which we will establish by induction on $\ell$. Suppose that
$$
\partial^{k,\ell-1} \psi \ = \  \sum_{j=2}^{k+\ell-1} Q^{k,\ell-1}_j \ \partial^j_t \psi
$$
holds for some $\ell - 1\ge 1$ and all $k\ge 1$. Differentiating this identity with respect to $t$ gives us
$$
\partial^{k,\ell} \psi \ = \ \sum_{j=2}^{k+\ell -1} \Bigl[ \partial_t Q^{k,\ell-1}_j \, \partial^j_t  \psi \ + \ 
Q^{k,\ell -1}_j \partial^{j+1}_t \psi \Bigr] \ = \ \partial_t Q^{k,\ell-1}_2 \partial^2_t \psi 
$$
$$
\ + \ \sum_{j=3}^{k+\ell - 1}
\Bigl[ \partial_t Q^{k,\ell-1}_j \, \partial^j_t  \psi \ + \ 
Q^{k,\ell -1}_{j-1} \partial^{j}_t \psi \Bigr] \ + \  Q^{k,\ell-1}_{k+\ell-1} \partial^{k+\ell}_t \psi.  
$$
Therefore $\partial^{k,\ell} \psi = \sum_{j=1}^{k+ \ell} Q^{k,\ell}_j \partial^j_t \psi$
where $\{Q^{k,\ell}_j \}_{j=2}^{k+\ell}$ satisfy the relations \eqref{ell = 2}. 

To be clear, for the final relation in \eqref{ell = 2}, we have
$Q^{k,\ell}_{k+\ell} = Q^{k,\ell-1}_{k+\ell-1}$ and upon iterating we obtain
$Q^{k,\ell}_{k+\ell} = Q^{k,1}_{k+1} = L^k$ by \eqref{ell=1}.
\end{proof}

As an immediate consequence of Lemma \ref{relations}, we have the following observation
which will be useful to us in the next section when we turn to developing a more robust one dimensional theory.

\begin{corollary}\label{degenerate-FH} Let $\psi$ be a real-analytic function on ${\mathbb T}^2 \times M$
satisfying the factorisation hypothesis (FH) on $M$. Suppose that there exists a point
$(x_0,y_0,\omega_0) \in {\mathbb T}^2 \times M$ such that 
$$
\partial^m_t \psi(x_0,y_0,\omega_0) \ = \ 0 \ \ \ {\rm for \ all} \ m\ge 2.
$$
Then $\psi(x,y,\omega_0) \equiv \psi(x,y_0, \omega_0)$.
\end{corollary}
 
\begin{proof} Lemma \ref{relations} implies $\partial^{k,\ell} \psi(x_0, y_0, \omega_0) = 0$ for all $k,\ell \ge 1$.
Hence
$$
\psi(x,y,\omega_0) \ = \ \sum_{k,\ell\ge 0} \frac{1}{k!\ell!} \partial^{k,\ell} \psi(x_0, y_0, \omega_0) (x - x_0)^k (y-y_0)^{\ell}
$$
\begin{equation}\label{y}
= \ \psi(x,y_0,\omega_0) \ + \ \partial_t \psi(x_0,y_0,\omega_0) (y - y_0)
\end{equation}
for $(x,y)$ near $(x_0,y_0)$ and hence for all $(x,y)$ by analyticity. However we also have
$$
\psi(x_0,y,\omega_0) \ = \ \sum_{\ell=0}^{\infty} \frac{1}{\ell!} \partial^{\ell} \psi (x_0,y_0,\omega_0) (y - y_0)^{\ell}
\ = \ \psi(x_0, y_0,\omega_0) + \partial_t \psi(x_0,y_0,\omega_0) (y - y_0)
$$
for $y$ near $y_0$ and hence for all $y$ by analyticity. This forces $\partial_t \psi(x_0,y_0, \omega_0) = 0$
since $\psi(x_0,y, \omega_0)$ is a periodic function in $y$. Plugging this back into \eqref{y} shows
$\psi(x,y,\omega_0) = \psi(x, y_0, \omega_0)$, completing the proof of the corollary.
\end{proof}

Another immediate consequence of Lemma \ref{relations} which is needed in the
next key proposition is the following.

\begin{corollary}\label{max}  Let $\psi$ be a real-analytic function on ${\mathbb T}^2 \times M$ satisfying
the factorisation hypothsis (FH) on $M$. Suppose that 
there is a point $(x,y,\omega) \in {\mathbb T}^2 \times M$ and a pair of integers $m,n\ge 2$ such that
$$
\partial^k_s \psi(x, y, \omega)  \ = \ 0 \ \ {\rm for} \ 2\le k < n \ \ \ {\rm and} \ \ 
\partial^{\ell}_t \psi(x, y, \omega)  \ = \ 0 \ \ {\rm for} \ 2\le \ell < m.
$$
If $m$ or $n$ is equal to 2, the corresponding condition above is vacuous.
Then $\partial^{k,\ell} \psi (x, y, \omega) = 0$ for all $k,\ell\ge 1$ satisfying $k+\ell < \max(m,n)$.
\end{corollary}

Before we give the proof, let us see how to interpret this corollary if $m=n=2$. In this case,
$\max(m, n) = 2$ and there are no pairs of integers $k, \ell \ge 1$ satisfying the condition $k+\ell < 2$.
Hence the conclusion is vacuous and there is nothing to prove in this case. 

\begin{proof} If $\max(m,n) = m$ and $k,\ell\ge 1$ satisfies $k+ \ell < m$, then
$$
\partial^{k,\ell} \psi (x, y, \omega) \ = \ \sum_{j=2}^{k+\ell} Q^{k,\ell}_j \partial^j_t \psi(x,y,\omega) \ = \ 0.
$$
Also if $\max(m,n) = n$ and $k,\ell\ge 1$ satisfies $k+ \ell < n$, then
$$
\partial^{k,\ell} \psi (x, y, \omega) \ = \ \sum_{j=2}^{k+\ell} R^{k,\ell}_j \partial^j_s \psi(x,y,\omega) \ = \ 0.
$$
This completes the proof of the corollary.
\end{proof}

At the heart of the proof of Theorem \ref{main} when we are assuming that the factorisation
hypothesis holds is the following proposition.

\begin{proposition}\label{key-prop} Let $\psi$ be a real-analytic function on ${\mathbb T}^2 \times M$ satisfying
the factorisation hypothsis (FH) on $M$. Suppose at a point $(x,y,\omega) \in {\mathbb T}^2 \times M$,
there exist integers $m,n\ge 2$ such that $\partial^m_t \psi(x,y,\omega) \not= 0, \ \partial^n_s \psi(x,y,\omega) \not= 0$
but
$$
\partial^k_s \psi(x, y, \omega)  \ = \ 0 \ \ {\rm for} \ 2\le k < n \ \ \ {\rm and} \ \ 
\partial^{\ell}_t \psi(x, y, \omega)  \ = \ 0 \ \ {\rm for} \ 2\le \ell < m.
$$
Fix $r\ge 0$ and suppose $m+r < n$.
%Furthermore suppose that $\partial^{k,\ell} L(x,y,\omega) = 0$
%for all $k,\ell$ satisfying $k+\ell \le r-1$ (this is a vacuous supposition when $r=0$). 
Then
$\partial^{k,\ell} L (x,y,\omega) = 0$ for all nonnegative $k,\ell$ satisfying $k+\ell \le r$.
\end{proposition}

Let us try to get some feeling for what this proposition says in the case $r=0$. In this case the condition $m + r < n$ 
simply says $m<n$. 
%Furthermore when $r=0$, we have no conditions on any derivative of $L$ at $(x,y,\omega)$. 
The only pair
of nonnegative integers $k,\ell$ satisfying $k+\ell \le 0$ is $k = \ell = 0$ and so we would like to conclude that
$L(x,y,\omega) = 0$ under our conditions on $\partial^k_s \psi$ and $\partial^{\ell}_t \psi$ at the
point $(x,y,\omega)$.

Corollary \ref{max} implies that $\partial^{k,\ell} \psi(x, y, \omega) = 0$ for
all integers $k,\ell \ge 1$ satisfying $k+\ell < \max(m,n) = n$. In particular if we choose
$k=1$ and $\ell = m-1 \ge 1$ (since $m\ge 2$), then $\partial^{1,m-1} \psi (x,y,\omega) = 0$
since $k+ \ell = 1 + m - 1 = m  <  n = \max(m,n)$.  
But on the other hand, by Lemma \ref{relations}, we have
$$
\partial^{1,m-1} \psi (x,y,\omega) \ = \ \sum_{j=2}^{m} Q^{1,m-1}_j \partial^j_t \psi (x,y,\omega)  \ = \
Q^{1,m-1}_m \, \partial^m_t \psi(x,y,\omega)
$$
since we are asumming $\partial^{j}_t \psi(x,y,\omega) = 0$ for all $2\le j < m$. Since
$Q^{1,m-1}_m = L$ by \eqref{ell=1} and \eqref{ell = 2} , we have
$$
0 \ = \ \partial^{1,m-1} \psi (x,y,\omega) \ = \ L(x,y,\omega) \, \partial^m_t \psi (x,y, \omega)
$$
and this
forces $L(x,y,\omega) = 0$ since we are assuming $\partial^m_t \psi(x,y,\omega) \not= 0$.

Unfortunately matters become more complicated for larger values of $r$ and it will be necessary to
have some knowledge of the coefficients $Q^{k,\ell}_j$ which are clearly certain polynomials
in $L$ and the various derivatives $D^{\beta} L = \partial^{\beta_1, \beta_2} L$ of $L$.
In fact we will need to have a fairly good understanding about the structure of these polynomials.

We organise the polynomial $Q^{k,\ell}_j$ in terms of  powers of $L$; it  is clear that the largest
power of $L$ which arises is $L^k$. We write
$$
Q^{k,\ell}_j \ = \ \sum_{r=0}^k \, P^{k,\ell,j}_r \, L^{k-r}
$$
where $P^{k,\ell,j}_r$ is a polynomial whose variables are the derivatives of $L$;
$$
Y_{\beta} \ = \ D^{\beta} L \ = \ \partial^{\beta_1,\beta_2} L, \ \  \beta = (\beta_1, \beta_2) \ \ {\rm with} \ 
|\beta| := \beta_1 + \beta_2 \ge 1.
$$
For convenience we enumerate the variables $\{Y_{\beta}\} = \{X_j\}_{j\ge1}$ sequentially; namely,
$X_j = Y_{\beta_j}$ where 
$$
\beta_1 = (1,0), \ \beta_2 = (0,1), \ \beta_3 = (2,0), \ \beta_4 = (1,1), \ \beta_5 = (0,2), \ \beta_6 = (3,0), \ {\rm etc}...
$$
Precisely, for each $n\ge 2$ and  $n(n-1)/2 \le j < n(n+1)/2$, we  define
$$
 \beta_j \ = \ \bigl(\frac{(n-1)(n+2)}{2} - j, \, j - \frac{n(n-1)}{2} \bigr).
$$
Then $P^{k,\ell,j}_r \in {\mathbb Z}[X_1, X_2, X_3, \ldots ]$. We will use the multi-index notation
$X^{\alpha} = X_1^{\alpha_1} \cdots X_n^{\alpha_n}$ for a multi-index 
$\alpha = (\alpha_1, \ldots, \alpha_n) \in {\mathbb N}^n$ of length $n$ and size $|\alpha| := \alpha_1 + \cdots + \alpha_n$.
It turns out that the  polynomials
$$
P^{k, \ell, j}_r \ = \ \sum_{\alpha} \, c_{\alpha}^{k,\ell,j, r} \, X^{\alpha}
$$
arising as the coefficients of $L^{k-r}$ in the expression for $Q^{k,\ell}_j$ possess
two different kinds of homogeneity. To formulate this we introduce two weights defined
on each multi-index $\alpha = (\alpha_1, \ldots, \alpha_n)$; namely, 
$$
h(\alpha) \ := \ \sum_{u=1}^n \alpha_u |\beta_u| \ \ \ {\rm and} \ \ \ 
w(\alpha) \ := \ \sum_{u=1}^n \alpha_u (\beta_u^1 + 1).
$$
Here $\beta_u = (\beta_u^1, \beta_u^2) \in {\mathbb N}^2$ with $|\beta_u| = \beta_u^1 + \beta_u^2$.

It will be helpful to use the notation $e^q = (0,0,\ldots,1,0.\ldots )$ to denote the multi-index of any length with all zeros
except for 1 in the $q$th place.

The following proposition contains the structural information we need about the polynomials $Q^{k,\ell}_j$
in order to carry out the proof of Proposiition \ref{key-prop}.

\begin{proposition}\label{Q} Each coefficient $Q^{k,\ell}_j$ arising in Lemma \ref{relations} can be written as
$$
\sum_{r=0}^k P^{k,\ell, j}_r L^{k-r} \ \ {\rm where} \  \
P^{k,\ell,j}_r \ =\!\!\!\!\!\!\!\sum_{\begin{array}{c}\scriptstyle
         \alpha:  \, w(\alpha) = r\\
         \vspace{-5pt}\scriptstyle h(\alpha) = k+\ell - j
         \end{array}}\!\!\!  c_{\alpha}^{k,\ell,j,r} X^{\alpha} \ \in \ {\mathbb Z}[X_1, X_2, \ldots].
$$
Furthermore for any $k,\ell \ge 1$ and $2\le j \le \ell + 1$,
%if $\alpha = e^q$ where $q = q(k,\ell,j)$ is
%the integer satisfying $\beta_q = (k-1, \ell + 1 - j)$ (hence $h(\alpha) = \ell + k - j, \, w(\alpha) = k$
the coefficient $c^{k,\ell,j,k}_{\alpha}$ of $X^{\alpha} = \partial^{k-1, \ell+1-j} L$
 is strictly positive.
\end{proposition}

\begin{proof} The proof will proceed by induction utilising the relations among the $Q^{k,\ell}_j$
described in \eqref{ell=1} and \eqref{ell = 2}. With this in mind we will need to understand the effect
of certain operations on any particular monomial 
$$
X^{\alpha} \ = \ X_1^{\alpha_1} X_2^{\alpha_2} \cdots X_n^{\alpha_n} \ = \
Y^{\alpha_1}_{\beta_1} Y_{\beta_2}^{\alpha_2} \cdots Y_{\beta_n}^{\alpha_n};
$$
namely, $\partial_s X^{\alpha}, \partial_t X^{\alpha}, (\partial^v_t L)  X^{\alpha}$ and $(\partial_s L) X^{\alpha}$. 

${\underline{(\partial_s L) X^{\alpha}}}$: \  Since
$\partial_s  L = Y_{\beta_{1}}$,we have
$(\partial_s L) X^{\alpha}  =  X^{{\tilde{\alpha}}}$ where
\begin{equation}\label{hw}
{\tilde{\alpha}} = \alpha + e^{1} \ \ {\rm and} \ \
h({\tilde{\alpha}}) \ = \ h(\alpha) + 1, \ w({\tilde{\alpha}}) \ = \  w(\alpha) + 2.
\end{equation}
 
${\underline{(\partial^v_t L) X^{\alpha}}}$: \ Since
$\partial^v_t  L = Y_{\beta_{q(v)}}$ where $q(v) = (v+1)(v+2)/2 - 1$ so that 
$\beta_{q(v)}= (0, v)$, we have
$(\partial^v_t L) X^{\alpha}  =  X^{\alpha^{*}}$ where
\begin{equation}\label{h-w}
\alpha^{*} = \alpha + e^{q(v)} \ \ {\rm and} \ \
h(\alpha^{*}) \ = \ h(\alpha) + v, \ w(\alpha^{*}) \ = \  w(\alpha) + 1.
\end{equation}

${\underline{\partial_t X^{\alpha}}}$: \  We have
\begin{equation}\label{t}
\partial_t X^{\alpha} = \sum_{u=1}^n \alpha_u X^{\alpha^u} \ {\rm and} \
X^{\alpha^u}  = Y_{\beta_1}^{\alpha_1} \cdots \bigl(Y_{\beta_u}^{\alpha_u -1} Y_{(\beta_u^1, \beta_u^2 + 1)} \bigr)
\cdots Y_{\beta_n}^{\alpha_n}
\end{equation}
 where $\alpha^u = (\alpha_1, \ldots, \alpha_u - 1, \ldots, \alpha_n) + e^{q_{*}}$ and $q_{*} = q_{*}(u)$ satisfies
$\beta_{q_{*}} = (\beta_u^1, \beta_u^2 + 1)$. 
Hence
$$
h(\alpha^u) \ = \ \sum_{v\not= u} \alpha_v |\beta_v| + (\alpha_u - 1)|\beta_u| + |\beta_u| + 1 \ = \ h(\alpha) + 1
$$
and
$$
w(\alpha^u) \ = \ \sum_{v\not= u} \alpha_v (\beta_v^1 + 1) + (\alpha_u - 1)(\beta_u^1 + 1) + \beta_u^1 +1 \ = \
w(\alpha).
$$

${\underline{\partial_s X^{\alpha}}}$: \  We have
\begin{equation}\label{s}
\partial_s X^{\alpha} = \sum_{u=1}^n \alpha_u X^{{\tilde{\alpha}}^u} \ {\rm and} \
X^{{\tilde{\alpha}}^u}  = Y_{\beta_1}^{\alpha_1} \cdots \bigl(Y_{\beta_u}^{\alpha_u -1} Y_{(\beta_u^1 + 1, \beta_u^2)} \bigr)
\cdots Y_{\beta_n}^{\alpha_n}
\end{equation}
 where ${\tilde{\alpha}}^u = (\alpha_1, \ldots, \alpha_u - 1, \ldots, \alpha_n) + e^{q^{*}}$ and $q^{*}$ satisfies
$\beta_{q^{*}} = (\beta_u^1 + 1, \beta_u^2)$. 
Hence
$$
h({\tilde{\alpha}}^u) \ = \ \sum_{v\not= u} \alpha_v |\beta_v| + (\alpha_u - 1)|\beta_u| + |\beta_u| + 1 \ = \ h(\alpha) + 1
$$
and
$$
w({\tilde{\alpha}}^u) \ = \ \sum_{v\not= u} \alpha_v (\beta_v^1 + 1) + (\alpha_u - 1)(\beta_u^1 + 1) + \beta_u^1 +2 \ = \
w(\alpha) + 1.
$$

Let us now begin the induction argument and we start with $\ell =1$.

{\bf The case $\ell = 1$}: \ Here we want to show that for $2\le j \le k+1$,
\begin{equation}\label{case-ell-1}
Q^{k,1}_j =
\sum_{r=0}^k P^{k,1, j}_r L^{k-r} \ \ {\rm where} \  \
P^{k,1,j}_r \ =\!\!\!\!\!\!\!\sum_{\begin{array}{c}\scriptstyle
         \alpha:  \, w(\alpha) = r\\
         \vspace{-5pt}\scriptstyle h(\alpha) = k+1 - j
         \end{array}}\!\!\!  c_{\alpha}^{k,1,j,r} X^{\alpha}
\end{equation}
and when $j=2$, the coefficient $c^{k,1,2,k}_{\alpha}$ of $X^{\alpha} = \partial^{k-1}_s L$
%in $P^{k,1,2}_k$ 
is strictly positive. We will do this  by induction on $k$.

${\underline{k=1}}$: \ Since $Q^{1,1}_2 = L$, we see \eqref{case-ell-1} holds and that
the coefficient $c^{1,1,2,1}_{\alpha}$ of $X^{\alpha} = L$ is equal to 1.

${\underline{k\ge 2}}$: \ Suppose now that \eqref{case-ell-1} holds for some $k' = k-1 \ge 1$ and any $2\le j \le k' +1$.

From \eqref{ell=1}, see see that when $2\le j \le k$, $Q^{k,1}_j$ is equal to $\partial_s Q^{k-1,1}_j$ plus
a sum of terms of the form $Q^{k-1,1}_{h} (\partial_t^v L)$ for $2\le h \le k$. 
Since by induction
we are assuming \eqref{case-ell-1} holds for $Q^{k-1,1}_j$ when $2\le j \le k' + 1 = k$, we have
$$
\partial_s Q^{k-1,1}_j \ = \
\sum_{r=0}^{k-1} \Bigl[ \partial_s(P^{k-1,1, j}_r) L^{k-1-r} + (k-1-r) P^{k-1,1,j}_r L^{k-2-r} (\partial_s L) \Bigr].
$$
This sum splits into two parts. Using \eqref{s}, we see that the first part is equal to 
$$
\sum_{r=0}^{k-1}\Bigl[\sum_{u=1}^n\!\!\!\!\!\!\!\sum_{\begin{array}{c}\scriptstyle
      \alpha:  \, w({\tilde{\alpha}}^u) = r +1\\
       \vspace{-5pt}\scriptstyle h({\tilde{\alpha}}^u) = k - j +1
      \end{array}}\!\!\!\!\alpha_u  c_{\alpha}^{k-1,1,j,r} X^{{\tilde{\alpha}}^u}\Bigr] L^{k-1-r} =
\sum_{r=1}^{k}\Bigl[\sum_{u=1}^n\!\!\!\!\!\!\!\sum_{\begin{array}{c}\scriptstyle
      \alpha:  \, w({\tilde{\alpha}}^u) = r\\
       \vspace{-5pt}\scriptstyle h({\tilde{\alpha}}^u) = k+1- j
      \end{array}}\!\!\!\!\alpha_u  c_{\alpha}^{k-1,1,j,r-1} X^{{\tilde{\alpha}}^u}\Bigr] L^{k-r} 
$$
which is of the form expressed in \eqref{case-ell-1}. Furthermore when $j=2$ and $r=k$, the coefficient of $\partial_s^{k-1} L$ in
the square bracket expression  is $c^{k-1,1,2,k-1}_{e^q}$ where $q = q(k-1,1,2)$ is the integer
satisfying $\beta_q = (k-2, 0)$. In this case, $k\ge 3$ necessarily and this part of $\partial_s Q^{k-1,1}_j$
does not arise when  $k=j=2$ and $\ell=1$ . The coefficient $c^{k-1,1,2,k-1}_{e^q}$ is strictly
positive by our induction hypothesis. 

 Using \eqref{hw}, we see that the second part is equal to 
$$
\sum_{r=0}^{k-1}(k-1-r) \Bigl[\!\!\!\!\!\!\!\sum_{\begin{array}{c}\scriptstyle
      \alpha:  \, w({\tilde{\alpha}}) = r +2\\
       \vspace{-5pt}\scriptstyle h({\tilde{\alpha}}) = k - j +1
      \end{array}}\!\!\!\!  c_{\alpha}^{k-1,1,j,r} X^{{\tilde{\alpha}}}\Bigr] L^{k-2-r} =
\sum_{r=2}^{k}(k+1-r)\Bigl[\!\!\!\!\!\!\!\sum_{\begin{array}{c}\scriptstyle
      \alpha:  \, w({\tilde{\alpha}}) = r\\
       \vspace{-5pt}\scriptstyle h({\tilde{\alpha}}) = k+1- j
      \end{array}}\!\!\!\! c_{\alpha}^{k-1,1,j,r-2} X^{{\tilde{\alpha}}}\Bigr] L^{k-r} 
$$
which is of the form expressed in \eqref{case-ell-1}.
When $j=2$ and $r=k$, the only way ${\tilde{\alpha}} = \alpha + e^1$
can equal $e^{q(k,1,2)}$ is when $\alpha = 0$ and $k=2$ and in this case,
since $\partial_s Q^{1,1}_2 = \partial_s L$, 
the coefficient of $\partial_s L$ in the above expression is $c^{1,1,2,0}_0 = 1$ which is strictly positive.

Altogether we see that $\partial_s Q^{k-1,1}_j$ is of the form in \eqref{case-ell-1}. Furthermore 
when $j=2$ and $r=k$, the coefficient of $\partial_s^{k-1} L$ is strictly positive.

Using \eqref{h-w} we have
$Q^{k-1,1}_h (\partial_t^v L) = $
$$\sum_{r=0}^{k-1} 
 \Bigl[\!\!\!\!\!\!\!\sum_{\begin{array}{c}\scriptstyle
      \alpha:  \, w(\alpha^{*}) = r + 1\\
       \vspace{-5pt}\scriptstyle h(\alpha^{*}) = k - h + v
      \end{array}}\!\!\!\!  c_{\alpha}^{k-1,1,j,r} X^{\alpha^{*}}\Bigr] L^{k-1-r} \ = \
\sum_{r=1}^{k} 
 \Bigl[\!\!\!\!\!\!\!\sum_{\begin{array}{c}\scriptstyle
      \alpha:  \, w(\alpha^{*}) = r \\
       \vspace{-5pt}\scriptstyle h(\alpha^{*}) = k - h + v
      \end{array}}\!\!\!\!  c_{\alpha}^{k-1,1,j,r-1} X^{\alpha^{*}}\Bigr] L^{k-r}
$$
and we see that the term $X^{\alpha^{*}} = Y_{(k-1, 0)} = \partial_t^{k-1} L$ does not
arise since each $\alpha^{*} = \alpha + e^{q(v)}$ is not of the form $e^q$ with $\beta_q = (k-1,0)$.

Since $Q^{k,1}_j$ for $2\le j \le k$  is equal to $\partial_s Q^{k-1,1}_j$ plus
a sum of terms of the form $Q^{k-1,1}_{h} (\partial_t^v L)$ we see that $Q^{k,1}_j$ satisfies
\eqref{case-ell-1} and the coefficient $c^{k,1,2,k}_{\alpha}$ of $X^{\alpha} = \partial_t^{k-1} L$ is strictly positive.

It remains to consider the case $j=k+1$ but here $Q^{k,1}_{k+1} = L^k$ and so \eqref{case-ell-1} clearly
holds and there is no coefficient of $X^{\alpha} = \partial^{k-1}_t L$ to consider.

{\bf The cases ${\ell\ge 2}$}: \ Here we will show that for any $\ell\ge 2, \, k\ge 1$ and $2\le j \le k+\ell$,
\begin{equation}\label{case-ell-2}
Q^{k,\ell}_j =
\sum_{r=0}^k P^{k,\ell, j}_r L^{k-r} \ \ {\rm where} \  \
P^{k,\ell,j}_r \ =\!\!\!\!\!\!\!\sum_{\begin{array}{c}\scriptstyle
         \alpha:  \, w(\alpha) = r\\
         \vspace{-5pt}\scriptstyle h(\alpha) = k+\ell - j
         \end{array}}\!\!\!  c_{\alpha}^{k,\ell,j,r} X^{\alpha}
\end{equation}
and when $j\le \ell + 1$, the coefficient $c^{k,\ell,j,k}_{\alpha}$ of $X^{\alpha} = \partial^{k-1, \ell + 1 -j} L$
%in $P^{k,1,2}_k$ 
is strictly positive. 

We will do this  by induction on $\ell$ and assume that \eqref{case-ell-2} holds 
for all $1\le \ell' \le \ell - 1, \ k\ge 1$ and $2 \le j\le k + \ell'$. Also we assume that the coefficient of
$X^{\alpha} = \partial^{k-1,\ell' + 1 - j} L$ is strictly positive.

When $2\le j \le k + \ell - 1$, the relations in \eqref{ell = 2} express $Q^{k,\ell}_j$ in terms of
$\partial_t Q^{k,\ell -1}_j$ which puts us in a position to use the induction hypothesis  and write
$$
\partial_t Q^{k,\ell-1}_j \ = \
\sum_{r=0}^{k} \Bigl[ \partial_t(P^{k,\ell-1, j}_r) L^{k-r} + (k-r) P^{k,\ell-1,j}_r L^{k-1-r} (\partial_t L) \Bigr].
$$ 
This sum breaks into two parts $I + II$. By \eqref{t},  
$$
I \ = \ \sum_{r=0}^k \partial_t P^{k,\ell-1,j}_r  L^{k-r} \ = \
\sum_{r=0}^k \Bigl[\sum_{u=1}^n\!\!\!\!\!\sum_{\begin{array}{c}\scriptstyle
      \alpha:  \, w(\alpha^u) = r \\
       \vspace{-5pt}\scriptstyle h(\alpha^u) = k + \ell - j
      \end{array}}\!\!\!\!\alpha_u  c_{\alpha}^{k,\ell-1,j,r} X^{\alpha^u}\Bigr] L^{k-r} 
$$
which is of the form \eqref{case-ell-2}. 

We now examine the coefficient of $\partial^{k-1, \ell + 1 - j} L$ arising in the square bracket above.
The term $\partial_s^{k-1} L$ clearly does not arise (in either $I$ or $II$) and so we may assume $2\le j \le \ell$.
In this case, the term $X^{\alpha^u}= \partial^{k-1, \ell + 1 - j} L$ is given by $\alpha = e^u$
where $q_{*}(u)$ satisfies $\beta_{q_{*}}(u) = (\beta_u^1, \beta_u^2 + 1) = (k-1, \ell + 1 - j)$.  
Therefore $\beta_u = (k-1, (\ell -1) + 1 - j)$ and so with $\alpha = e^u$,
the coefficient of $\partial^{k-1,\ell + 1 - j} L$ above is $c_{\alpha}^{k,\ell-1,j,k}$ which is strictly
positive by the induction hypothesis. Here we are implicitly assuming $k\ge 2$ when $j = \ell$
since otherwise the integer $u$ satisfying $\beta_u = (k-1,\ell-j) = (0,0)$ does not exist. In fact the
term $\partial^{k-1,\ell + 1 - j} L$ does not arise in $I$ when $k=1$ and $j=\ell$ since $Q^{k,\ell-1}_j = L^k$.

We now turn our attention to the sum $II$. By \eqref{h-w}, we have
$$
II \ = \
\sum_{r=0}^k (k-r) 
 \Bigl[\!\!\!\!\!\!\!\sum_{\begin{array}{c}\scriptstyle
      \alpha:  \, w(\alpha^{*}) = r + 1\\
       \vspace{-5pt}\scriptstyle h(\alpha^{*}) = k + \ell - j
      \end{array}}\!\!\!\!  c_{\alpha}^{k,\ell - 1,j,r} X^{\alpha^{*}}\Bigr] L^{k-1-r} 
$$
$$
\ = \ 
\sum_{r=1}^k (k-r +1) 
 \Bigl[\!\!\!\!\!\!\!\sum_{\begin{array}{c}\scriptstyle
      \alpha:  \, w(\alpha^{*}) = r \\
       \vspace{-5pt}\scriptstyle h(\alpha^{*}) = k + \ell - j
      \end{array}}\!\!\!\!  c_{\alpha}^{k,\ell - 1,j,r-1} X^{\alpha^{*}}\Bigr] L^{k-r} 
$$
 which again is of the form \eqref{case-ell-2}. The term $\partial^{k-1,\ell + 1 - j} L$
arises in the sum $II$ precisely when $\alpha = (\alpha_1, \alpha_2 + 1, \alpha_3, \ldots, ) = e^q$
with $\beta_q = (k-1,\ell + 1 - j)$; that is, when $\alpha =0, k=1$ and $j = \ell$. Hence in this case, the
coeffiicient of $\partial^{k-1, \ell + 1 - j} L = \partial_t L$ is $c^{1,\ell - 1, \ell, 0}_{0} = 1$ since
$Q^{1,\ell-1}_{\ell} = L$ and so $\partial_t Q^{1,\ell-1}_{\ell} = \partial_t L$.

Putting $I$ and $II$ together so that $\partial_t Q^{k,\ell-1}_j = I + II$, we see that
$\partial_t Q^{k,\ell -1}_t$ is of the form \eqref{case-ell-2} and the coefficient of
$\partial^{k-1,\ell + 1 - j} L$ is strictly positive. 

By \eqref{ell = 2}, we have $Q^{k,\ell}_2 = \partial_t Q^{k,\ell-1}_2$ and so we have
the desired conclusion \eqref{case-ell-2} in this case. Furthermore when $3\le j \le k+ \ell -1$, we have 
$Q^{k,\ell}_j = \partial_t Q^{k,\ell -1}_j + Q^{k,\ell-1}_{j-1}$ and so by the induction hypothesis,
we see that \eqref{case-ell-2} holds for $Q^{k,\ell}_j$. Furthermore (excluding the case $k=1$
and $j = \ell$) the coefficient of
$\partial^{k-1,\ell + 1 - j} L$ is equal to $c^{k,\ell-1,j,k}_{e^u} + c^{k,\ell -1, j-1, k}_{e^q}$
where $u$ and $q$ are the integers such that $\beta_u = (k-1, \ell - j)$ and $\beta_q = (k-1,\ell + 1 - j)$. 
Being the sum of two strictly positive integers, this coefficient is also strictly positive. In the
case $k=1$ and $j=\ell$, the coefficient is $1 + c^{1,\ell-1,\ell-1,1}_{e^2}$ which is strictly positive. 

Finally since $Q^{k,\ell}_{k+\ell} =  L^k$, we see that \eqref{case-ell-2} holds in all cases and this
completes the proof of the proposition. 
%and so if $Q^{k,\ell}_j$ satisfies the conclusion of the proposition, we have
%$$
%(\partial^v_t L) \, Q^{k,\ell}_j \ = \ \sum_{r=0}^k \ \Bigl( \!\!\!\!\!\!\!\sum_{\begin{array}{c}\scriptstyle
 %        \alpha:  \, w(\alpha^{*}) = r + 1\\
 %        \vspace{-5pt}\scriptstyle h(\alpha^{*}) = k+\ell - j + v
 %        \end{array}}\!\!\!  c_{\alpha}^{k,\ell,j,r} X^{\alpha^{*}} \Bigr)\  L^{k+1-(r+1)}
%$$
%\begin{equation}\label{partial L}
%= \ \ \sum_{r=1}^{k+1} \ \Bigl( \!\!\!\!\!\!\!\!\!\!\!\!\!\!\!\sum_{\begin{array}{c}\scriptstyle
 %        \alpha:  \, w(\alpha^{*}) = r\\
%         \vspace{-5pt}\scriptstyle h(\alpha^{*}) = k + 1+\ell - (j-v+1)
%         \end{array}}\!\!\!\!\!\!\!\!\!\! c_{\alpha}^{k,\ell,j,r} X^{\alpha^{*}} \Bigr) \  L^{k+1-r} .
%\end{equation} 
%
%Next we examine $\partial_t X^{\alpha}$ and note that
%$$
%\partial_t X^{\alpha} \ = \ \sum_{u=1}^n \alpha_u X^{\alpha^u} \ \ {\rm where} \ \ 
%\alpha^u = (\alpha_1, \ldots, \alpha_u - 1, \ldots, \alpha_n) + (0,\ldots,1,\ldots, 0) 
%$$
%and 
%$$
%X^{\alpha^u} \ = \ Y_{\beta_1}^{\alpha_1} \cdots \bigl(Y_{\beta_u}^{\alpha_u -1} Y_{(\beta_u^1, \beta_u^2 + 1)} \bigr)
%\cdots Y_{\beta_n}^{\alpha_n}.
%$$
%The  nonzero entry occuring in the definition of $\alpha^u$ is the 
%$*$th place where $\beta_{*} = (\beta_u^1, \beta_u^2 + 1)$. 
\end{proof}

We can now give the proof of Proposition \ref{key-prop}.

{\bf Proof of Proposition \ref{key-prop}} \ Recall that we are assuming $2\le m$ and $m+r < n$ for some $r\ge 0$.
Furthermore, we suppose
$\partial^m_t \psi(x,y,\omega) \not= 0, \ \partial^n_s \psi(x,y,\omega) \not= 0$
but
$$
\partial^k_s \psi(x, y, \omega)  \ = \ 0 \ \ {\rm for} \ 2\le k < n \ \ \ {\rm and} \ \ 
\partial^{\ell}_t \psi(x, y, \omega)  \ = \ 0 \ \ {\rm for} \ 2\le \ell < m.
$$
%Also we are assuming that $\partial^{k,\ell} L(x,y,\omega) = 0$
%for all $k,\ell$ satisfying $k+\ell \le r-1$ and 
Under these assumptions we are trying to conclude that
$\partial^{k,\ell} L (x,y,\omega) = 0$ for any nonnegative $k,\ell$ satisfying $k+\ell \le r$.

We proceed by induction on $r$. The case $r=0$ was already treated immediately after
the statement of Proposition \ref{key-prop}. We assume $m + r < n$ and suppose that the conclusion of the proposition
holds for some $r-1 \ge 0$; that is, $\partial^{k,\ell} L(x,y,\omega) = 0$
for all $k,\ell$ satisfying $k+\ell \le r-1$.  

It suffices to show that $\partial^{k-1,r-k+1} L(x,y,\omega) = 0$
for any fixed $1\le k \le r+1$. We first note that $\partial^{k, m+r-k} \psi(x,y,\omega) = 0$
by Corollary \ref{max} and so by Lemma \ref{relations},
$$
0 \ = \ \partial^{k,m+r-k}\psi = \sum_{j=2}^{m+r} Q^{k,m+r-k}_j \partial^j_t \psi = Q^{k,m+r-k}_m \partial^m_t \psi
+ \sum_{j=m+1}^{m+r} Q^{k,m+r-k}_j \partial^j_t \psi.
$$
Here we used our assumption that $\partial^j_t \psi = \partial^j_t \psi (x,y,\omega) = 0$ for $2\le j \le m-1$. 

We now claim that $Q^{k,m+r-k}_j (x,y,\omega) = 0$ for every $j\ge m+1$. To see this, consider the monomial
$$
X^{\alpha}\ = \ Y^{\alpha_1}_{\beta_1} \cdots Y^{\alpha_n}_{\beta_n}
$$
for any multi-index $\alpha = (\alpha_1, \ldots, \alpha_n)$ satisfying $h(\alpha) = m+r - j \le r-1$. For any
$1\le u \le n$, this implies that $|\beta_u| \le h(\alpha) \le r-1$. However for any $\beta_u$ with $|\beta_u| \le r-1$, we have
$$
Y_{\beta_u}(x,y,\omega) \ = \ \partial^{\beta_u^1,\beta_u^2} L(x,y,\omega) = 0
$$
by assumption and so we conclude that every monomial $X^{\alpha}(x,y,\omega) = 0$ for any $\alpha$ with
$h(\alpha) = m+r-j$. Therefore when $j\ge m+1$,
$$
Q^{k,m+r-k}_j \ = \ \sum_{s=0}^k P^{k,m+r-k,j}_s L^{k-s} \ = P^{k, m+r-k,j}_k \ = \ 
\sum_{\begin{array}{c}\scriptstyle
         \alpha:  \, w(\alpha) = k\\
         \vspace{-5pt}\scriptstyle h(\alpha) = m+r - j
         \end{array}}\!\!\!  c_{\alpha}^{k,m+r-k,j,k} X^{\alpha} \ = \ 0.
$$
Hence
$$
0 \ = \ \partial^{k,m+r-k}\psi(x,y,\omega) \ = \ Q^{k,m+r-k}_m (x,y,\omega) \, \partial^m_t \psi(x,y,\omega)
$$
and so $Q^{k,m+r-k}_m = 0$ as well since we are assuming $\partial^m_t \psi(x,y,\omega) \not=0$. However
$$
Q^{k,m+r-k}_m \ = \ \sum_{s=0}^k P^{k,m+r-k,m}_s L^{k-s} \ = P^{k, m+r-k,m}_k \ = \ 
\sum_{\begin{array}{c}\scriptstyle
         \alpha:  \, w(\alpha) = k\\
         \vspace{-5pt}\scriptstyle h(\alpha) = r 
         \end{array}}\!\!\!  c_{\alpha}^{k,m+r-k,m,k} X^{\alpha}.
$$
We saw before that $X^{\alpha}$ vanishes if $\alpha = (\alpha_1, \ldots, \alpha_n)$
contiains any component $u$ with $|\beta_u| \le r-1$. Hence for $\alpha$ satisfying
$h(\alpha) = r$, then $X^{\alpha}$ being nonzero forces $\alpha = e^q$ for some $q$ so that
$X^{\alpha} = Y_{\beta_q}$ with $|\beta_q| = r$. If also $k = w(\alpha)$, then
$$
k = w(\alpha) = w(e^q) = \beta_q^1 + 1 \ \Rightarrow \ \beta_1^1 = k-1 \ \ {\rm and} \ \
\beta_q^2 = r - k +1.
$$
In other words, there is only one nonzero monomial $X^{\alpha}$ in the sum above for 
$Q^{k,m+ r - k}_m$ and in fact, we have $Q^{k,m+r-k}_m = c^{k,m+r-k,m,k}_{\alpha} \partial^{k-1, r-k+1} L$.

Since we saw above that $Q^{k,m+r-k}_m(x,y,\omega) = 0$ and since the coefficient $c^{k,m+r-k,m,k}_{\alpha}$
of $\partial^{k-1,r-k+1} L$ is strictly positive by Proposition \ref{Q}, we see that necessarily 
$\partial^{k-1, r-k+1} L(x,y,\omega)$ must be zero. The completes the proof of Proposition \ref{key-prop}.

Before leaving this section, we record one consequence of Proposition \ref{key-prop}.

\begin{corollary}\label{cor-key-prop} Let $\psi$ be a real-analytic function on ${\mathbb T}^2 \times M$ satisfying
the factorisation hypothsis (FH) on $M$. Suppose at a point $(x,y,\omega) \in {\mathbb T}^2 \times M$,
there exist integers $m,n\ge 2$ such that $\partial^m_t \psi(x,y,\omega) \not= 0, \ \partial^n_s \psi(x,y,\omega) \not= 0$
but
$$
\partial^k_s \psi(x, y, \omega)  \ = \ 0 \ \ {\rm for} \ 2\le k < n \ \ \ {\rm and} \ \ 
\partial^{\ell}_t \psi(x, y, \omega)  \ = \ 0 \ \ {\rm for} \ 2\le \ell < m.
$$
Suppose $2^{\mu} m < n$ for some $\mu \ge 0$. Then $Q^{k,\ell}_j (x,y,\omega) = 0$
for $\ell + 2^{-\mu-1} k < j$. 
\end{corollary}

\begin{proof} Since $m\ge 2$, we have 
$$
m + 2(2^{\mu} - 1) \ \le \ m +  m (2^{\mu} -1)  \ = \ 2^{\mu} m \ < \ n
$$
and so we can apply Proposition \ref{key-prop} with $r = 2 (2^{\mu} -1)$ to conclude that
$Y_{\beta} = \partial^{\beta^1, \beta^2} L =  0$ whenever $|\beta| = \beta^1 + \beta^2 \le r$.
In particular we have $L(x,y,\omega) = 0$ and so
%On the other hand,
%The proof is by induction on $\mu$. When $\mu = 0$, then $m < n$ and we can apply
%Proposition \ref{key-prop} to conclude that $Q^{k,\ell}_{k+\elll}(x,y,\omega) = L^k(x,y,\omega) = 0$. 
by Proposition \ref{Q}, we have
$$
Q^{k,\ell}_j = \sum_{r=0}^k P^{k,\ell,j}_r L^{k-r} \ = \ P^{k,\ell,j}_k =
\!\!\!\!\!\!\!\sum_{\begin{array}{c}\scriptstyle
         \alpha:  \, w(\alpha) = k\\
         \vspace{-5pt}\scriptstyle h(\alpha) = k+\ell - j
         \end{array}}\!\!\!  c_{\alpha}^{k,\ell,j,k} X^{\alpha}.
$$

If a multi-index $\alpha = (\alpha_1, \ldots )$ has a component $\alpha_u$ such that
$|\beta_u| \le r = 2 (2^{\mu}-1)$, then the monomial $X^{\alpha} = 0$. 
When $n(n-1)/2 \le u < n(n+1)/2$, 
$|\beta_u| = n-1$ and $|\beta_u| = n-1 \le 2(2^{\mu} - 1)$ precisely when $n \le 2^{\mu+1} -1$.
Hence, in order for $X^{\alpha}$ to be nonzero, then necessarily
the components $\alpha_u$ of $\alpha$ with $n(n-1)/2 \le u < n(n+1)/2$ and $n \le 2^{\mu+1} -1$
must vanish.

Recall that $h(\alpha) = \sum \alpha_u |\beta_u|$ and  $w(\alpha) = \sum \alpha_u (\beta_u^1 + 1)$. 
For any multi-index $\alpha$ and $n\ge 2$, we will denote
$$
N_n \ := \ \sum_{u = n(n-1)/2}^{n(n+1)/2} \alpha_u.
$$
Thus for any multi-index $\alpha$ with $X^{\alpha} \not= 0$, we have
$$
w(\alpha) \ = \ \sum_u \alpha_u (\beta_u^1 + 1) \ = \ \sum_{n\ge 2^{\mu +1}} \sum_{u+n(n-1)/2}^{n(n+1)/2} 
\alpha_u (\beta_u^1 + 1) \le \sum_{n\ge 2^{\mu +1}} n N_n 
$$
$$
= \ \frac{2^{\mu+1}}{2^{\mu+1} - 1} \sum_{n\ge 2^{\mu+1}} \frac{2^{\mu+1} -1}{2^{\mu+1}} 
n N_n \  \le \  \frac{2^{\mu+1}}{2^{\mu+1} - 1} \sum_{n\ge 2} (n-1) N_n \
= \ \frac{2^{\mu+1}}{2^{\mu+1} - 1} \sum_u \alpha _u |\beta_u| 
$$
and the sum in this last quantity is equal to $h(\alpha)$.
Hence if $w(\alpha) = k$ and $h(\alpha) = k+\ell -j$, 
$$
(2^{\mu +1} - 1)k  \ = \ (2^{\mu+1} -1) w(\alpha) \ \le \ 2^{\mu+1} h(\alpha) \ = \ 2^{\mu+1}  (k + \ell - j)
$$ 
or $j \le \ell + 2^{-\mu -1} k$ and this implies $Q^{k,\ell}_j \not= 0 \Rightarrow j \le \ell + 2^{-\mu-1} k$.
This is the desired conclusion.
\end{proof}

\section{A more robust one dimensional theory}\label{robust}

At the heart of Proposition \ref{alpar-ext} is the one dimensional oscillatory integral
estimate \eqref{osc-est}. When we move to the more interesting higher dimensional situations,
the heart of the matter will be higher dimensional oscillatory integrals such as the one appearing in \eqref{form-double}
from Section \ref{k=2},
\begin{equation}\label{form-double-again}
 \int\!\!\!\int_{{\mathbb T}^2} e^{2\pi i [\lambda  \omega 
\cdot {\overline{\phi}}(x + s, y + t) + \rho s + \eta  t]} \, \frac{ds}{s} \frac{dt}{t}.
\end{equation}
 In Section \ref{k=2} we saw quickly the need to understand more
robust versions of \eqref{osc-est}; namely \eqref{osc-est-II} where the one dimensional phase 
$\varphi(s) = \lambda \omega \cdot {\overline{\phi}}(x+s, \tau) + \rho s$ is a certain analytic pertubation
of the phase $\lambda \omega \cdot {\overline{\phi}}(x + s) + \rho s$ appearing in \eqref{osc-est}.
There are other reasons where a more robust estimate such as \eqref{osc-est-II} is needed.
For instance, as in the proof of Propsition \ref{alpar-ext}, we will want to localise the integral above to small $s$ and $t$,
\begin{equation}\label{form-double-local}
 \int\!\!\!\int_{|s|, |t| < \delta} e^{2\pi i [\lambda  \omega 
\cdot {\overline{\phi}}(x + s, y + t) + \rho s + \eta  t]} \, \frac{ds}{s} \frac{dt}{t}.
\end{equation}
for some $\delta>0$ depending on properties of the analytic function $\psi(s,t,\omega) = \omega 
\cdot {\overline{\phi}}(s,t)$ (for instance, perhaps certain derivative bounds of $\psi(x+s,y+t,\omega)$ are satisfied
for $(x,y, \omega)$ near some fixed point $(x_0,y_0, \omega_0)$ and all $s,t$ with $|s|,|t| < \delta$). 
In order to pass from \eqref{form-double-again} to \eqref{form-double-local}, two applications of \eqref{osc-est-II}
are needed with $O(1/\delta)$ error terms. 

Also after we reduced to the integral \eqref{form-double-local}, we
may run into a degenerate situation where $\omega \cdot {\overline{\phi}}(x+s, y+t)
\equiv \omega \cdot {\overline{\phi}}(x+s, y)$. In this case the integral in \eqref{form-double-local} splits
into a product of two oscillatory integrals, one trivial (in $t$) to bound and the other is an integral
of the form appearing in \eqref{osc-est-II}.  There may (and will) be other degenerate situations which arise
where the oscillation splits into a more complicated product of an $s$ oscillation and $t$ oscillation BUT
the region of integration may not be a product such as 
$|s|, |t| < \delta$, and so the integral no longer decomposes into a product of two 1 dimensional oscillatory
integrals. 

Specifically we will find ourselves examining an integral of the form \eqref{form-double-local}
but where $\lambda \omega \cdot {\overline{\phi}}(x+s,y+t) = \lambda_1 \mu \cdot {\overline{f}}(x+s, \tau) + 
\lambda_2 \nu \cdot {\overline{g}}(\sigma, y+t)$ and the region of integration is 
${\mathcal R} = \{|s|,|t| <\delta : A |s|^a < |t| < B |s|^b \}$; hence
$$
\int\!\!\!\int_{{\mathcal R}} e^{2\pi i [\lambda  \omega 
\cdot {\overline{\phi}}(x + s, y + t) + \rho s + \eta  t]} \, \frac{ds}{s} \frac{dt}{t} \ = 
% \int\!\!\!\int_{{\mathcal R}} e^{2\pi i [\lambda_1  \mu \cdot {\overline{f}}(x+s, \tau) +  \rho s]}
%e^{2\pi i [\lambda_2 \nu \cdot {\overline{g}}(\sigma, y+t) + \eta t]} \frac{ds}{s} \frac{dt}{t}  \ =
$$
$$
\int_{|s|<\delta} e^{2\pi i [\lambda_1  \mu \cdot {\overline{f}}(x+s, \tau) +  \rho s]}  \frac{1}{s}\,
\left[\int_{{\mathcal R}(s)} e^{2\pi i [\lambda_2 \nu \cdot {\overline{g}}(\sigma, y+t) + \eta t]} \frac{dt}{t}\right] \, ds
$$
where ${\mathcal R}(s) = \{|t| < \delta : (s,t) \in {\mathcal R} \}$. If we denote by $F(s)$ the inner integral
above, we see that $F$ is an even function and for $s>0$, $|F'(s)| \le  C s^{-1}$ where $C$ depends on
$A,B, a$ and $b$. Of course if we hope
that \eqref{osc-est-II} holds, then we would hope that $F(s)$ defines a bounded function of $s$. This motivates
the following extension of Proposition \ref{alpar-ext} as encapsulated in \eqref{osc-est}.

\begin{proposition}\label{osc-est-III} Let
${\overline{\phi}}(s,t)$ be a $d$-tuple of real-analytic,
1-periodic functions of two variables and set $\psi(s,t,\omega) = \omega \cdot {\overline{\phi}}(s,t)$.
Suppose that $F(s)$ is an even function on $[-1,1]$, $|F(s)| \le A$ and $F \in C^1$ on $(0,1)$
such that $|F'(s)| \le B s^{-1}$ for $s\in (0,1)$. 

Then for every $\omega \in {\mathbb S}^{d-1}$,
$$
C(\omega) \ := \
\sup_{a,b<1, \lambda, \tau, \rho, x} \,
\int_{a<|s|<b} e^{2\pi i [\lambda \psi(x+s, \tau, \omega) +  \rho s]} F(s) \, \frac{ds}{s}  \Bigr|  \ < \ \infty
$$
where $C(\omega)$ depends on $A, B$ and ${\overline{\phi}}$ as well. 
Furthermore, if $\psi$ satisfies the factorisation hypothesis (FH)  on ${\mathbb S}^{d-1}$, 
then $sup_{\omega\in {\mathbb S}^{d-1}} C(\omega) < \infty$.
\end{proposition}

\begin{proof} We will give the proof assuming $F \equiv 1$ so that we replace the amplitude
${\mathcal A}(s) = F(s)/s$ with ${\mathcal A}(s) = 1/s$. As will be evident from the argument below, the only properties
of the amplitute ${\mathcal A}(s)$ we will be using is that it is an odd function, $|{\mathcal A}(s)| \lesssim |s|^{-1}$ and
$|{\mathcal A}'(s)| \lesssim |s|^{-2}$ which clearly holds for $F(s)/s$ by our assumptions. Hence, mainly for
notational ease, we assume $F \equiv 1$. 

We first fix $\omega \in {\mathbb S}^{d-1}$ and allow our estimates to depend
on $\omega$. Initially we put no restrictions on 
the $d$-tuple of analytic functions ${\overline{\phi}}$ and
in particular, the factorisation hypothesis is not assumed to hold. For notational convenience, we suppress
$\omega$ in what follows, writing $\psi(s,t,\omega) = \psi(s,t)$, etc...

{\bf Case 1}: \ There exists a point $(x_0,\tau_0) \in {\mathbb T}^2$ such that
$$
\partial^{k,\ell} \psi (x_0, \tau_0) \ := \ \frac{\partial^{k+\ell}\psi}{\partial s^k \partial t^{\ell}} (x_0,\tau_0) \ = \ 0
\ \ {\rm for \ all}  \  \ k \ge 2, \ \ell \ge 0.
$$
Then
$$
\psi(x_0+s, \tau) = \psi(x_0,\tau) + \partial_s \psi( x_0, \tau) s + \sum_{k=2}^{\infty} \frac{1}{k!} \partial^k_s \psi (x_0, \tau) s^k
\ \ \ {\rm for \ small} \ \ s.
$$
For $k\ge 2$,
$$
\partial^k_s \psi(x_0,\tau) \ = \ \sum_{\ell=0}^{\infty} \frac{1}{\ell!} \partial^{k,\ell} \psi(x_0,\tau_0) (\tau - \tau_0)^{\ell}
\ \equiv 0 \ \ {\rm for} \ \tau \ {\rm near} \ \tau_0.
$$
Hence $\partial^k_s \psi(x_0, \tau) \equiv 0$ for all $\tau$ by analyticity. Therefore
$\psi(x_0 + s, \tau) = \psi(x_0, \tau) + \partial_s \psi(x_0, \tau) s$ for small $s$ and hence all $s$ by
analyticity once again. This forces $\partial_s \psi(x_0, \tau) = 0$ since otherwise we would contradict the
periodicity of $\psi(x_0 + s, \tau)$ in $s$, arriving at the conclusion that for any $\tau \in {\mathbb T}$,
$\psi(x + s, \tau) \equiv \psi(x_0, \tau)$. Hence
$$
\int_{a<|s|<b} e^{2\pi i [\lambda \psi(x+s, \tau, \omega) +  \rho s]} \, \frac{ds}{s} \ = \ 
e^{2\pi i \lambda \psi (x_0, \tau)}\int_{a<|s|<b} e^{2\pi i \rho s} \ \frac{ds}{s}
$$ 
and this last integral is easily seen to be uniformly bounded in $a, b$ and $\rho$.

{\bf Case 2}: \ For every point $(x,\tau)\in  {\mathbb T}^2$, there exists $k = k(x,\tau)\ge 2$
and $\ell = \ell(x,\tau) \ge 0$ such that $\partial^{k,\ell} \psi(x, \tau) \not= 0$.

We choose $\ell = \ell(x,\tau)$ to be minimial with this property; that is, if $\ell < \ell(x,\tau)$, then
$\partial^{k,\ell} \psi(x,\tau) = 0$ for all $k\ge 2$. There is an $\epsilon = \epsilon(x,\tau) > 0$ such that
\begin{equation}\label{der-below}
\bigl| \partial^{k,\ell} \psi(x',\tau') \bigr| \ \ge \ \frac{1}{2} \bigl|\partial^{k,\ell} \psi(x,\tau) \bigr| \ \ 
{\rm for \ all} \ (x',\tau') \in B_{2\epsilon}(x,\tau). 
\end{equation}
We cover the compact ${\mathbb T}^2$ with the family of balls
$\{ B_{\epsilon(x,\tau)}(x,\tau)\}_{(x,\tau)\in {\mathbb T}^2}$ and extract a finite subcover 
$\{B_{\epsilon_j}(x_j,\tau_j)\}_{j=0}^M$,
reducing to establishing a uniform bound for
\begin{equation}\label{T}
 T \ := \ \int_{a< |s| < \epsilon_j }\!\!\!\!\!\! e^{2\pi i [\lambda \psi(x+s, \tau) +  \rho s]} \ \frac{ds}{s}  
\end{equation}
uniformly for $(x,\tau) \in B_{\epsilon_j}(x_j, \tau_j)$ (and also, $a,b, \lambda$ and $\rho$)
for each $0\le j \le M$. This follows from the simple observation that
$$
\int_{a<|s|<b}\!\!\!\!\!\!\! e^{2\pi i [\lambda \psi(x+s, \tau) +  \rho s]} \, \frac{ds}{s}  \ = \
\int_{a< |s| < \epsilon_j }\!\!\!\!\!\! e^{2\pi i [\lambda \psi(x+s, \tau) +  \rho s]} \, \frac{ds}{s} \ +  \ O(\log(1/\epsilon_j)).
$$

For convenience we take $j=0$ and denote $B_0 = B_{\epsilon_0}(x_0,\tau_0)$. 
We also write $k_0 = k(x_0,\tau_0)\ge 2$ and
$\ell_0 = \ell(x_0, \tau_0)\ge 0$ and note that from \eqref{der-below}, we have
\begin{equation}\label{Der-below-II}
\bigl|\partial^{k_0,\ell_0} \psi(x+s, \tau) \bigr| \ge \frac{1}{2} \bigl|\partial^{k_0,\ell_0} \psi(x_0, \tau_0) \bigr| \ =: \
c_0
\end{equation}
for all $(x,\tau) \in B_0$ and $|s| < \epsilon_0$. 

{\bf Claim}: For $(x,\tau) \in {\mathbb T}^2$ and $|s| < \epsilon_0$, we have
$\psi(x+s, \tau) =  P(\tau) + E(s,\tau)$ where $P(\tau) = P_{x_0, \tau_0}(\tau)$ is some polynomial in $\tau$
of degree at most $\ell_0 - 1$ and $E(s,\tau) =  E_{x,\tau_0}(s,\tau)$ satisfies the derivative bound
\begin{equation}\label{key-below}
\bigl| \partial^{k_0}_s E(s, \tau)| \ \ge \ c_0 |\tau - \tau_0|^{\ell_0}
\end{equation}
for all $|s| < \epsilon_0$. 

When $\ell_0 = 0$, we take $P = 0$ so that $E_{x,\tau_0}(s, \tau) = \psi(x+s, \tau)$ and the
claim simply states that $|\partial^{k_0}_s \psi(x+s, \tau)| \ge c_0$ whenever
$(x,\tau) \in B_0$ and $|s| < \epsilon_0$ which of course holds by \eqref{Der-below-II}. So in order
to prove the claim, it suffice to take $\ell_0 \ge 1$. By the minimality of $\ell_0$, we see that for any $\ell < \ell_0$, 
$\partial^{k,\ell} \psi (x_0, \tau_0) = 0$ for all $k\ge 2$ and hence
$$
\partial^{\ell}_t \psi(x_0 + s, \tau_0) = \partial^{\ell}_t \psi(x_0, \tau_0) + \partial^{1,\ell} \psi(x_0, \tau_0) s +
\sum_{k=2}^{\infty} \frac{1}{k!} \partial^{k,\ell} \psi(x_0,\tau_0) s^k 
$$
$$
\ = \  \partial^{\ell}_t \psi (x_0, \tau_0) + \partial^{1,\ell} \psi (x_0, \tau_0) s \ \ \ {\rm for \ small} \ \ s
$$
and therefore for all $s$ by analyticity. Again the periodicity of $\partial^{\ell}_t \psi (x_0 + s, \tau_0)$ in $s$
forces $\partial^{1,\ell} \psi (x_0, \tau_0) = 0$ and so 
\begin{equation}\label{ell}
\partial^{\ell}_t \psi (x + s, \tau_0) \ \equiv \ \partial^{\ell}_t \psi(x_0, \tau_0).
\end{equation}
Using \eqref{ell} with $\ell = 0$,  we see that for any $(x,\tau) \in B_0$ and $|s| < \epsilon_0$, 
$$
\psi(x+s, \tau) \ = \ \psi(x_0, \tau_0) \ + \ \psi(x+s,\tau) \ - \ \psi(x+s, \tau_0) 
$$
$$
\ = \ \psi(x_0, \tau_0) +
(\tau - \tau_0) \int_0^1 \partial_t \psi (x+s, \tau_0 + r (\tau - \tau_0) dr.
$$
Noting that $(x, \tau_0 + \tau') := (x, \tau_0 + r (\tau - \tau_0)) \in B_0$ for any $0<r<1$ and using
\eqref{ell} with $\ell = 1$, we reason similarly to see 
$$
\partial_t \psi (x+s, \tau_0 + \tau') \ = \ \partial_t \psi (x_0, \tau_0) \ + \
\partial_t \psi (x+s, \tau_0 + \tau') \ - \ \partial_t \psi(x+s, \tau_0)
$$
$$
\ = \ \partial_t \psi(x_0, \tau_0) +
(\tau - \tau_0) \int_0^1 \partial^2_t \psi (x+s, \tau_0 + u \tau') du.
$$
Hence
$$
\psi(x+s, \tau) \ = \ \psi(x_0, \tau_0) + \partial_t \psi(x_0,\tau_0) (\tau-\tau_0) + (\tau - \tau_0)^2
\int_0^1\!\!\!\int_0^1 \partial^2_t \psi(x + s, \tau_0 + r u (\tau - \tau_0)) r dr du.
$$
Iterating we conclude that $\psi(x+s, \tau) = P(\tau) + E(s, \tau)$ where
$$
P(\tau) \ = \ P_{x_0, \tau_0}(\tau) \ := \ 
\sum_{\ell=0}^{\ell_0 -1} \frac{1}{\ell!} \partial^{\ell}_t \psi(x_0, \tau_0) (\tau - \tau_0)^{\ell}
$$
and $E(s,\tau) =$
$$
E_{x,\tau_0}(s, \tau)  = 
 (\tau-\tau_0)^{\ell_0} \int_0^1\!\!\!\cdots\!\!\!\int_0^1 \partial^{\ell_0}_t 
\psi (x+s, \tau_0 + r_1\cdots r_{\ell_0} (\tau - \tau_0)) r_2 r_3^2 \cdots r_{\ell_0}^{\ell_0 - 1} dr_1 \cdots dr_{\ell_0}.
$$
Finally the derivative bound \eqref{Der-below-II} implies \eqref{key-below}, completing the
proof of the claim.

 Returning to the oscillatory integral in \eqref{T}, we have
$$
T = e^{2\pi i \lambda P_{x_0,\tau_0}(\tau)} \int_{a<|s|<\epsilon_0} e^{2\pi i [\lambda E(s,\tau) + \rho s]}  \
\frac{ds}{s} 
$$
and by van der Corput's lemma, together with an integration by parts argument, we see that
\eqref{key-below} implies that
$$
T^{>} \ := \ \Bigl|  \int_{|(\tau-\tau_0)^{\ell_0}\lambda|^{-1/k_0} < |s|<\epsilon_0} e^{2\pi i [\lambda E(s,\tau) + \rho s]}  \
\frac{ds}{s} \Bigr| \ \le \ C
$$
where $C$ only depends on $k_0$ and $c_0 = |\partial^{k_0,\ell_0} \psi(x_0,\tau_0)|/2$ (also on the constants $A$
and $B$ if our amplitude is the original $F(s)/s$). In the other region where $|s| <\kappa$ and
$\kappa = \min(|(\tau-\tau_0)^{\ell_0} \lambda|^{-1/k_0}, \epsilon_0)$, we will compare the
complimentary integral
$$
T^{<} :=  \int_{a< |s|<\kappa} e^{2\pi i [\lambda E(s,\tau) + \rho s]}  \
\frac{ds}{s}
$$
to an integral with a certain polynomial phase
$$
{\mathcal C} \ = \   \int_{a< |s|<\kappa} e^{2\pi i [\lambda E(s,\tau) + \rho s]}  \
\frac{ds}{s} \ - \
\int_{a< |s|<\kappa} e^{2\pi i [\lambda Q(s) + \rho s]}  \
\frac{ds}{s}.
$$
The polynomial $Q$ is simply the Taylor polynomial of $E(s,\tau)$ of order $k_0$, thought of as a function of $s$; 
$$
Q(s) \ = \ Q(s,x,x_0,\tau, \tau_0) \ = \ \sum_{k=0}^{k_0-1} \frac{1}{k!} \partial^k_s E(0, \tau) s^k.
$$
Hence $E(s,\tau)  = \ Q(s) \ + \ R_{k_0}(s)$ where the remainder term can be expressed as
$$
R_{k_0}(s) \ = \ s^{k_0} \int_0^1\!\!\!\cdots\!\!\!\int_0^1 \partial^{k_0}_s E(r_1 \cdots r_{k_0} s, \tau) r_2 r_3^2 \cdots
r_{k_0}^{k_0 -1} dr_1 \cdots dr_{k_0}.
$$
Using the formula for $E$ above, we see that the remainder term $R_{k_0}(s)$
satisfies the bound $|R_{k_0}(s)| \lesssim |\tau - \tau_0|^{\ell_0} |s|^{k_0}$ and therefore
$$
|{\mathcal C}| \ \lesssim \  \ |\lambda(\tau - \tau_0)^{\ell_0}| \int_{|s|<\kappa} |s|^{k_0 - 1} ds \ \lesssim \ 1
$$
%$$
%\Bigl[\int_0^1\!\!\!\cdots\!\!\!\int_0^1 \partial^{k_0,\ell_0}_{s,t} \psi (x + r_1...s, \tau_0 + u_1 \cdots (\tau-\tau_0)) u_2 
%u_3^2 \cdots  dr_1 \cdots dr_{k_0} du_1 \cdots du_{\ell_0} \Bigr] (\tau - \tau_0)^{\ell_0} s^{k_0}
%$$
%and hence $|R_{k_0}(s)| \lesssim |\tau - \tau_0|^{\ell_0} |s|^{k_0}$. 
%Therefore we can make a successful comparison
%of the complimentary integral 
%$$
%T^{<} :=  \int_{a< |s|<\kappa} e^{2\pi i [\lambda E(s,\tau) + \rho s]}  \
%\frac{ds}{s};
%$$
%namely
%$$
%\Bigl|  \int_{a< |s|<\kappa} e^{2\pi i [\lambda E(s,\tau) + \rho s]}  \
%\frac{ds}{s}  \  - \  \int_{a< |s|<\kappa} e^{2\pi i [\lambda Q(s) + \rho s]}  \
%\frac{ds}{s} \Bigr| \ \lesssim \ |\lambda(\tau - \tau_0)^{\ell_0}| \int_{|s|<\kappa} |s|^{k_0 - 1} ds 
%$$
since $\kappa \le |\lambda(\tau - \tau_0)^{\ell_0}|^{-1/k_0}$. This shows that, up to a uniformly bounded
error term, the oscillatory integral
$$
T \ = \ e^{2\pi i \lambda P_{x_0,\tau_0}(\tau)} \int_{a<|s|<\kappa} e^{2\pi i [\lambda Q(s) + \rho s]}  \
\frac{ds}{s} \ + \ O(1)
$$
and this last remaining oscillatory integral is uniformly bounded thanks to \eqref{SW}, the useful one dimensional
result of Stein and Wainger.

We now turn to discuss how to achieve uniformity in the parameter $\omega \in {\mathbb S}^{d-1}$ when 
$\psi(s,t,\omega) = \omega \cdot {\overline{\phi}}(s,t)$
satisfies the factorisation hypothesis (FH) on ${\mathbb S}^{d-1}$. Again we divide the analysis into two cases
but the cases are different.

First suppose there exists a point $(x_0,\tau_0, \omega_0) \in {\mathbb T}^2 \times {\mathbb S}^{d-1}$ such that
$$
\partial^k_s \psi (x_0, \tau_0, \omega_0)  \ = \ 0
\ \ \ {\rm for \ all}  \ \  k \ge 2.
$$
Then by Corollary \eqref{degenerate-FH}, we have $\psi(s,\tau,\omega_0) \equiv \psi(x_0, \tau, \omega_0)$
or
$\omega_0 \cdot {\overline{\phi}}(s,\tau) \equiv \omega_0 \cdot {\overline{\phi}}(x_0,\tau)$. Hence there exists
a $\phi_j$ in the $d$-tuple ${\overline{\phi}}$ which we can solve in terms of the others; namely,
$$
\phi_j(s,\tau) \ = \ \sum_{k\not= j} c_k \phi_k(s, \tau) \ + \ \omega_0 \cdot {\overline{\phi}}(x_0,\tau)
$$
and so (after re-writing the integral appearing in the statement of the proposition) we have reduced the dimension
by one, reducing matters to examining an oscillatory integral with $d-1$ real-analytic functions defining
the phase. 

Therefore by a simple induction on dimension argument, we  may assume without loss of generality that for every point 
$(x,\tau,\omega) \in {\mathbb T}^2 \times {\mathbb S}^{d-1}$,
there is a $k = k(x,\tau,\omega) \ge 2$ such that $\partial^k_s \psi (x,\tau,\omega) \not= 0$. Hence there is an
$\epsilon = \epsilon(x,\tau,\omega) >0$ such that 
\begin{equation}\label{der-below-II}
\bigl| \partial^k_s \psi(x',\tau', \omega') \bigr| \ge  \frac{1}{2} \bigl|\partial^k_s \psi(x,\tau, \omega) \bigr| \ 
{\rm for \ all} \ (x',\tau', \omega') \in B_{2\epsilon}(x,\tau,\omega). 
\end{equation}
So we are in a much better position than before, effectively we have reduced ourselves to the
{\bf Case 2} situation above but with $\ell = 0$ and even more, the derivative bound in \eqref{der-below-II}
is more robust than we had in \eqref{der-below} since it holds in an entire neighbourhood of $(x,\tau,\omega)$
and not just in a neighbourhood of $(x,\tau)$ for a fixed $\omega$.

The derivative bound \eqref{der-below-II} is the analogue of \eqref{derivative-bound}
but with one additional parameter $\tau$ and the argument to show that our oscillatory integral
$$
\int_{a<|s|<b} e^{2\pi i [\lambda \psi(x+s, \tau, \omega) +  \rho s]} \, \frac{ds}{s} 
$$
is uniformly bounded now proceeds exactly as in the proof of Proposition \ref{alpar-ext}, establishing
\eqref{osc-est}. The reader can readily check the details.
\end{proof}

\section{Prelude to the proof of Theorem \ref{main}}\label{prelude-I}

Theorem \ref{main} determines exactly when $(\Phi)_{rect}$ holds for any given
real-analytic map $\Phi : {\mathbb T}^2 \to {\mathbb T}^d$. 
In a previous section we saw that $(\Phi)_{rect}$ holds precisely when the oscillatory
integral 
\begin{equation}\label{double-int}
 \int\!\!\!\int_{{\mathbb T}^2} e^{2\pi i {\overline{n}} 
\cdot [{\overline{\phi}}(x + s, y + t) + {\overline{L}}\cdot (s,t)]} D_M(s) D_N(t) \, ds dt
\end{equation}
is uniformly bounded in the parameters ${\overline{n}}, x, y, M$ and $N$; see \eqref{main-est}
and the beginning of Section \ref{k=2}.

The proof of Proposition \ref{alpar-ext} used two standard types of cancellation,
as outlined in Section \ref{perspective}, to reduce the estimate of the one dimensional oscillatory integral
\eqref{osc-est} to the integral \eqref{key-reduction-1d}, an oscillatory integral with a polynomial
phase of bounded degree (but we have no control on the coefficients). At this point we appealed
to a result of Stein and Wainger \eqref{SW} which has no satisfactory counterpart in the multiparameter
setting,  when we move from oscillatory integrals involving a classical CZ kernel $1/t$ to ones involving
a product-type kernel, namely $1/s t$. 

In this section, we begin the analysis of the oscillatory integral \eqref{double-int}. In
a way analogous to what we did in the one dimensional setting, passing from \eqref{osc-est} to \eqref{key-reduction-1d},
 we will reduce the study of \eqref{double-int} to
an oscillatory integral with an {\it almost} polynomial phase of bounded degree.
By {\it almost polynomial} we mean a polynomial modulo a real-analytic part that is a function
of $s$ and $t$ separately; that is, the phase will be of the form $P(s,t) + f(s) + g(t)$ where
$P$ is a polynomial of bounded degree and $f$ and $g$ are real-analytic functions of a single variable.
In this section we will accomplish this under certain derivative bounds of the phase function. This is
the main goal of this section and to acheive this goal, 
we will follow the arguments in \cite{CWW} which dealt with a related problem in the
theory of multiparameter singular Radon transforms.

As we saw in a previous section, using properties of the Dirichlet kernel,  the study of the integral
\eqref{double-int} can be reduced to studying oscillatory integrals of the form
$$
 \int\!\!\!\int_{{\mathbb T}^2} e^{2\pi i [\lambda \omega \cdot {\overline{\phi}}(x +  s, y + t) + \rho s + \eta  t]} \, \frac{ds}{s} \frac{dt}{t}
$$
for some parameters $\lambda, \omega, x, y, \rho$ and $\eta$. Employing estimates for the one 
dimensional oscillatory integral
in Proposition \ref{osc-est-III} twice, 
%either in its strong form $\sup_{\omega \in {\mathbb S}^{d-1}} C(\omega) < \infty$
%which holds under (FH) or in its weak form $C(\omega) <\infty$ for every $\omega \in {\mathbb S}^{d-1}$
%which always hold, 
matters are reduced to considering the oscillatory integral
$$
I \ := \ \int\!\!\!\int_{|s|,|t|< a} e^{i [\lambda \psi(x+ s, y + t, \omega)] + \rho s + \eta t} \frac{ds}{s} \frac{dt}{t}
$$
for some small $0<a<1$ depending on the $d$-tuple ${\overline{\phi}}$
of real-analytic functions.  Here $\psi (s,t, \omega) := \omega \cdot {\overline{\phi}}(s,t)$ is an analytic function
depending on $\omega$. The estimates we derive for $I$ above will
be uniform in $\omega \in {\mathbb S}^{d-1}$ as well as $x,y, \lambda, \rho$ and  $\eta$. They
will depend only on ${\overline{\phi}}$ and certain {\it assumed} derivative bounds of $\psi$ from below.
For notational convience, we suppress the dependence on $\omega$ for the rest of this section
and write $\psi(s,t)$ instead of $\psi(s,t,\omega)$. 

Suppose that there is an integer $m_0\ge 2$ such that 
\begin{equation}\label{derivative-hyp}
\bigl| \partial^{m_0}_t \psi (x+s, y + t)\bigr| =  \Bigl| \frac{\partial^{m_0} \psi}{\partial t^{m_0}} (x+s, y+t)\Bigr|  
\ge  A \ \ {\rm for \ all} \  |s|, |t| < a
\end{equation}
for some $A \lesssim 1$.
Under such a condition, we will show how to successfully compare parts of $I$ to an oscillatory integral
with an {\it almost} polynomial phase of bounded degree. We will consider those parts of $I$ where
the monomial $t^{m_0}$ dominates the majority of the monomials arising in the Taylor expansion
$$
\psi(x+s, y+t) \ = \ \sum_{k,\ell\ge 0} \frac{1}{k!\ell!} \partial^{k,\ell}\psi(x,y) \,  s^k t^{\ell}.
$$
In order to do this efficiently, it will be convenient to decompose $I$ dyadically, resolving the singularities
arising from the kernel $1/s t$. We introduce an appropriate smooth 
nonnegative function $\zeta$, supported in $[-2,-1] \times [1,2]$, and
write the integral $I$ as
$$
\sum_{p,q} \,  \int\!\!\!\int  e^{i [\lambda \psi(x+ s, y + t)] + \rho s + 
\eta t} \zeta(2^p s) \zeta(2^q t) \frac{ds}{s} \frac{dt}{t} \ =: \ \sum_{p,q} \, I_{p,q}
$$
where the sum is over large positive integers $(p,q)$ such that both $2^{-p}, 2^{-q} < a$.

For any pair $(n_1, m_1)$ with $n_1 > 0$ and $m_1 < m_0$, we will consider the part of $I$ where
the monomial $t^{m_0}$ is pointwise larger than $s^{n_1} t^{m_1}$; that is, we consider
$$
I^{>}_{m_0,m_1, n_1}  \ := \ \int\!\!\!\int_{|s^{n_1} t^{m_1}| < |t^{m_0}|} e^{i [\lambda \psi(x+ s, y + t)] + \rho s + \eta t} \frac{ds}{s} \frac{dt}{t} \ = \ \sum_{(p,q) \in {\mathcal R}} I_{p,q}
$$
where ${\mathcal R} = \{(p,q) \, {\rm large} : n_1 p - (m_0 - m_1) q \ge 0 \}$. 
In the region $|s^{n_1}t^{m_1}| <  |t^{m_0}|$,
the only monomials $s^k t^{\ell}$ which are not pointwise dominated by
$t^{m_0}$ are those $(k,\ell)$ satisfying  $k (m_0 - m_1) + \ell n_1 < m_0 n_1$. In other words the
part of the Taylor expansion of $\psi$ NOT controlled in some weak pointwise sense
by the monomial $t^{m_0}$ is
$$
\sum_{\begin{array}{c}\scriptstyle
         k, \ell \ge 0 \\
         \vspace{-5pt}\scriptstyle k (m_0 - m_1) + \ell n_1 < m_0 n_1
         \end{array}} \!\!\!\!\!\!\!\!\!\!\!\!\frac{1}{k!\ell!}\,  \partial^{k,\ell}\psi(x,y) \,  s^k t^{\ell}
$$
which is a polynomial whose degree only depends on $m_0, m_1$ and $n_1$. It will be necessary to
keep all the pure terms in the Taylor expansion (monomials of the form $s^k$ and $t^{\ell}$) and so
we compare $I^{>}_{m_0,m_1, n_1}$ to $II^{>}_{m_0,m_1, n_1} := \sum_{(p,q)\in {\mathcal R}} II_{p,q}$ where
$$
II_{p,q} \ := \  \int\!\!\!\int  e^{i [\lambda (P_{x,y}(s,t) + \psi(x+s,t) +  \psi(x,y+t) - \psi(0,0)) + \rho s + 
\eta t]} \zeta(2^p s) \zeta(2^q t) \frac{ds}{s} \frac{dt}{t} 
$$
where
$$
P_{x,y}(s,t) \ \ := \ \sum_{\begin{array}{c}\scriptstyle
         k,  \ell \ge 1 \\
         \vspace{-5pt}\scriptstyle k (m_0 - m_1) + \ell n_1 < m_0 n_1
         \end{array}} \!\!\!\!\!\!\!\!\!\!\!\!\frac{1}{k!\ell!}\,  \partial^{k,\ell}\psi(x,y) \,  s^k t^{\ell}.
$$
%and
%$$
%E_{x,y}(t) \ := \ \psi(x,y+t) - \sum_{\ell=0}^{m_0 -1} \frac{1}{\ell !} \partial^{\ell}_t \psi(x,y) t^{\ell} \ = \ t^{m_0}
%\int_0^1\!\!\!\!\cdots\!\!\!\!\int_0^1 \partial^{m_0}_t \psi(x,y+{\overline{u}} t) w({\overline{u}}) \,  d{\overline{u}}.
%$$
%Here ${\overline{u}} = u_1 \cdots u_{m_0}$, $w({\overline{u}}) = u_2 u_3^2 \cdots u_{m_0}^{m_0 -1}$
%and $d{\overline{u}} = du_1 \cdots du_{m_0}$.

We set $D_{p,q} =  I_{p,q} - II_{p,q}$ and  make the changes of variables $s' = 2^p s$ and $t' = 2^q t$
in both integrals. We have the simple comparison bound
\begin{equation}\label{compare}
|D_{p,q}|   \lesssim\!\!\!\!\!\!\!\!\!\!\!\!\!\!\sum_{\begin{array}{c}\scriptstyle
         k \ge 1, \ell \ge 1 \\
         \vspace{-5pt}\scriptstyle k (m_0 - m_1) + \ell n_1 \ge m_0 n_1
         \end{array}} \!\!\!\!\!\!\!\!\!\!\!\!\!\!\! |\lambda| 2^{-pk - q\ell} \bigl| \frac{1}{k!\ell!}\,  \partial^{k,\ell}\psi(x,y)\bigr| \le  C
|\lambda| 2^{-\delta N} 2^{-m_0 q}.
\end{equation}
Here $C$ depends  only on ${\overline{\phi}}$ (and in particular independent of $\omega \in {\mathbb S}^{d-1}$),
$\delta = 1/n_1$ and $N := n_1 p - (m_0 - m_1) q \ge 0$. In fact for any $(k,\ell)$ satisfying $k \ge 1$, we have
$$
2^{-kp - \ell q} \ = \
2^{-\frac{k}{n_1}(n_1 p - (m_0 - m_1) q)} 2^{- q(\ell + \frac{m_0 - m_1}{n_1} k)}
\le 2^{-\delta N} 2^{- \frac{1}{n_1} q ( \ell n_1 +  k (m_0-m_1))} 
$$
and therefore $2^{-k p - \ell q} \le 2^{-\delta N} 2^{- m_0 q}$ if in addition, 
$k (m_0 - m_1) + \ell n_1 \ge m_0 n_1$.

The phase function in $I_{p,q}$, after making the changes of variables $s' = 2^p s$ and $t' = 2^q t$, is
given by 
$$
\varphi(s,t) \ = \
\lambda \psi(x+ 2^{-p} s,y+ 2^{-q} t) + \rho 2^{-p} s + \eta 2^{-q} t 
%= \lambda \omega \cdot 
%{\overline{\phi}}(x+ 2^{-p}s,y+ 2^{-q} t) + \rho 2^{-p} s + \eta 2^{-q} t,
$$ 
and so the derivative bound in \eqref{derivative-hyp} implies that
that $|\partial^{m_0}_t \varphi(s,t)| \ge A |\lambda| 2^{-m_0 q}$ for $|s|, |t| \sim 1$. Hence van der Corput's
lemma implies that 
\begin{equation}\label{decay-I}
|I_{p,q}| \ \le \ C_{m_0} \bigl(A |\lambda| 2^{-m_0 q}\bigr)^{- 1/m_0}.
\end{equation}
Similarly, since the largest power of $t$ appearing in $P_{x,y}(s,t)$ is strictly less than $m_0$, we
see that 
\begin{equation}\label{decay-II}
|II_{p,q}| \ \le \ C_{m_0} \bigl(A |\lambda| 2^{-m_0 q}\bigr)^{- 1/m_0}.
\end{equation}
Taking a convex combination of the estimates \eqref{compare}, \eqref{decay-I} and \eqref{decay-II} (and
noting $A\lesssim 1$) we have
$$
|D_{p,q}| \ \le \ C \, 2^{-\delta N} A^{-\epsilon_0} \min(|\lambda| 2^{-q m_0}, (|\lambda| 2^{-q m_0})^{-\epsilon_1})
$$
for some exponents $\delta>0$ (possibly different), $\epsilon_0, \epsilon_1 >0$, depending only on $m_0$ and $n_1$.
The constant $C$ depends on $m_0, n_1$ and ${\overline{\phi}}$. Therefore
$$
\bigl| I^{>}_{m_0, m_1, n_1} - II^{>}_{m_0, m_1, n_1}\bigr| \ = \ | \sum_{(p,q) \in {\mathcal R}} D_{p,q} | \ \le \
C A^{-\epsilon_0}\!\!\!\!\!\!\!\mathop{\sum_{N\ge 0}\sum_{p,q:}}_{n_1 p - (m_0 - m_1) q = N}\!\!\!\!\! |D_{p,q}| 
$$
$$
\le \ C \, A^{-\epsilon_0}
\sum_{N\ge 0} 2^{-\delta N}\!\!\!\!\!\!\!\!\!\!\!\!\!\!\sum_{\begin{array}{c}\scriptstyle
         q: \\
         \vspace{-5pt}\scriptstyle n_1 p -  (m_0 - m_1)q = N 
         \end{array}}\!\!\!\!\!\!\!\!\! \min(|\lambda| 2^{-q m_0}, (|\lambda| 2^{-q m_0} )^{-\epsilon_1})
\ \le \ C \, A^{-\epsilon_0}.
$$

Summarising, under the derivative bound \eqref{derivative-hyp}, we have successfully compared 
$I^{>}_{m_0, m_1, n_1}$, the part 
of the integral $I$ where the
integration is taken over the region $R = \{|s|,|t| < a: |s^{n_1} t^{m_1}| < |t^{m_0}| \}$ to a corresponding
integral with an almost polynomial phase $\varphi(s,t) = P_{x,y}(s,t) + \psi(x+s,y) + \psi(x,y + t)$
of bounded degree;
$$
I^{>}_{m_0,m_1,n_1} \ = \ \int\!\!\!\int_{R} e^{i [\lambda \psi(x+ s, y + t) + \rho s + \eta t]} \frac{ds}{s} \frac{dt}{t}
\ = \ \int\!\!\!\int_R e^{i [\lambda \varphi(s,t) + \rho s + \eta t]} \frac{ds}{s} \frac{dt}{t}
+ O(A^{-\epsilon_0}).
$$
%When we want to apply this argument we will need to ensure that the constant $A\lesssim 1$ stays
%uniformly away from 0. 
We exmphasise here that the constant appearing implicitly in the error bound
$O(A^{-\epsilon_0})$ is a uniform constant and can be taken to be independent of
$x,y,\lambda, \rho, \eta$ and $\omega$. Furthermore in many of our applications, the constant $A$ appearing
above and in a \eqref{derivative-hyp} is a uniform constant, depending on ${\overline{\phi}}$ but
otherwise independent on $x,y,\lambda,\rho, \eta$ and $\omega$. However there is one situation 
where $A$ is not uniform but we still need our bounds to be uniform; that is, independent of $A$.
We discuss this situation in Section \ref{main-fail} below.

\subsection{A small variant}\label{variant}
We note that the above argument still works if the phase function $\lambda \psi(x+s,y+t) + \rho s + \eta t$
 in $I$ or $I^{<}_{m_0,m_1,n_1}$ is replaced by $\lambda \psi(x+s,y+t) + H(s) + \rho s + \eta t$ where $H$
is any real-analytic function. The derivative bound \eqref{derivative-hyp} still implies the
corresponding bound for this new phase since the function $H(s)$ is killed off when computing any derivative in $t$.
Hence we still conclude that if
$$
I^{>}_{m_0,m_1,n_1} \ = \ \int\!\!\!\int_{R} e^{i [\lambda \psi(x+ s, y + t) + H(s) + \rho s + \eta t]} \frac{ds}{s} \frac{dt}{t},
$$
then
$$
I^{>}_{m_0,m_1,n_1} \ = \ 
 \int\!\!\!\int_R e^{i [\lambda \varphi(s,t) + H(s) + \rho s + \eta t]} \frac{ds}{s} \frac{dt}{t}
+ O(A^{-\epsilon_0}).
$$
where again $\varphi(s,t) = P_{x,y}(s,t) + \psi(x+s,y) + \psi(x,y + t)$ is the almost polynomial from before.

\subsection{The case $(n_1, m_1) = (n_0, 0)$} \label{00}
Let us apply the discussion above
to the case $(n_1, m_1) = (n_0, 0)$ to conclude that
$$
I^{>}_{m_0,n_0} \ := \ \int\!\!\!\int_{R} e^{i [\lambda \psi(x+ s, y + t)] + \rho s + \eta t} \frac{ds}{s} \frac{dt}{t}
\ = \ \int\!\!\!\int_R e^{i \lambda \varphi(s,t) + \rho s + \eta t} \frac{ds}{s} \frac{dt}{t}
+ O(A^{-\epsilon_0})
$$
for some $\epsilon_0 > 0$ depending on $m_0$ and $n_0$.  Here
$R = \{|s|,|t|<a : |s|^{n_0} < |t|^{m_0}\}$ and
$\varphi(s,t) = P_{x,y}(s,t) + \psi(x+s,y) + \psi(x,y + t)$ where now
\begin{equation}\label{Pxy}
P_{x,y}(s,t) \ = \ \sum_{\begin{array}{c}\scriptstyle
         k, \ell \ge 1 \\
         \vspace{-5pt}\scriptstyle k m_0  + \ell n_0 < m_0 n_0
         \end{array}} \!\!\!\!\!\!\!\!\!\!\!\!\frac{1}{k!\ell!}\,  \partial^{k,\ell}\psi(x,y) \,  s^k t^{\ell}.
\end{equation}
Now suppose we knew in addition to the derivative bound \eqref{derivative-hyp}, that the
derivative bound (this time with respect to $s$)
\begin{equation}\label{derivative-hyp-II}
\bigl| \partial^{n_0}_s \psi (x+s, y + t)\bigr| =  \Bigl| \frac{\partial^{n_0} \psi}{\partial s^{n_0}} (x+s, y+t) \Bigr|  \ge  B \ \ 
{\rm for \ all} \  |s|, |t| < a
\end{equation}
holds for some $B \lesssim 1$ but with the same $x,y$ (and implicitly for the same $\omega$ in the definition of 
$\psi(s,t) = \psi_{\omega}(s,t) = \omega \cdot {\overline{\phi}}(s,t)$).
For the complimentary integral
where we integrate over $R' = \{|s|,|t|<a: |t|^{m_0} \le |s|^{n_0}\}$, we argue similarly to conclude that
$$
I^{<}_{m_0,n_0} \ = \ \int\!\!\!\int_{R'} e^{i [\lambda \psi(x+ s, y + t)] + \rho s + \eta t} \frac{ds}{s} \frac{dt}{t}
\ = \ \int\!\!\!\int_{R'} e^{i [\lambda \varphi(s,t) + \rho s + \eta t]} \frac{ds}{s} \frac{dt}{t}
+ O(B^{-\epsilon_0})
$$
where $\varphi(s,t) = P_{x,y}(s,t) + \psi(x+s,y) + \psi(x,y+t)$ is the same as before.
% and $P_{x,y}(s,t)$ is defined in \eqref{Pxy}.
%$$
%{\tilde{P}}_{x,y}(s,t) \ = \ \sum_{\begin{array}{c}\scriptstyle
 %        k\ge 0, \ell \ge 1 \\
 %        \vspace{-5pt}\scriptstyle k m_0  + \ell n_0 < m_0 n_0
 %        \end{array}} \!\!\!\!\!\!\!\!\!\!\!\!\frac{1}{k!\ell!}\,  \partial^{k,\ell}\psi(x,y) \,  s^k t^{\ell}.
%$$
Again
we exmphasise here that the constant appearing implicitly in the error bound
$O(B^{-\epsilon_0})$ is a uniform constant and can be taken to be independent of
$x,y,\lambda, \rho, \eta$ and $\omega$. 

Since the oscillatory integral
$$
I \ := \ \int\!\!\!\int_{|s|,|t|< a} e^{i [\lambda \psi(x+ s, y + t)] + \rho s + \eta t} \frac{ds}{s} \frac{dt}{t} \ = \
I^{<}_{m_0,n_0} + I^{>}_{m_0,n_0},
$$
we see that the above argument successfully reduces the study of $I$ to an oscillatory integral
with an almost polynomial phase of bounded degree {\it when} we have managed to reduce ourselves to certain derivative bounds
\eqref{derivative-hyp} and \eqref{derivative-hyp-II} holding. This was the main goal of the section.

\section{The proof of Theorem \ref{main} -- when (FH) holds}

In this section we will show that if $\psi(s,t,\omega) = \omega \cdot {\overline{\phi}}(s,t)$
satisfies the factorisation hypothesis (FH) on ${\mathbb S}^{d-1}$, then the oscillatory integral
\begin{equation}\label{main-double-int}
%\sup_{\lambda, \rho,\eta \in {\mathbb R}, (x,y)\in {\mathbb T}^2} 
\int\!\!\!\!\int_{|s|,|t|<1/2} e^{2\pi i [\lambda \psi(x+s,y+t, \omega) + \rho s + \eta t]} \,
\frac{ds}{s} \frac{dt}{t} 
\end{equation}
is uniformly bounded in $\lambda, \rho, \eta \in {\mathbb R}, \, \omega \in {\mathbb S}^{d-1}$ and $(x,y)\in {\mathbb T}^2$.
From our discussion in Section \ref{preliminary}, this will give a proof of Theorem \ref{main} under the factorisation hypothesis.

First suppose that there exists an $\omega_0 \in {\mathbb S}^{d-1}$ and $(x_0,y_0) \in {\mathbb T}^2$ such that
$$
\partial_s^k \psi(x_0,y_0,\omega_0) = 0 \ \ {\rm for \ all} \ k\ge 2, \ \ \ {\rm and} \ \ \
\partial_t^{\ell} \psi(x_0,y_0,\omega_0) = 0 \ \ {\rm for \ all} \ \ell\ge 2.
$$
By two applications of Corollary \ref{degenerate-FH}, we see that $\psi(s,t,\omega_0) \equiv \psi(x_0,y_0,\omega_0)$
(in fact this follows easily by examining the Taylor expansion of $\psi$ and using periodicity -- one does not need
the factorisation hypothesis here). Hence $\omega_0 \cdot {\overline{\phi}}(s,t) \equiv constant$ and so
there exists
a $\phi_j$ in the $d$-tuple ${\overline{\phi}}$ which we can solve in terms of the others; namely,
$$
\phi_j(s,t) \ = \ \sum_{k\not= j} c_k \phi_k(s, t) \ + \ \omega_0 \cdot {\overline{\phi}}(x_0, y_0)
$$
and so (after re-writing the integral \eqref{main-double-int}) we have reduced the dimension
by one, reducing matters to examining an oscillatory integral with $d-1$ real-analytic functions defining
the phase. Note that after this reduction, the factorisation hypothesis still holds for the $(d-1)$-tuple of real-analytic functions. 

Therefore by a simple induction on dimension argument, we  may assume without loss of generality that for every point 
$(x,y,\omega) \in {\mathbb T}^2 \times {\mathbb S}^{d-1}$, either 
%\begin{enumerate}[label=\protect\circled{\arabic*}]
$$
(1) \ \ 
{\rm there \ exists \ an} \ \  n\ge 2 \ {\rm \ such \ that } \ \partial^n_s \psi(x,y,\omega) \not= 0; \ {\rm or}
$$
$$
(2) \ \ 
{\rm there \ exists \ an} \ \  m\ge 2 \ {\rm \ such \ that } \ \partial^m_t \psi(x,y,\omega) \not= 0.
$$
 For each $(x,y,\omega) \in {\mathbb T}^2 \times {\mathbb S}^{d-1}$, we define 
$\epsilon = \epsilon(x,y,\omega) > 0$ as follows: if (1) holds, choose $\epsilon>0$ such that 
$$
\bigl| \partial^n_s \psi (x',y',\omega') \bigr| \ \ge \ \frac{1}{2} \bigl|\partial^n_s \psi(x,y,\omega)\bigr| \ \ 
{\rm for\ all} \ (x',y',\omega') \in B_{2\epsilon}(x,y,\omega),
$$
and if $(1)$ does not hold (then (2) necessarily holds), we define $\epsilon>0$ analogously using
the $m$th derivative with respect to $t$. 

This gives us a covering $\{B_{\epsilon}(x,y,\omega)\}$
of the compact space ${\mathbb T}^2 \times {\mathbb S}^{d-1}$ from which we extract a finite
subcover $\{B_j\}_{j=0}^M$ and reduce matters to estimating the oscillatory integral in \eqref{main-double-int}
for $(x,y,\omega) \in B_j$ for some $0\le j \le M$. For convenience, we take $j=0$ and write
$B_0 = B_{\epsilon_0}(x_0,y_0,\omega_0)$. We will assume, without loss of generality, that 
(1) holds for $(x_0,y_0,\omega_0)$; that is, 
\begin{equation}\label{derivative-hyp-A}
|\partial^{n_0}_s \psi(x+s,y+t,\omega)|  \ge   A \ \ {\rm for \ all} \ \ (x,y,\omega) \in B_0 \ \ {\rm and} \ \ 
|s|, |t| <\epsilon_0
\end{equation}
where $A = (1/2) |\partial^{n_0}_s \psi(x_0,y_0,\omega_0)| > 0$ and $n_0 = n_0(x_0,y_0,\omega_0) \ge 2$.

Using Proposition \ref{osc-est-III} twice, we see that the integral in \eqref{main-double-int}
is equal to 
$$
\int\!\!\!\!\int_{|s|,|t|<\epsilon_0} e^{2\pi i [\lambda \psi(x+s,y+t, \omega) + \rho s + \eta t]} \,
\frac{ds}{s} \frac{dt}{t} \ + \ O(\log(1/\epsilon_0)).
$$
Here we are using the conclusion of Proposition \ref{osc-est-III} under the factorisation
hypothesis (FH) and so
the implicit constant appearing in the O term can be taken
to be independent of $\omega \in {\mathbb S}^{d-1}$ as well as $\lambda, \rho, \eta, x$ and $y$. 

For the next lemma we use the notation $\psi_e(s,t,\nu) := \nu_1 \phi_1(s,t) + \cdots + \nu_e \phi_e(s,t)$
when $1\le e \le d$ and $\nu = (\nu_1,\ldots, \nu_e) \in {\mathbb S}^{e-1}$. We note that if
$\psi = \psi_d$ satisfies (FH) on ${\mathbb S}^{d-1}$, then for each $1\le e \le d$, 
$\psi_e$ satisfies (FH) on ${\mathbb S}^{e-1}$.

\begin{lemma}\label{k-reduction} After a possible permutation of the $d$-tuple $\overline{\phi} = (\phi_1,\ldots, \phi_d)$,
there exists an $0\le e \le d$ such that 

(a) \ for every $(x,y,\nu) \in {\mathbb T}^2 \times {\mathbb S}^{e-1}$,
there is an $m\ge 2$ such that $\partial^m_t \psi_e(x,y,\nu) \not= 0$,  

{\rm and}

(b) \ there exists $\omega_1 \in {\mathbb S}^{d-1}, \omega_2 \in {\mathbb S}^{d-2}, \ldots, \omega_{d-e} \in {\mathbb S}^e$
and $\{y_1, y_2, \ldots, y_{d-e}\} \subset {\mathbb T}$ such that
\begin{align*}\label{bs}
 b_1 \ \ \ \ \ \ \ \ \ \ \ \ \ \ \ \psi_d(s,t,\omega_1) \ &\equiv \  \psi_d(s,y_1,\omega_1), \ \ \ \ \ \ \ \ \ \ \ \
\omega_{1,d} \not= 0 \\
 b_2  \ \ \ \ \ \ \ \ \ \ \ \ \psi_{d-1}(s,t,\omega_2) \ &\equiv \ \psi_{d-1}(s,y_2,\omega_2), \ \ \ \ \ \ \ \ \ 
 \omega_{2,d-1} \not= 0  \\
\ \ \ \ \ \ \ \ \ \ \ \ \ \ \ \ \ \ \ \ \ \ \ \ \ \ \ \ \ \ & \ \ \ \ \ \ \  \vdots \ \ \ \ \ \ \ \ \ \ \ \ \ \ \ \ \ \ \ \ \ \ \ \ \ \ \ \ \ \ \ \ \  \\
b_{d-e}\ \ \ \ \ \ \psi_{d-e}(s,t,\omega_{d-e}) \ &\equiv \ \psi_{d-e}(s,y_{d-e},\omega_{d-e} ), \ \ \  \omega_{d-e,e+1} \not= 0
\end{align*}
\end{lemma}

We remark that the cases $e=0$ and $e=d$ are allowed. The case $e=0$ means that (a) is vacuous and (b) holds
with $d$ equations. The case $e=d$ means that (b) is vacuous and (a) holds with the original $d$-tuple,
$\psi_d = \psi$.

\begin{proof} If (a) holds with $e=d$, then we are done and so suppose there exists 
$(x_1,y_1, \omega_1) \in {\mathbb T}^d \times {\mathbb S}^{d-1}$ such that $\partial^m_t \psi(x_1,y_1,\omega_1) = 0$
for all $m\ge 2$. Then Corollary \ref{degenerate-FH} implies that $\psi(s,t,\omega_1) \equiv \psi(x,y_1,\omega_1)$;
that is, $b_1$ holds. Without loss of generality, we may assume $\omega_{1,d} \not= 0$ (otherwise
make a permutation). Hence (b) holds with $e = d-1$. If (a) holds with $e=d-1$, we are done.

If (a) does not hold with $e=d-1$, then there exists $(x_2,y_2,\omega_2) \in {\mathbb T}^2 \times {\mathbb S}^{d-2}$
such that $\partial^m_t \psi_{d-1}(x_2,y_2, \omega_2) = 0$ for all $m\ge 2$. Since $\psi_{d-1}$ satsifies
(FH) on ${\mathbb S}^{d-2}$, Corollary \ref{degenerate-FH} implies that 
$\psi_{d-1}(s,t,\omega_2) \equiv \psi_{d-1}(s,y_2,\omega_2)$. Without loss of generality, we may
assume that $\omega_{2,d-1}\not= 0$. Hence (b) holds with $e = d-2$. If (a) holds with $e=d-2$
we are done. 

If not, there exists $(x_3,y_3,\omega_3) \in {\mathbb T}^2 \times {\mathbb S}^{d-3}$ such that
$\partial^m_t \psi_{d-2}(x_3,y_3,\omega_3) = 0$ for all $m\ge 2$. And so on... 

A simple induction argument gives a complete proof of the lemma.
\end{proof}

Lemma \ref{k-reduction} implies that we may write 
$$\psi(s,t,\omega) \ = \ \omega \cdot {\overline{\phi}}(s,t) \ = \
\sum_{j=1}^e \omega_j \cdot \phi_j(s,t) + \omega_{e+1} \bigl(\sum_{j=1}^d c_j^{e+1} \phi_j(s,t) + G_{e+1}(s)\bigr)
$$
$$
+ \ \omega_{e+2} \Bigl(\sum_{j=1}^e c_j^{e+2} \phi_j(s,t) + c^{e+2}_{e+1} \Bigl[ \sum_{j=1}^e c_j^{e+1} \phi_j(s,t) +
G_{e+1}(s)\Bigr] + G_{e+2}(s)\Bigr) + \cdots
$$
\begin{equation}\label{pre-psi-form}
\ = \ \ \ \sum_{j=1}^e c_j(\omega) \phi_j(s,t) \ + \ \sum_{j=e+1}^d \omega_j F_j(s)
\end{equation}
for some real-analytic, 1-periodic functions $\{F_j \}_{j=e+1}^d$.  If ${\overline{c}}_e(\omega) := 
(c_1(\omega), \ldots, c_e(\omega)) \in {\mathbb R}^e \setminus \{0\}$, then 
\begin{equation}\label{psi-form}
\psi(s,t,\omega) \ =: \ \lambda_e(\omega) \psi_e(s,t, \nu_e(\omega)) \ + \ \sum_{j=e+1}^d \omega_j F_j (s)
\end{equation}
where this equation defines $\lambda_e(\omega) > 0$ and $\nu_e(\omega) \in {\mathbb S}^{e-1}$
as functions of $\omega \in {\mathbb S}^{d-1}$ {\it when} ${\overline{c}}_e(\omega) \not= 0$.

Lemma \ref{k-reduction} also implies that for every $(x,y,\nu) \in {\mathbb T}^2 \times {\mathbb S}^{e-1}$,
there is an $m = m(x,y,\nu) \ge 2$ such that $\partial^m_t \psi_e(x,y,\nu) \not= 0$. Hence for
every $(x,y,\nu) \in {\mathbb T}^2 \times {\mathbb S}^{e-1}$, we can choose a 
$\delta = \delta(x,y,\nu) >0$ such that
$$
\bigl| \partial^m_t \psi_e(x',y',\nu') \bigr| \ \ge \ \frac{1}{2} \bigl| \partial^m_t \psi_e(x,y,\nu) \bigr| \ \ 
{\rm for \ all} \ (x',y',\nu') \in B_{2\delta}(x,y,\nu).
$$
This gives us an open covering $\{B_{2\delta}(x,y,\nu)\}$ of the compact space ${\mathbb T}^2 \times {\mathbb S}^{e-1}$
from which we extract a finite subcover $\{B_j\}_{j=0}^M$ with the property that for
every $(x,y,\nu) \in {\mathbb T}^2 \times {\mathbb S}^{e-1}$, there is a $j = j(x,y,\nu) \in \{0,1,\ldots, M\}$
such that $(x,y,\nu) \in B_{2\delta_j}(x_j,y_j,\nu_j)$ and 
\begin{equation}\label{derivative-hyp-B}
|\partial^{m_j}_s \psi_e (x+s,y+t,\omega)|  \ge   B_j \ \ \ {\rm for \ all}  \ \  
|s|, |t| <\delta_j
\end{equation}
where $B_j = (1/2) |\partial^{m_j}_t \psi_e (x_j, y_j, \nu_j)| > 0$.

If $e=0$, then 
$$
\psi(s,t,\omega) \ = \ \omega \cdot {\overline{\phi}}(s,t) \ = \ \sum_{j=1}^d \omega_j F_j(s)  \ =: \
\omega \cdot {\overline{F}}(s)
$$ 
and so
$$
 \int\!\!\!\!\int_{|s|,|t|<\epsilon_0} e^{2\pi i [\lambda \psi(x+s,y+t, \omega) + \rho s + \eta t]} \,
\frac{ds}{s} \frac{dt}{t} \ = \ \int_{|t|<\epsilon_0} e^{2\pi i \eta t} \frac{dt}{t} \cdot
\int_{|s|<\epsilon_0} e^{2\pi i [\lambda \omega \cdot {\overline{F}}(s) + \rho s]} \frac{ds}{s}
$$
splits as  a product of two oscillatory integrals both of which are uniformly bounded by Proposition \ref{osc-est-III}.

We may therefore assume $e\ge 1$. We again apply Proposition \ref{osc-est-III} to reduce (up to error terms 
depending only on $\{\delta_j\}_{j=1}^M$) our analysis of the
oscillatory integral \eqref{main-double-int} to
$$
I \ := \  \int\!\!\!\!\int_{|s|,|t|<\kappa} e^{2\pi i [\lambda \psi(x+s,y+t, \omega) + \rho s + \eta t]} \,
\frac{ds}{s} \frac{dt}{t} 
$$
where $\kappa = \min(\epsilon_0, \delta_0, \ldots, \delta_M )$ and where the derivative bound
\eqref{derivative-hyp-A} holds. We fix $\omega \in {\mathbb S}^{d-1}$ and consider two cases.

{\bf Case 1}: \ ${\overline{c}}(\omega) = 0$. In this case, we see by \eqref{pre-psi-form},
$$
\psi(x+s,y+t,\omega) \ = \ \omega \cdot {\overline{\phi}}(x+s,y+t) \ = \ \sum_{j=e+1}^d \omega_j F_j(x+s)
$$
 and so
$$
I \ = \  \int_{|t|<\kappa} e^{2\pi i \eta t} \frac{dt}{t} \ \cdot \
\int_{|s|<\kappa} e^{2\pi i [\sum_{j=e+1}^d \lambda \omega_j F_j(x+s) + \rho s]} \frac{ds}{s}
$$
splits once again into a product of two oscillatory integrals both of which are uniformly bounded
by Proposition \ref{osc-est-III}.

{\bf Case 2}: \ ${\overline{c}}(\omega) \not= 0$. Then $\nu_e = \nu_e(\omega) \in {\mathbb S}^{e-1}$ and we see
by \eqref{psi-form},
$$
\psi(x+s,y+t,\omega) = \omega \cdot {\overline{\phi}}(x+s,y+t) = \lambda_{e} \psi_e(x+s,y+t,\nu_e) \ + \ 
\sum_{j=e+1}^d \omega_j F_j(x+s)
$$
where $\lambda_{e} = \lambda_e(\omega)$. Furthermore $(x,y,\nu_e(\omega)) \in B_{2\delta_j}(x_j, y_j, \nu_j)$
for some $0\le j \le M$. For convenience we assume $j=0$. Hence by \eqref{derivative-hyp-B},
$$
\bigl|\partial^{m_0}_t \psi_e (x+s,y+t,\nu_e(\omega))\bigr| \ \ge \ B_0 \ \ {\rm for \ all} \ \ |s|,|t| < \delta_0
$$
or
\begin{equation}\label{derivative-hyp-B-II}
\bigl| \partial^{m_0}_t \psi (x+s,y+t,\omega) \bigr| = \bigl| \partial^{m_0}_t (\omega \cdot {\overline{\phi}})(x+s,y+t)\bigr|
\ \ge \ B_0 \lambda_{e}
\end{equation}
holds for all $|s|,|t| < \delta_0$. We are now in a position to apply the arguments in Section \ref{prelude-I} to successfully
compare parts of the integral $I$ to oscillatory integrals with almost polynomial phases of bounded degrees. 

We split $I = I^{>}_{m_0,n_0} + I^{<}_{m_0,n_0} =: I^{>}_0 + I^{<}_0$ as in Section \ref{prelude-I} (more specifically, see
Section \ref{variant}). The treatment of these two integrals is similar and we choose to concentrate on
$$
I^{>}_{0} \ = \ \int\!\!\!\int_{R_0} 
e^{i [\lambda \psi(x+ s, y + t, \omega) + \rho s + \eta t]} \frac{ds}{s} \frac{dt}{t}
$$
where $R_0 = \{|s|,|t|<\kappa: |s|^{n_0} <|t|^{m_0}\}$. 
In fact $I^{>}_0$ is slightly more involved than the complimentary integral $I^{<}_{0}$.
We note that
$$
\lambda \psi(x+s,y+t,\omega) = (\lambda \lambda_e) \psi_e(x+s,y+t,\nu_e) + \sum_{j=e+1}^d \lambda \omega_j F_j(x+s)
$$
and that \eqref{derivative-hyp-B} holds with $j=0$. This puts us in a position to apply the estimate in Section
\ref{variant} to conclude that
$$
I^{>}_{0} \ = \ \int\!\!\!\int_{R_0} 
e^{i [(\lambda\lambda_e) {\tilde{\varphi}}(s,t) + \sum_{j=e+1}^d \lambda \omega_j F_j(x+s)
\rho s + \eta t]} \frac{ds}{s} \frac{dt}{t}
\ + \  O(1)
$$
where 
$${\tilde{\varphi}}(s,t)  \ = \
P^{*}_{x,y,\omega}(s,t) \ + \ \psi_e(x+s,y, \nu_e) \ + \ \psi_e(x,y + t, \nu_e) \ - \ \psi_e(0,0,\nu_e)
$$ 
and
\begin{equation}\label{Qxy}
P^{*}_{x,y,\omega}(s,t) \ = \ \sum_{\begin{array}{c}\scriptstyle
         k, \ell \ge 1 \\
         \vspace{-5pt}\scriptstyle k m_0  + \ell n_0 < m_0 n_0
         \end{array}} \!\!\!\!\!\!\!\!\!\!\!\!\frac{1}{k!\ell!}\,  \partial^{k,\ell}\psi_e(x,y,\nu_e(\omega)) \,  s^k t^{\ell}.
\end{equation}
Since $\partial^{k,\ell} \psi (x,y,\omega) = \lambda_e(\omega) \partial^{k,\ell} \psi_e (x,y,\nu_e(\omega))$
whenever $\ell \ge1$, we see that 
\begin{equation}\label{P=P*}
P_{x,y,\omega}(s,t) \ = \ \lambda_e P^{*}_{x,y,\omega}(s,t) 
\end{equation}
where $P_{x,y,\omega}$ is defined in \eqref{Pxy} with respect to $\psi(s,t) = \psi(s,t,\omega)$. 
Furthermore,
$$
\psi(x+s,y,\omega) + \psi(x,y+t,\omega)  - \psi(0,0,\omega) = \lambda_e  \bigl[ \psi_e(x+s,y, \omega) + \psi_e(x,y+t,\nu_e)
$$
\begin{equation}\label{psi=psie}
 - \psi_e(0,0,\nu_e) \bigr] \ + \  \sum_{j=e+1}^d \omega_j F_j(x+s) \ + \
\sum_{j=e+1}^d \omega_j (F_j(x)-F_j(0)) 
\end{equation}
and therefore
$$ 
I^{>}_{0} \ = \ e^{-2\pi i \lambda A(x,\omega)}\int\!\!\!\int_{R_0} 
e^{i [\lambda \varphi(s,t) + \rho s + \eta t]} \frac{ds}{s} \frac{dt}{t}
\ + \  O(1)
$$
where $A(x,\omega) = \sum_{j=e+1}^d \omega_j (F_j(x) - F_j(0))$ and 
\begin{equation}\label{varphi}
\varphi(s,t) \ = \ P_{x,y,\omega}(s,t) + \psi(x+s,y,\omega) + \psi(x,y + t,\omega) - \psi(0,0,\omega).
\end{equation}
Slightly more
straightforward reasoning leads to
$$
I^{<}_{0} \ = \ \int\!\!\!\int_{|t|^{m_0}<|s|^{n_0}} 
e^{i [\lambda \varphi(s,t) + \rho s + \eta t]} \frac{ds}{s} \frac{dt}{t}
\ + \  O(1).
$$

Therefore we are left with understanding
two oscillatory integrals with almost polynomial phases of bounded degree. We choose, without loss
of generality, to concentrate on
$$
{\mathcal I} \ := \ \int\!\!\!\int_{R_0} 
e^{i [\lambda \varphi(s,t) + \rho s + \eta t]} \frac{ds}{s} \frac{dt}{t}
$$
where
$\varphi = \varphi_{x,y,\omega}$ is the almost polynomial phase defined above in \eqref{varphi}.
Recall that we are fixing $x,y$ and $\omega$ in this discussion and we have deduced various properties 
of $\psi$ which hold at $(x,y,\omega)$. In particular we know that 
both $\partial^{n_0}_s \psi(x,y,\omega)$ and $\partial^{m_0}_t \psi(x,y,\omega)$
are nonzero,  see \eqref{derivative-hyp-A} and \eqref{derivative-hyp-B-II}. However the exponents
$m_0, n_0\ge 2$ may not be minimal with respect to this property. Let $m=m(x,y,\omega) \ge 2$ be minimal
with repect to $\partial^m_t \psi(x,y,\omega) \not= 0$. In particular we have
$\partial^{\ell}_t \psi(x,y,\omega) = 0$
for all $2\le \ell < m$. Hence $\partial^m_t \psi_e(x,y,\nu_e(\omega)) \not= 0$ and
$\partial^{\ell}_t \psi_e(x,y,\nu_e(\omega)) = 0$ for $2\le \ell < m$.
 
Likewise, let $n = n(x,y,\omega) \ge 2$
be minimal with respect to $\partial^n_s \psi(x,y,\omega) \not= 0$. 
By Corollary \ref{max}, we have $\partial^{k,\ell}\psi(x,y,\omega) = 0$ whenever $k+\ell < \max(m,n)$.
Hence we may write
$$
P_{x,y,\omega}(s,t) \ = \ \sum_{\begin{array}{c}\scriptstyle
         k + \ell \ge \max(m,n) \\
         \vspace{-5pt}\scriptstyle k m_0  + \ell n_0 < m_0 n_0
         \end{array}} \!\!\!\!\!\!\!\!\!\!\!\!\frac{1}{k!\ell!}\,  \partial^{k,\ell}\psi(x,y,\omega) \,  s^k t^{\ell}.
$$

We decompose ${\mathcal I} = {\mathcal I}'+ {\mathcal I}''$ where
$$
{\mathcal I}' \ := \  \int\!\!\!\int_{R'} 
e^{i [\lambda \varphi(s,t) + \rho s + \eta t]}\  \frac{ds}{s} \frac{dt}{t}
$$
where $R' = \{(s,t) \in R_0 : |s|^{n} < |t|^m \}$. We will show that
${\mathcal I}'$ is uniformly bounded. The argument for ${\mathcal I}''$ is similar although slightly
less notationally cumbersome. The first step is to reduce the region of integration to
$R = \{(s,t) \in R_0 : C |s|^n \le |t|^m\}$ for some large but absolute constant $C$ which may depend on the
$d$-tuple of analytic functions ${\overline{\phi}}$ but it will not depend on the other parameters,
$x,y,\lambda, \rho$ and $\eta$. 

\begin{proposition}\label{C}
The integral
$$
\int\!\!\!\int_{\{(s,t)\in R_0: C^{-1}|t|^m < |s|^n < |t|^m\}} 
e^{i [\lambda \varphi_{x,y,\omega}(s,t) + \rho s + \eta t]}\  \frac{ds}{s} \frac{dt}{t}
$$
is uniformly bounded in $x,y,\omega, \lambda, \rho$ and $\eta$ 
\end{proposition}

Fixing either $s$ or $t$, the integral in the other variable has a uniformly bounded logarithmic measure
and so one expects that this double oscillatory integral is effectively a single oscillatory integral and
we should then be appealing to Proposition \ref{osc-est-III} or some variant. However it is a slightly technical
argument to decouple the integration and reduce to a one dimensional integral. We choose to postpone the
proof of this proposition until the end of the section so as not to detract from the main line of argument.

By Proposition \ref{C} the study of ${\mathcal I}'$ is reduced to 
$$
{\mathcal I}_1 \ := \  \int\!\!\!\int_{R} 
e^{i [\lambda \varphi(s,t) + \rho s + \eta t]}\  \frac{ds}{s} \frac{dt}{t}.
$$
As we have done previously, we decompose ${\mathcal I}_1$ dyadically
${\mathcal I}_1 \ = \ \sum_{(p,q) \in {\mathcal R}_1} I_{p,q}$ where
$$
I_{p,q} \ = \  \int\!\!\!\int  e^{2\pi i [\lambda \varphi(s, t) + \rho s + 
\eta t]} \zeta(2^p s) \zeta(2^q t) \frac{ds}{s} \frac{dt}{t}
$$
where ${\mathcal R}_1 = \{(p,q) \in {\mathcal R} : C 2^{-pn}\le 2^{-qm} \} = \{(p,q) \in {\mathcal R} :
 pn-qm \ge N_0\}$ for some large absolute constant $C$, large integer $N_0$. 
Here ${\mathcal R}$ is as in Section \ref{prelude-I} which is defined as those
pairs $(p,q)$ with both $2^{-p}, 2^{-q} < \kappa$ such that $2^{-pn_0} \le 2^{-qm_0}$.

We will compare ${\mathcal I}_1$ to ${\mathcal I}_2$ where
$$
{\mathcal I}_2 \ = \ \int\!\!\!\int_{R} e^{2\pi i [\lambda (\psi(x+s,t,\omega) + \psi(x,y+t,\omega)) + \rho s + \eta t]}
\frac{ds}{s} \frac{dt}{t}.
$$
In fact we will show that ${\mathcal I}_1 = {\mathcal I}_2 + O(1)$. Since we can write
$$
{\mathcal I}_2 = \int_{|s|\le \kappa} e^{2\pi i [\lambda \psi(x+s,y,\omega) + \rho s]} G(s) \frac{ds}{s} \ \ 
{\rm where} \ G(s) = \int_{R(s)} e^{2\pi i [\lambda \psi(x,y+t,\omega) + \eta t]} \frac{dt}{t}
$$
where $R(s) = \{t: (s,t) \in R\}$. By Proposiiton \ref{osc-est-III}, we see that $G(s) = O(1)$. Also since
$|G'(s)| \lesssim |s|^{-1}$, we can appeal to Proposition \ref{osc-est-III} once again to conclude that
${\mathcal I}_2$ and hence ${\mathcal I}_1$ is $O(1)$ as desired.

As with ${\mathcal I}_1$ we will write ${\mathcal I}_2 =\sum_{(p,q) \in {\mathcal R}_1} II_{p,q}$ where
$$
II_{p,q} \ = \  \int\!\!\!\int  e^{2\pi i [\lambda (\psi(x+s, y,\omega) + \psi(x,y+t,\omega)) + \rho s + 
\eta t]} \zeta(2^p s) \zeta(2^q t) \frac{ds}{s} \frac{dt}{t}.
$$
We write $D_{p,q} = I_{p,q} - II_{p,q}$ and derive various estimates for $D_{p,q}$ which will
add up to a uniform estimate when we sum over $(p,q) \in {\mathcal R}_1$. 

 We make the changes of variables $s \to 2^{-p} s$
and $t\to 2^{-q} t$ in both integrals and after making these changes of variables, 
the difference of the
phases in $I_{p,q}$ and $II_{p,q}$ is $2\pi i \lambda P_{x,y,\omega}(2^{-p}s, 2^{-q}t)$. Hence the difference
$D_{p,q}$ has the bound
$$
|D_{p,q}| \ \le \ 2\pi |\lambda| \sum_{\begin{array}{c}\scriptstyle
         k + \ell \ge \max(m,n) \\
         \vspace{-5pt}\scriptstyle k m_0  + \ell n_0 < m_0 n_0
         \end{array}} \!\!\!\!\!\!\!\!\!\!\!\!\frac{1}{k!\ell!}\, \bigl| \partial^{k,\ell}\psi(x,y,\omega)\bigr| \,  
2^{-pk - q\ell}.
$$
Recall that for $k,\ell\ge 1$, we have $\partial^{k,\ell} \psi(x,y,\omega) = \lambda_e \partial^{k,\ell} \psi_e(x,y,\nu_e)$
where $\lambda_e = \lambda_e(\omega)$ and $\nu_e =\nu_e (\omega)$. Hence
\begin{equation}\label{Dpq}
|D_{p,q}| \ \le \ 2\pi |\lambda \lambda_e| \sum_{\begin{array}{c}\scriptstyle
         k + \ell \ge \max(m,n) \\
         \vspace{-5pt}\scriptstyle k m_0  + \ell n_0 < m_0 n_0
         \end{array}} \!\!\!\!\!\!\!\!\!\!\!\!\frac{1}{k!\ell!}\, \bigl| \partial^{k,\ell}\psi_e(x,y,\nu_e)\bigr| \,  
2^{-pk - q\ell}.
\end{equation}

At the heart of the argument is a comparison of the weighted mixed derivatives \\
$2^{-pk-q\ell} \partial^{k,\ell} \psi_e(x,y,\nu_e)$ to weighted pure derivatives 
$2^{-q\ell}\partial^{\ell}_t \psi_e(x,y,\nu_e)$
with an exponential decay gain. This is the analogue of controlling mixed vector fields by pure ones in the
work of Stein and Street, see \cite{BS}. The following proposition is the main proposition which gives this
comparison.

\begin{proposition}\label{main-prop} For
any $k,\ell$ satisfying $\max(m,n) \le k+\ell$ and $km_0 + \ell n_0 < m_0 n_0$, we have
\begin{equation}\label{Main-est}
2^{-pk-q\ell} \bigl| \partial^{k,\ell} \psi_e(x,y,\nu_e) \bigr| \lesssim 2^{-\delta (pn-qm)} 
\sum_{j=m}^{m_0} 2^{-q j} \bigl|\partial^{j}_t \psi_e(x,y,\nu_e)\bigr|
\end{equation}
where $\delta = 1/n_0 >0$.
\end{proposition}

\begin{proof} The proof relies on the consequences of the factorisation hypothesis (FH)
we derived in Section \ref{consequences}. We will apply the results there to $\psi = \psi_e$
and $M = {\mathcal S}^{e-1}$, recalling that $\psi_e$ satisfies (FH) on ${\mathcal S}^{e-1}$.
By Lemma \ref{relations}, we have
$$
\partial^{k,\ell} \psi_e(x,y,\nu_e) \ = \ \sum_{j=2}^{k+\ell} Q^{k,\ell}_j(x,y,\nu_e) \partial^j_t \psi_e(x,y,\nu_e) \ = \
\sum_{j=m}^{k+\ell} Q^{k,\ell}_j \partial^j_t \psi_e
$$
since $\partial^j_t \psi_e(x,y,\nu_e) = 0$ for $2\le j <m$. 
Therefore
%\begin{equation}\label{step-1}
$$
2^{-pk-q\ell} \bigl| \partial^{k,\ell} \psi_e(x,y,\nu_e)\bigr| \le  \sum_{j=m}^{k+\ell}
2^{-pk-q\ell} \bigl|Q^{k,\ell}_j(x,y,\nu_e)\partial^j_t \psi_e(x,y,\nu_e)\bigr|.
$$
To show \eqref{Main-est}, it suffices to show 
\begin{equation}\label{step-2}
\sum_{j=m}^{k+\ell} 2^{-pk - q\ell} \bigl|Q^{k,\ell}_j \partial^j_t \psi_e(x,y,\nu_e)\bigr| \ \lesssim \ 
2^{-\delta N} \sum_{j=m}^{k+\ell} 2^{-q j} \bigl|\partial^j_t \psi_e(x,y,\nu_e)\bigr|.
\end{equation}
 where $N := pn - qm$. This is clearly the case if $k+\ell \le m_0$. However if $m_0 < k + \ell$,
then for any $m_0 < j \le k+\ell$, we have 
$$
2^{-q j} |\partial^j_t \psi_e(x,y,\nu_e)| \ \lesssim \ B_0^{-1} 2^{-q m_0} |\partial^{m_0}_t \psi_e(x,y,\nu_e)|
$$
by \eqref{derivative-hyp-B} with $j=0$. 

Since $k\ge 1$ and $n \le n_0$, we see that
\begin{equation}\label{m/n}
2^{-pk-q\ell} \ = \ 2^{-\frac{k}{n} (pn - qm)} 2^{-q(\ell + km/n)} \ \le \ 2^{-\delta N} 2^{-q(\ell + km/n)}
\end{equation}
and so
\begin{equation}\label{nlessm}
2^{-pk-q\ell} \ \le \ 2^{-\delta N} 2^{-q(\ell + k)} \ \ \ {\rm if} \ \ \  n \ \le \ m.
\end{equation}
Therefore when $n \le m$, we have $2^{-pk - q\ell} \le 2^{-\delta N} 2^{-q j}$ for all $m\le j \le k+\ell$
and this shows that \eqref{step-2} and hence \eqref{Main-est} holds.

We may therefore assume that $m<n$. In particular we have $\max(m,n) = n$. Our aim now is to show
that when $m<n$, certain coefficients $Q^{k,\ell}_j(x,y,\nu_e)$ necessarily vanish which then gives us some
hope in establishing \eqref{step-2}. 

Let $\mu$ be the largest integer such that $2^{\mu} m < n$. 
Since $m<n$, we see that $\mu \ge 0$ and $n \le 2^{\mu +1} m$.
We apply Corollary \ref{cor-key-prop} with this $\mu$ to conclude that 
$Q^{k,\ell}_j (x,y,\nu_e) = 0$ for every $j > \ell + 2^{-\mu -1} k$. 
Therefore we see  that several terms on the left hand side of \eqref{step-2} vanish and so 
we can update \eqref{step-2} and  reduce matters to showing
$$
\sum_{j=m}^{\ell + 2^{-\mu -1}k}\!\! 2^{-pk - q\ell} \bigl|Q^{k,\ell}_j \partial^j_t \psi_e(x,y,\nu_e)\bigr| \lesssim 
2^{-\delta N} \sum_{j=m}^{\ell + 2^{-\mu-1} k} 2^{-q j} \bigl|\partial^j_t \psi_e(x,y,\nu_e)\bigr|
$$
instead.
It therefore suffices to establish that $2^{-pk-q\ell} \le 2^{-\delta N} 2^{-qj}$ for all $m\le j \le \ell + 2^{-\mu -1} k$
and this would follow if we could improve \eqref{nlessm} to
$$
2^{-pk-q\ell} \ \le \ 2^{-\delta N} 2^{-q(\ell + 2^{-\mu -1}k)}.
$$
From \eqref{m/n}, we see that this is true since $n\le 2^{\mu + 1} m$.
\end{proof}

By \eqref{Dpq} and Proposition \ref{main-prop}, there is an absolute, uniform constant $C$ such that
\begin{equation}\label{Dpq-II}
|D_{p,q}| \ \le \ C |\lambda \lambda_e| 2^{-\delta N} \sum_{j=m}^{m_0} 2^{-qj} \bigl|\partial^j_t \psi_e(x,y,\nu_e)\bigr|  
\end{equation}
where we recall $N = pn - qm$. Let us denote the sum on the right hand side by ${\mathcal S}_q = {\mathcal S}_q(x,y,\nu_e)$; 
that is,
${\mathcal S}_q := \sum_{j=m}^{m_0} 2^{-qj} |\partial^j_t \psi_e(x,y,\nu_e)|$. Our goal now is to seek
a complimentary bound to \eqref{Dpq-II}; namely, we will show that
\begin{equation}\label{Ipq-IIpq}
|I_{p,q}|, \, |II_{p,q}| \ \le \ C \bigl(|\lambda \lambda_e| {\mathcal S}_q \bigr)^{-\epsilon}
\end{equation}
for some absolute, uniform constant $C$ and $\epsilon>0$.
If we can do this, then by taking a convex combination of the bounds in \eqref{Dpq-II} and \eqref{Ipq-IIpq}
we can conclude that
\begin{equation}\label{Dpq-III}
|D_{p,q} | \ \le \ C 2^{-\delta' N} \min( |\lambda \lambda_e| {\mathcal S}_q, (|\lambda \lambda_e| {\mathcal S}_q)^{-1} )^{\epsilon'}
\end{equation}
for some $\delta', \epsilon' > 0$. Hence
$$
|{\mathcal I}_1 - {\mathcal I}_2| \ \le \ \sum_{(p,q) \in {\mathcal  R}_1} |D_{p,q}| \ \le \ C \,
\sum_{N\ge 0} 2^{-\delta' N} \sum_{pn - qm = N} 
\min(|\lambda \lambda_e| {\mathcal S}_q, (|\lambda \lambda_e| {\mathcal S}_q)^{-1} )^{\epsilon'}.
$$
For fixed $N\ge 0$, the inner sum above is easily seen to be uniformly bounded. In fact, the sum is a sum only
over integers $q$ since $p$ is uniquely determined by $q$ once $N$ is fixed. Splitting this inner sum in $q$ further,
depending on which term $2^{-qj} |\partial^j_t \psi_e(x,y,\nu_e)|\}$ in the sum ${\mathcal S}_q$  is largest, and using the
complimentary estimates mainifest in \eqref{Dpq-III}, we see that the inner sum 
is uniformly bounded. This implies that ${\mathcal I}_1  =  {\mathcal I}_2
+ O(1)$ as was to be shown.

It remains to show \eqref{Ipq-IIpq}. We first concentrate on establishing \eqref{Ipq-IIpq} for $I_{p,q}$,
$$
I_{p,q} = \int\!\!\!\int  e^{2\pi i [\lambda \varphi(s, t) + \rho s + 
\eta t]} \zeta(2^p s) \zeta(2^q t) \frac{ds}{s} \frac{dt}{t}  =: 
\int\!\!\!\int e^{2\pi i \Phi(2^{-p}s, 2^{-q}t)} \zeta(s)\zeta(t) \frac{ds}{s} \frac{dt}{t}.
$$
Here $\Phi(2^{-p}s,2^{-q} t) =  \lambda \varphi(2^{-p}s,2^{-q}t) + \rho 2^{-p}s + \eta 2^{-q} t$ and
%We will think of the phase of $I_{p,q}$ as a function of  $t$ and denote it by $\Phi(t) =$ 
$$
\varphi(2^{-p}s, 2^{-q} t) =  P_{x,y,\omega}(2^{-p}s,2^{-q} t) + \psi(x+ 2^{-p}s,y,\omega) +
\psi(x,y+2^{-q}t,\omega) 
$$
We will think of $\Phi(2^{-p}s, 2^{-q}t)$ as a function of $t$, denoting it by $\Phi(t)$.
By \eqref{P=P*} and \eqref{psi=psie}, we have
$$
|\Phi^{(m_0)}(t)| \ = \ 2^{-qm_0} \bigl| (\lambda \lambda_e) \partial^{m_0}_t \psi_e (x, y+2^{-q} t, \nu_e) \bigr| 
\ \gtrsim \ 2^{-q m_0} |\lambda \lambda_e|
$$
by  \eqref{derivative-hyp-B}. Therefore if $2^{-q m_0} \ge \delta \sum_{j=m}^{m_0 - 1} 2^{-q j} |\partial^j_t \psi_e(x,y,\nu_e)|$ for some $\delta > 0$, then an application of van der Corput's lemma establishes  \eqref{Ipq-IIpq}.
Suppose instead that
\begin{equation}\label{d-small}
2^{-q m_0} \ \le \ \delta \sum_{j=m}^{m_0 -1 } 2^{-q j} \bigl|\partial^j_t \psi_e (x,y,\nu_e)\bigr| \ \ {\rm for \ some \ small} \ \
\delta > 0
\end{equation}
which we will choose small enough later. 
We write 
$$
\psi(x,y+2^{-q} t, \omega) = \sum_{\ell=m}^{m_0-1} \frac{1}{\ell!} \partial^{\ell}_t \psi(x,y,\omega) (2^{-q} t)^{\ell}
\ + \  E_q(t) \ =: \ Q_{x,y,\omega} (t) \ + \ E_q(t)
$$
where for any $2\le m' \le m_0 -1$, $|E^{(m')}_q(t)| \lesssim |\lambda_e| 2^{-q m_0}$. A simple equivalence of norms argument
shows there exists a $\delta_{m_0} > 0$ and an $m \le m' \le m_0 -1$ such that 
$$
|Q_{x,y,\omega}^{(m')}(t)| \ \ge \ \delta_{m_0} |\lambda_e| \ \sum_{j=m}^{m_0 -1} 2^{-q j} 
\bigl| \partial^j_t \psi_e(x,y,\nu_e)\bigr|
$$
where we continue to employ \eqref{psi=psie}. Choosing $\delta>0$ small enough in \eqref{d-small}, we have
$$
\bigl| \partial^{m'}_t \psi(x,y+t,\omega) \bigr| \ \ge \ (\delta_{m_0}/2) |\lambda_e| \ \sum_{j=m}^{m_0 -1}
2^{-q j} \bigl| \partial^j_t \psi_e(x,y,\nu_e)\bigr| \ \ge \ (\delta_{m_0}/4) |\lambda_e| \, {\mathcal S}_q.
$$
By \eqref{P=P*}, we have
$$
P_{x,y,\omega}(2^{-p} s, 2^{-q}t) \ = \ \lambda_e\!\!\!\!\!\!\!\!\!\! \sum_{\begin{array}{c}\scriptstyle
         k, \ell \ge 1 \\
         \vspace{-5pt}\scriptstyle k m_0  + \ell n_0 < m_0 n_0
         \end{array}} \!\!\!\!\!\!\!\!\!\!\!\!\frac{1}{k!\ell!}\,  \partial^{k,\ell}\psi_e(x,y,\nu_e(\omega)) \,  
(2^{-p} s)^k (2^{-q} t)^{\ell}
$$
and if we think of this as a polynomial in $t$, denoting it  by ${\mathcal P}(t)$, we see that
$$
\bigl| {\mathcal P}^{(m')}(t)\bigr| \ \le \ |\lambda_e|\!\!\!\!\!\!\!\!\!\!\! \sum_{\begin{array}{c}\scriptstyle
         k, \ell \ge 1 \\
         \vspace{-5pt}\scriptstyle k m_0  + \ell n_0 < m_0 n_0
         \end{array}} \!\!\!\!\!\!\!\!\!\!\!\!\frac{1}{k!\ell!}\,  2^{-pk -q\ell} 
\bigl|\partial^{k,\ell}\psi_e(x,y,\nu_e(\omega))\bigr| \ \le \ 2^{-\delta N} |\lambda_e| {\mathcal S}_q
$$
by Proposition \ref{main-prop}. Since $N\ge N_0$, we now choose $N_0$ large enough so that the 
bound 
$$
|\Phi^{(m')}(t)|  \ \ge  \ |\lambda| \bigl( |\partial^{m'}_t \psi(x,y+t,\omega)| \ -  |{\mathcal P}^{(m')}(t)| \bigr) \ \ge \ 
(\delta_{m_0}/8) |\lambda \lambda_e| {\mathcal S}_q
$$ 
holds. Another application of van der Corput's lemma establishes \eqref{Ipq-IIpq}. 

When we turn to $II_{p,q}$, we see that the only difference is that we replace $\Phi(t)$ with 
$$
\Psi(t) = \lambda \bigl(\psi(x+ 2^{-p}s,y,\omega) +
\psi(x,y+2^{-q}t,\omega) \bigr) + \rho 2^{-p} s + \eta 2^{-q} t,
$$
the difference being that 
the polynomial $P_{x,y,\omega}(s,t)$ is no longer present. So the
first half of the argument above shows that there is an $m\le m' \le m_0$ such that
$|\Psi^{(m')}(t)| \gtrsim |\lambda \lambda_e| {\mathcal S}_q$ and so van der Corput's lemma
establishes \eqref{Ipq-IIpq} in this case as well. This establishes \eqref{Ipq-IIpq} in all cases,
completing the proof of Theorem \ref{main} when the factorisation hypothesis (FH) holds.

\subsection{Proof of Proposition \ref{C}} Here we return to the proof of Proposition \ref{C} where we now establish
the uniform boundedness of 
$$
{\mathcal O}  := \ \int\!\!\!\int_{R_2} 
e^{i [\lambda \varphi(s,t) + \rho s + \eta t]}\  \frac{ds}{s} \frac{dt}{t}.
$$
Here $\varphi(s,t) = \varphi_{x,y,\omega}(s,t) = P_{x,y,\omega}(s,t) + \psi(x+s,y,\omega) + \psi(x,y+t,\omega)$
and $R_2  = \{(s,t)\in R_0: C^{-1} |t|^m < |s|^n < |t|^m\}$. In order to isolate the main one
dimensional oscillatory integral in ${\mathcal O}$ effectively, we will first reduce matters to the following
oscillatory integral with polynomial phase ${\mathcal P}(s,t)$,
$$
{\mathcal T} \ = \ \int\!\!\int_{R_2} e^{2\pi i [\lambda {\mathcal P}(s,t) + \rho s + \eta t]} \frac{ds}{s} 
\frac{dt}{t},
$$
showing that ${\mathcal O} = {\mathcal T} + O(1)$. Here 
$$
{\mathcal P}(s,t) \ = \ P_{x,y,\omega}(s,t) + \sum_{k=n}^{n_0 -1} \frac{1}{k!} \partial^k_s \psi(x,y,\omega) s^k
+ \sum_{\ell=m}^{m_0 -1} \frac{1}{\ell!}\partial^{\ell}_t \psi(x,y,\omega) t^{\ell}
$$
$$
= \ 
\sum_{\begin{array}{c}\scriptstyle
         k, \ell \ge 0 \\
         \vspace{-5pt}\scriptstyle k m_0  + \ell n_0 < m_0 n_0
         \end{array}} \!\!\!\!\!\!\!\!\!\!\!\!\frac{1}{k!\ell!}\,  \partial^{k,\ell}\psi(x,y,\omega) \,  
s^k t^{\ell}.
$$

We will accomplish this in two steps. We first compare ${\mathcal O}$ with
$$
{\mathcal T}_1 \ = \ \int\!\!\int_{R_2} e^{2\pi i [\lambda {\mathcal P}_1(s,t) + \rho s + \eta t]} \frac{ds}{s} 
\frac{dt}{t}
$$
where 
$$
{\mathcal P}_1(s,t) \ = \ P_{x,y,\omega}(s,t) + \sum_{k=n}^{n_0 -1} \frac{1}{k!} \partial^k_s \psi(x,y,\omega) s^k
+ \psi(x,y+t,\omega),
$$
and show ${\mathcal O} = {\mathcal T}_1 + O(1)$.

As with ${\mathcal I}_1$ we decompose these
integrals dyadically, writing ${\mathcal O} = \sum_{(p,q) \in {\mathcal R}_2} {\mathcal O}_{p,q}$
where 
$$
{\mathcal O}_{p,q} \ = \ \int\!\!\int e^{2\pi i [\lambda \varphi(2^{-p}s,2^{-q}t) + \rho 2^{-p}s + 
\eta 2^{-q} t]} \zeta(s) \zeta(t) \frac{ds}{s}
\frac{dt}{t}.
$$
and ${\mathcal T}_1 = \sum_{(p,q) \in {\mathcal R}_2} {\mathcal T}^1_{p,q}$ where
$$
{\mathcal T}^1_{p,q} \ = \ \int\!\!\int e^{2\pi i [\lambda {\mathcal P}_1(2^{-p}s,2^{-q}t) + \rho 2^{-p}s + 
\eta 2^{-q} t]} \zeta(s) \zeta(t) \frac{ds}{s}
\frac{dt}{t}.
$$
Here ${\mathcal R}_2 = \{(p,q) \in {\mathcal R}: 0 \le pn - qm \le N_0 \}$. Again we establish various estimates on the 
difference $D_{p,q} = {\mathcal O}_{p,q} - {\mathcal T}^1_{p,q}$ but our task is much easier now 
since the sum over $(p,q) \in {\mathcal R}_2$
is effectively a sum over $p$ alone since for each $0\le N \le N_0$, the integer $q$ is uniquely determined by $p$
once $N$ is fixed and there are only boundedly many $N$ to consider. 

Since the difference of phases appearing
in ${\mathcal O}_{p,q}$ and ${\mathcal T}^1_{p,q}$ is 
$$
2\pi i \lambda \Bigl[ \psi(x+2^{-p}s,y,\omega) \ - \ \sum_{k=0}^{n_0-1} \frac{1}{k!} \partial^k_s 
\psi(x,y,\omega) (2^{-p}s)^k\Bigr] \ = \
O(|\lambda| 2^{-pn_0}),
$$
we have $|D_{p,q}| \le  C |\lambda| 2^{-p n_0}$. But for both phases in ${\mathcal O}_{p,q}$
and ${\mathcal T}^1_{p,q}$ the $n_0$th derivative with respect to $s$ is bounded below by $|\lambda| 2^{-p n_0}$.
Hence
van der Corput's lemma gives us the complimentary estimate $|D_{p,q}| \le C (|\lambda| 2^{-p n_0})^{-1/n_0}$
and so for each $0\le N \le N_0$,
$$
\sum_{p: pn - qm = N} |D_{p,q}| \ \le \ C \sum_{p: pn - qm = N} \min(|\lambda| 2^{-p n_0}, (|\lambda| 2^{-p n_0})^{-1/n_0}) 
\ \le \ C_N.
$$ 
This shows that ${\mathcal O} = {\mathcal T}_1 + O(1)$. 

We now compare ${\mathcal T}_1$ to ${\mathcal T}$ and show ${\mathcal T}_1 = {\mathcal T} + O(1)$. As with
${\mathcal O}$ and ${\mathcal T}_1$, we decompose ${\mathcal T} = \sum_{(p,q) \in {\mathcal R}_2} {\mathcal T}_{p,q}$
dyadically where
$$
{\mathcal T}_{p,q} \ = \ \int\!\!\int e^{2\pi i [\lambda {\mathcal P}(2^{-p}s,2^{-q}t) + \rho 2^{-p}s + 
\eta 2^{-q} t]} \zeta(s) \zeta(t) \frac{ds}{s}
\frac{dt}{t}.
$$
In a similar way as above, we have the following bound for the difference $D_{p,q} = {\mathcal T}^1_{p,q}
- {\mathcal T}_{p,q}$;
$$
|D_{p,q}| \ \le \ C \min(|\lambda \lambda_e| 2^{-q m_0}, (|\lambda \lambda_e| 2^{-q m_0})^{-1/m_0})
$$
and so for each $0 \le N \le N_0$,
$$
\sum_{p: pn - qm = N} |D_{p,q}| \ \le \ C \sum_{q: pn - qm = N} 
\min(|\lambda\lambda_e| 2^{-q m_0}, (|\lambda\lambda_e| 2^{-q m_0})^{-1/m_0}) 
\ \le \ C_N.
$$ 
This shows that ${\mathcal T}_1 = {\mathcal T} + O(1)$ and hence ${\mathcal O} = {\mathcal T} + O(1)$,
as claimed.

We make the change of variables $s\to s/|t|^{m/n}$ in the $s$ integral so that
$$
{\mathcal T} = \int_{|t|<\kappa} e^{2\pi i \eta t} \frac{1}{t} \Bigl[\int_{T(t)} 
e^{2\pi i [\lambda ({\mathcal P}(|t|^{m/n} s, t)) + \rho |t|^{m/n} s]} 
\frac{ds}{s} \Bigr] dt
$$
where $T(t) = \{ C^{-1}<|s|<1, |s|^{n_0 n} \le |t|^{m_0 n - m n_0}\}$. We now interchange
the order of integration so that
$$
{\mathcal T} \ = \ \int_{C^{-1}<|s| < 1} \Bigl[ \int_{T} 
e^{2\pi i [\lambda ({\mathcal P}(|t|^{m/n} s, t)) + \rho |t|^{m/n} s + \eta t]} 
\frac{dt}{t} \Bigr] \frac{ds}{s}
$$
where $T = \{|t|<\kappa: a<|t|<b\}$ for some $a,b$, depending on $s$. It is possible, after a few
more changes of variables, to reduce the inner integral to one with a purely polynomial phase and then we can appeal
to the result of Stein and Wainger \eqref{SW}. However in their original paper \cite{SW},
Stein and Wainger in fact prove an oscillatory integral estimate for phases with rational powers
and in particular, their result  implies that the inner integral above is uniformly bounded in the coefficients of 
${\mathcal P}$,  $\lambda, \rho, \eta$ and $a,b$. The bound only depends on the degree of ${\mathcal P}$
which in turn depends only on $n_0$ and $m_0$. Hence ${\mathcal T} = O(1)$ and this in turn  implies
${\mathcal O} = O(1)$, completing the proof of Proposition \ref{C}.

\section{The proof of Theorem \ref{main} -- when (FH) fails}\label{main-fail}

We begin with a simple lemma about analytic functions whose proof is an application
of the Weierstrass Preparation Theorem. We only present a special case which serves
our purposes. The general case and the details of the proof can be found in \cite{O}.

\begin{lemma}\label{OW} Let $M$ be a real-analytic manifold and let $\psi$ be a real-analytic
function on ${\mathbb T}^2 \times M$. Then the factorisation hypothesis (FH) fails for $\psi$ on $M$ if and only if
there is a point $(x,y,\omega) \in {\mathbb T}^2 \times M$ such that $\psi_{st}(x,y,\omega) \not=0$
and either $\psi_{ss}(x,y,\omega) = 0$ or $\psi_{tt}(x,y,\omega) = 0$ (or both).
\end{lemma}

We now turn to the proof of Theorem \ref{main} when the factorisation hypothesis (FH) fails.
In this case it suffices to show that the norms $\|e^{2\pi i {\overline{n}} \cdot {\overline{g}}} \|_{U_{rect}({\mathbb T}^2)}$
are unbounded in ${\overline{n}} \in {\mathbb Z}^d$ where ${\overline{g}}(s,t) = {\overline{\phi}}(s,t) + {\vec{L}}_1 s
+ {\vec{L}}_2 t$ parametrises our map $\Phi : {\mathbb T}^2 \to {\mathbb T}^d$. Here
${\overline{\phi}} = (\phi_1, \ldots, \phi_d)$ is a $d$-tuple of real-analytic, periodic functions and 
${\overline{L}} := ({\vec{L}}_1, {\vec{L}}_2)$
is a pair of lattice points in ${\mathbb Z}^d$. 

Recall that
$$
\|e^{2\pi i {\overline{n}}\cdot {\overline{g}}} \|_{U_{rect}({\mathbb T}^2)} \ = \
\sup_{M,N} \| S_{M,N} (e^{2\pi i {\overline{n}} \cdot {\overline{g}}})\|_{L^{\infty}({\mathbb T})}
$$
$$
= \
\sup_{M,N,x,y} \ \Bigl| \int\!\!\!\int_{{\mathbb T}^2} e^{2\pi i {\overline{n}} 
\cdot [{\overline{\phi}}(x + s, y + t) + {\overline{L}}\cdot (s,t)]} D_M(s) D_N(t) \, ds dt\Bigr|
$$
where $D_M$ denotes the Dirichlet kernel of order $M$. Here 
${\overline{L}} \cdot (s,t) = {\vec{L}}_1 s + {\vec{L}}_2 t$.

As before we set $\psi(x,t,\omega) = \omega \cdot {\overline{\phi}}(s,t)$ which defines a real-analytic map
on ${\mathbb T}^2 \times {\mathbb S}^{d-1}$. If $\psi$ does not satsify (FH) on
${\mathbb S}^{d-1}$, then Lemma \ref{OW} implies that there is a point $(x_0,y_0, \omega_0) \in 
{\mathbb T}^2 \times {\mathbb S}^{d-1}$ such that (without loss of generality)
\begin{equation}\label{FH-fail}
\psi_{st}(x_0,y_0,\omega_0) \ \not= \ 0 \ \ \ {\rm and} \ \ \ \psi_{tt}(x_0,y_0,\omega_0) = 0.
\end{equation}

\subsection{An intuitive approach which does not quite work}
It seems  natural now to consider the above oscillatory integral 
$$
 \int\!\!\!\int_{{\mathbb T}^2} e^{2\pi i {\overline{n}} 
\cdot [{\overline{\phi}}(x_0 + s, y_0 + t) + {\overline{L}}\cdot (s,t)]} D_M(s) D_N(t) \, ds dt
$$
evalutated for $x=x_0, \, y=y_0$ and for ${\overline{n}} \in {\mathbb Z}^d$ with 
${\overline{n}} 
= \|{\overline{n}}\| \, \omega$ where $\omega$ is as close as possible to $\omega_0$. In general
$\omega_0 \not= {\overline{n}}/\|{\overline{n}}\|$ for any ${\overline{n}}\in {\mathbb Z}^d$ but
for this discussion, let us simply matters and assume we can take $\omega = \omega_0$.

With strategic
choices of $\lambda := \|{\overline{n}}\|$, $M$ and $N$, we would then like to show that the above
integral is unbounded as $\lambda \to \infty$. Using the property \eqref{Dirichlet} of  Dirichlet kernels,
together with Proposition \ref{osc-est-III}, we see that
$$
 \int\!\!\!\int_{{\mathbb T}^2} e^{2\pi i {\overline{n}} 
\cdot [{\overline{\phi}}(x_0 + s, y_0 + t) + {\overline{L}}\cdot (s,t)]} D_M(s) D_N(t) \, ds dt  \  = \
$$
$$
 \int\!\!\!\int_{|s|, |t|\le 1/2} e^{2\pi i \lambda [
\psi(x_0 + s, y_0 + t) + \omega_0 \cdot {\overline{L}}(s,t)]} \sin(Ms) \sin(Nt) \, \frac{ds}{s} \frac{dt}{t} \ + \ 
O_{\omega_0}(1)
$$
where $\omega_0 \cdot {\overline{L}}(s,t) := \omega_0 \cdot {\vec{L}}_1 \, s + \omega_0 \cdot {\vec{L}}_2 \, t$.

Let us denote by $S$ the integral above with the double Hilbert transform singularity $1/st$. We scale $S$
by making the change of variables $s \to \delta^3 s$ and $t \to \delta^2 t$. Hence
$$
S \ = \  \int_{|s|\le \delta^{-3}}\!\!\int_{|t|\le\delta^{-2}} e^{2\pi i \lambda [
\psi_{\delta}(s, t) + \omega_0 \cdot {\overline{L}} (\delta^3 s,\delta^2 t)]} \sin(M\delta^3 s) \sin(N\delta^2 t) \, \frac{ds}{s} \frac{dt}{t} 
$$
where 
$$
\psi_{\delta} (s,t) := \psi(x_0 + \delta^3 s, y_0 + \delta^2 t, \omega_0) = \psi(x_0, y_0,\omega_0) + 
\nabla \psi(x_0,y_0, \omega_0) \cdot (\delta^3 s, \delta^2 t) \ +
$$
$$
\delta^5 \bigl[ \psi_{st}(x_0,y_0.\omega_0) s t + \delta \psi_{ss}(x_0,y_0,\omega_0) s^2/2 + 
\sum_{k+\ell\ge 3} \frac{\delta^{3k+2\ell - 5}}{k!\ell!} \partial^{k,\ell}\psi(x_0,y_0,\omega_0) s^k t^{\ell} \bigr].
$$
The gradient $\nabla = \nabla_{s,t}$ is taken with respect to the variables $(s,t)$ on ${\mathbb T}^2$.

This scaling is suggested by the Newton diagram of $\psi(x,y,\omega_0)$ based at the point $(x_0, y_0)$
whose essential feature is the content of \eqref{FH-fail}.  By choosing $\lambda = \delta^{-5}$,
we see that 
$$
e^{2\pi i \lambda [\psi_{\delta}(s,t)  + \omega_0 \cdot {\overline{L}}(\delta^3 s,\delta^2 t)] } \ = \ C \,
 e^{2\pi i [ c s t + {\vec v}\cdot (\delta^{-2} s, \delta^{-3} t) + O(\delta)]}
$$
for some $|C| = 1$ and $c \not= 0$. Here ${\vec v} = {\vec{v}}(x_0,y_0,\omega_0)  =  \omega_0 \cdot {\overline{L}} + 
\nabla\psi(x_0,y_0,\omega_0)$ (recall $\omega_0 \cdot {\overline{L}} = (\omega_0 \cdot {\vec{L}}_1,
\omega_0\cdot {\vec{L}}_2)$).

Now if we choose
$(M,N)$ so that 
$$
{\vec v}(x_0,y_0,\omega_0) \cdot (\delta^{-2}, \delta^{-3}) \ \  
\sim \ \  (M\delta^3, N\delta^2)
$$ 
and take $\lambda$ large (or equivalently, $\delta$ small), 
we might hope that all the terms in the Taylor expansion of $(s,t) \to \psi(s,t,\omega_0)$ at $(x_0,y_0)$
of order 2 and larger become less and less significant compared to the 
the hyperbolic oscillation $\psi_{st}(x_0,y_0,\omega_0) s t$, putting us in an analogous situation
of \eqref{CF} arising in C. Fefferman's work \cite{CF}
 so that we can deduce that the oscillatory integral blows up like $\log \lambda = \log(1/\delta)$. 
Such scaling arguments were used in \cite{CWW} and with
some effort all this can be carried out here {\it IF} the vector
$$
{\vec{v}}(x_0,y_0,\omega_0)  \ = \ \omega_0 \cdot {\overline{L}} + 
\nabla\psi(x_0,y_0,\omega_0) \ = \ \omega_0 \cdot \bigl[\nabla {\overline{\phi}}(x_0, y_0) + {\overline{L}}\bigr]
$$
has the property that both its components are nonzero. 

Unfortunately it may happen that ${\vec{v}}$ vanishes. For example, consider the case $d=1$
and $\phi(s,t) = \sin(2\pi s) \cos(2\pi t)$ so that $\psi = \phi$. At the point $(x_0,y_0) = (0, 1/4)$, we have
$\phi_{tt}(0,1/4) = 0$ yet $\phi_{st}(0,1/4) = - 4 \pi^2$ and so the factorisation hypothesis (FH) fails
for this example. We take ${\overline{L}} = (L_1, L_2) = 0$ and so
in this case, ${\vec{v}} = \nabla \phi(0, 1/4) = 0$. On the other hand, it is not too difficult to show that
$$
\sup_{n,M,N} \ \Bigl| \int\!\!\!\int_{|s|, |t|\le 1/2} e^{2\pi i n [\sin(2\pi s)\cos(2\pi t + \pi/2) ]} \sin(Ms) \sin(Nt) \, \frac{ds}{s} \frac{dt}{t}\Bigr|
\ <\ \infty.
$$

\subsection{A rigorous approach}
Therefore we need to proceed in a slightly different way. We will perturb the point 
$(x_0,y_0) \in {\mathbb T}^2$ in an arbitrarily small
neighbourhood so that  the two components of 
$$
{\vec{v}}(x,y,\omega_0)  \ = \ \omega_0 \cdot \bigl[\nabla {\overline{\phi}}(x, y) + {\overline{L}}\bigr]
$$
are nonzero, 
$\psi_{st}(x,y,\omega_0) \not=0$ and $0<|\psi_{tt}(x,y,\omega_0)| \le \varepsilon$ is as small as we like. 
We will then show that there exists ${\overline{n}} \in {\mathbb Z}^d$ and choices for $M$ and $N$
such that the oscillatory 
$$
 \int\!\!\!\int_{{\mathbb T}^2} e^{2\pi i {\overline{n}} 
\cdot [{\overline{\phi}}(x + s, y + t) + {\overline{L}}\cdot (s,t)]} \sin(Ms) \sin(Nt) \, \frac{ds}{s} \frac{dt}{t}
$$
is bounded below by $\log(1/\varepsilon)$. This will establish that the mapping property $(\Phi)_{rect}$ fails
for $\Phi = \Phi_{{\overline{\phi}}, {\overline{L}}}$ as well as exhibit an inherent discontinuity in these
oscillatory integrals we are studying.

First we observe  that $\psi_{tt}(x,y,\omega_0)$ is not identically zero; otherwise
$\partial_t \psi(x,y,\omega_0) \equiv \partial_t \psi(x,y_0,\omega_0)$ which in turn implies
that $\psi_{st}(x_0,y_0,\omega_0) = 0$, contradicting \eqref{FH-fail}. Similarly $\psi_{ss}(x,y,\omega_0)$
is not identically zero.
Hence
the zero set 
\begin{equation}\label{zero}
Z \ =  \ \{(x,y) \in {\mathbb T}^2  : 
\psi_{tt}(x,y,\omega_0) = 0  \ {\rm or} \ \psi_{ss}(x,y,\omega_0) = 0\} 
\end{equation} 
is a set of measure zero
and in particular,
we can find points $(x,y)$ arbitrarily close to $(x_0,y_0)$ such that both 
$|\psi_{tt}(x,y,\omega_0)|, |\psi_{ss}(x,y,\omega_0)| >0$ and $|\psi_{tt}(x,y,\omega_0)|$
is as small as we like.

In the next lemma we will derive derivative bounds for $\psi$ in a neighbourhood $V$ of $(x_0,y_0)$.

\begin{lemma}\label{m,n>2} There exists a pair $m_0, n_0 \ge 2$ such  that for any sufficiently small neighbourhood $V$ of $(x_0,y_0)$, we have
\begin{equation}\label{m>2} 
A/2 \le \bigl| \partial^{m_0}_t\psi(x,y,\omega_0) \bigr| \le 2 A \ \ \ {\rm and} \ \ \ 
B/2 \le \bigl| \partial^{n_0}_s \psi(x,y,\omega_0) \bigr| \le  2 B
\end{equation}
for all $(x,y) \in V$.\footnote{In fact these bounds hold for all $(z,w) \in W$ in a complex neighbourhood 
$W \subset {\mathbb C}^2 = {\mathbb R}^2 + i {\mathbb R}^2$ of $(x_0, y_0) + i (0,0)$. This remark
will allow us to use standard Cauchy estimates on the derivatives of $\partial^{m_0}_t \psi$ and $\partial^{n_0}_s \psi$.} 
 Here $A = |\partial^{m_0}_t \psi(x_0,y_0,\omega_0)|$ or 
$A = | \partial^{1,m_0} \psi(x_0,y_0,\omega_0)| \, |x-x_0|$ depending (respectively) on whether or not
there is an $m\ge 3$ such that
$\partial^{m}_t \psi(x_0,y_0,\omega_0) \not= 0$.
Similarly $B =  |\partial^{n_0}_s \psi(x_0,y_0,\omega_0)|$ or 
$B = | \partial^{n_0, 1} \psi(x_0,y_0,\omega_0)| \, |y-y_0|$.

Furthermore since $\psi_{st}(x_0,y_0,\omega_0) \not= 0$, we can choose $V$ 
so that for all $(x,y) \in V$,
\begin{equation}\label{n>2}
D/2 \le \bigl|\psi_{st}(x,y,\omega_0) \bigr| \le  2 D \ \ {\rm where} \ \ D :=  \bigl| \psi_{st}(x_0,y_0,\omega_0)\bigr| .
\end{equation}

\end{lemma}

\begin{proof} We consider two cases. First suppose there exists an $m_0\ge 3$ such that
$\partial^{m_0}_t \psi(x_0,y_0,\omega_0) \not= 0$. In this case we set $A = |\partial^{m_0}_t \psi(x_0,y_0,\omega_0)|>0$
and choose by continuity an open neighbourhood $V$ of $(x_0,y_0)$ 
such that $ A/2 \le | \partial^{m_0}_t\psi(x,y,\omega_0) \bigr| \le  2 A$ for all 
$(x,y) \in V$.

 Next suppose
that $\partial^m_t \psi(x_0,y_0,\omega_0) = 0$ for all $m\ge 2$. Then 
$$
\psi(x_0,y,\omega_0) \ = \ \psi(x_0,y_0,\omega_0)  + \partial_t \psi(x_0,y_0,\omega_0) (y - y_0) 
$$
for all $y\in {\mathbb T}$ by analyticity. This forces $\partial_t \psi(x_0,y_0,\omega_0) = 0$ since
$y \to \psi(x_0,y,\omega_0)$ is a periodic function. Hence $\psi(x_0,y,\omega_0) \equiv constant$
and so $\partial^m_t \psi(x_0,y,\omega_0) \equiv 0$ for all $m\ge 1$. Thus for every $m\ge 1$,
$$
\partial^m_t \psi(x,y,\omega_0) \ = \ \int_{x_0}^x \partial^{1,m} \psi(u, y, \omega_0) \, du 
$$
and we claim there exists an $m_0 \ge 2$ such that $\partial^{1,m}\psi(x_0,y_0,\omega_0) \not= 0$. If
this is not the case, then
$$
\partial_s \psi(x_0,y,\omega_0) \ = \ \partial_s \psi(x_0,y_0,\omega_0) + \partial^{1,1}\psi(x_0,y_0,\omega_0)(y-y_0)
$$
for all $y\in {\mathbb T}$ by analyticity. But $\psi_{st}(x_0,y_0,\omega_0) = \partial^{1,1} \psi(x_0,y_0,\omega_0) \not= 0$
which implies that the right hand side is a nonconstant, nonperiodic function of $y$. This contradicts the 
periodicity in $y$ of the left hand side. Thus there exists an $m_0\ge 2$ such that 
$\partial^{1,m}\psi(x_0,y_0,\omega_0) \not= 0$ and hence a neighbourhood $V$ of $(x_0, y_0)$ such that
for all $(x,y) \in V$,
$$
\frac{1}{2}\bigl|\partial^{1,m_0} \psi(x_0,y_0,\omega_0)| \, |x-x_0|  \ \le \ 
|\partial^{m_0}_t \psi(x,y, \omega_0)| \ \le  \  2 \bigl|\partial^{1,m_0} \psi(x_0,y_0,\omega_0)| \, |x-x_0| .
$$
Therefore the first part of \eqref{m>2} holds with $A = |\partial^{1,m_0} \psi(x_0,y_0,\omega_0)| \, |x-x_0| $.

A similar argument shows the existence of an $n_0\ge 2$ such that the second part of \eqref{m>2}
holds.
\end{proof}

For convenience, we write the neighbourhood $V$ of $(x_0,y_0)$
as a ball $B_{r_0}(x_0,y_0)$ so that  Lemma \ref{m,n>2} holds for any sufficiently 
small $r_0 > 0$. We will determine later exactly how small we will take $r_0>0$ but then we
will fix $r_0$ and allow all our estimates to depend on $r_0$.  

Given any $\varepsilon>0$, we can certainly find an $r = r(\varepsilon)>0$
such that $0< r < r_0/2$ and
$|\psi_{tt}(x,y,\omega_0)| \le \varepsilon$ for all $(x,y) \in B_{2r}(x_0,y_0)$.
This simply follows from the fact that $\psi_{tt}(x_0,y_0,\omega_0) = 0$.

Since $B_{2r}(x_0,y_0) \subset B_{r_0}(x_0,y_0) = V$, 
we see that \eqref{m>2} holds with $(x,y)$ replaced by $(x+s,y+t)$ where
$(x,y) \in B_r(x_0,y_0)$ and $|s|,|t| < r_0/2$. Furthermore, 
\eqref{n>2} and $|\psi_{tt}(x,y,\omega_0)| \le \varepsilon$ hold for all $(x,y) \in B_r(x_0, y_0)$. 
Our goal  is to find a point $(x,y) \in B_r(x_0,y_0)$  with a few additional properties. 
But first we need some information about certain zero sets.

\begin{lemma}\label{vector-nonzero} Each component of the analytic function
$$
{\vec{v}}(x,y,\omega_0) \ = \ \nabla\psi(x,y,\omega_0) + \omega_0 \cdot {\overline{L}} \ = \                          
\omega_0 \cdot \bigl[\nabla {\overline{\phi}}(x,y) + {\overline{L}}\bigr] 
$$
is not identically equal to zero and so the zero sets
$$
Z_1 \ :=\  \bigl\{(x,y) \in {\mathbb T}^2 : 
\partial_s \psi(x,y,\omega_0) + \omega_0\cdot{\vec{L}}_1 = 0\bigr\}
$$
and
$$
Z_2 \ :=\ \bigl\{(x,y) \in {\mathbb T}^2 : 
\partial_t \psi(x,y,\omega_0) + \omega_0\cdot{\vec{L}}_2 = 0\bigr\}
$$
are sets of measure zero. 

\end{lemma}

\begin{proof} We treat each component of ${\vec{v}}$ separately. 
If $\partial_t \psi(x,y,\omega_0) \equiv - \omega_0 \cdot {\vec{L}}_2$, then 
$\omega_0 \cdot {\vec{L}}_2 = 0$ since 
derivatives of periodic functions must vanish somewhere. Hence
 $\psi_{st}(x_0,y_0,\omega_0) = 0$
which contradicts \eqref{FH-fail}.
\end{proof}

We are in a position to put our various observations together.

%Our goal now is to find a point $(x,y,\omega) \in {\mathbb T}^2 \times {\mathbb S}^{d-1}$ 
%sufficiently close to $(x_0,y_0,\omega_0)$  with certain properties. 

\begin{lemma}\label{point-properties} Given $\varepsilon>0$. There exist a positive $r = r(\varepsilon)$
with $0<r < r_0/2$ and a point $(x,y) \in B_r(x_0,y_0)$ such that 
for all $|s|, |t| < r_0/2$,
\begin{equation}\label{m>2-again}
A/2 \ \le \ 
|\partial^{m_0}_t \psi(x + s, y+t, \omega_0)| \ \le \ 2 A, \ \ {\rm and} \ \ A >  0,
\end{equation}
\begin{equation}\label{n>2-again}
B/2 \ \le \ 
 |\partial^{n_0}_s \psi(x + s, y+t, \omega_0)| \ \le \  2 B, \ \ {\rm and} \ \ B > 0 
\end{equation}
hold. Also
$D/2 \le |\partial_{st} \psi(x, y, \omega_0)| \le  2 D$ and 
$|\psi_{tt}(x,y,\omega_0)| \le \varepsilon$ hold. Finally, 
$0 <|\psi_{tt}(x,y,\omega_0)|, |\psi_{ss}(x,y,\omega_0)|$ and
both components of the vector
$$
{\vec{v}}(x,y,\omega_0) \ = \ \nabla\psi(x,y,\omega_0) + \omega_0\cdot {\overline{L}} \ = \                          
\omega_0 \cdot \bigl[\nabla {\overline{\phi}}(x,y) + {\overline{L}}\bigr] 
$$
are nonzero.  
%$0 < |\partial_{tt}\psi(x,y,\omega)| \le \varepsilon$. 
\end{lemma}

\begin{proof} We take $r = r(\varepsilon)$ as in the
discussion immediately prior to the statement of Lemma \ref{vector-nonzero}.
Since $B_{2r}(x_0,y_0) \subset B_{r_0}(x_0,y_0)$, we
see that \eqref{m>2-again} and \eqref{n>2-again} hold for all $(x,y) \in B_r(x_0,y_0)$ and $|s|,|t| < r_0/2$.
As long as we choose $(x,y)$ such that $x \not= x_0$ and $y \not= y_0$, then both $A$ and $B$ will
be nonzero.
Furthermore, 
%First we take a small neighbourhood $U$ of $(x_0,y_0,\omega_0)$
%such that $|\psi_{tt}(x',y',\omega')|$  $\le \varepsilon/2$ for all $(x'.y',\omega')\in U$. Ensuring that
%$U$ is sufficiently small,  Lemma \ref{m,n>2} then states that both
%\eqref{m>2} and \eqref{n>2} hold for all points in $U$. 
since the union of the zero sets $Z \cup Z_1 \cup Z_2$ is a set
of measure zero, we see that $B_r(x_0,y_0) \setminus (Z \cup Z_1 \cup Z_2)$
is nonempty. By Lemma \ref{vector-nonzero}, we need only to choose any 
$(x,y) \in B_r(x_0,y_0) \setminus (Z \cup Z_1 \cup Z_2)$.
\end{proof}

%From now on, we fix a point $(x,y,\omega) \in {\mathbb T}^2 \times {\mathbb S}^{d-1}$
%with the properties described in Lemma \ref{point-properties}. 
The next proposition completes the proof
of Theorem \ref{main} when the factorisation hypothesis (FH) fails.

\begin{proposition}\label{logvar} Suppose that (FH) fails on ${\mathbb S}^{d-1}$ for 
$\psi(x,y,\omega) = \omega \cdot {\overline{\phi}}(x,y)$. Then for sufficiently small $\varepsilon>0$, we have
$$
\sup_{{\overline{n}}, M,N, x,y} \, 
 \Bigl| \int\!\!\!\int_{{\mathbb T}^2} e^{2\pi i {\overline{n}} 
\cdot [{\overline{\phi}}(x + s, y + t) + {\overline{L}}\cdot (s,t)]} \sin(Ms) \sin(Nt) \, \frac{ds}{s} \frac{dt}{t}\Bigr| \ \gtrsim \
\log(1/\varepsilon).
$$
\end{proposition}

The proof will be carried out in several stages.

 For given $\varepsilon>0$, we fix an $r = r(\varepsilon) > 0$ and 
$(x,y) = (x_{\varepsilon}, y_{\varepsilon}) \in {\mathbb T}^2$ with the properties
given in Lemma \ref{point-properties}.

With this point choice $(x,y) \in {\mathbb T}^2$, our main task is to establish
\begin{equation}\label{Main-osc-est} 
\sup_{\lambda, M,N} \, 
 \Bigl| \int\!\!\!\int_{{\mathbb T}^2} e^{2\pi i \lambda  \varphi(s,t)} \sin(Ms) \sin(Nt) \, \frac{ds}{s} \frac{dt}{t}\Bigr|  \ \gtrsim \
\log(1/\varepsilon)
\end{equation}
where $\varphi(s,t) = \varphi_{\omega_0,x, y}(s,t)
:=  \omega_0 \cdot [{\overline{\phi}}(x + s, y + t) + {\overline{L}}\cdot (s,t)]$. 

Since $r_0$ is independent of $\varepsilon>0$, we see that by a couple applications of Proposition \ref{osc-est-III}, 
the oscillatory integral in \eqref{Main-osc-est} is equal to
$$
\int\!\!\!\int_{|s|,|t| < r_0/2} 
e^{2\pi i \lambda  \varphi(s,t)} \sin(Ms) \sin(Nt) \, \frac{ds}{s} \frac{dt}{t} \ + \ O_{r_0,\omega_0}(1),
$$
with the important observation that the bound in the $O(1)$ term is independent of $\varepsilon$.
Let us denote by ${\mathcal I}$ the oscillatory integral in the above displayed equation. Hence
\eqref{Main-osc-est} follows from the existence of a $\Lambda(\varepsilon)$ such that
whenever $\lambda > \Lambda(\varepsilon)$, we can find choices for $M$ and $N$ such that 
\begin{equation}\label{MOE-reduction}
|{\mathcal I}| \ \gtrsim \ \log(1/\varepsilon)
\end{equation}
for sufficiently small $\varepsilon>0$. 

Writing $\sin(t) = [e^{it} - e^{-it}]/2i$, we see that
${\mathcal I} = I + II + III + IV$ splits as a sum of four oscillatory integrals of the form
$$
{\mathcal J} \ = \ \int\!\!\!\int_{|s|,|t| < r_0/2} 
e^{2\pi i \lambda  \varphi(s,t) + \rho s + \eta t} \ \frac{ds}{s} \frac{dt}{t}
$$
where $\rho = \pm M$ and $\eta = \pm N$. We  now analyse ${\mathcal J}$, establishing
uniform bounds for various parts of ${\mathcal J}$ and eventually
reducing matters to the heart of ${\mathcal J}$ which
we will then bound from below. It is crucial that all our bounds from above are
independent of $\varepsilon$ and in particular independent of $x = x_{\varepsilon}$ 
and $y = y_{\varepsilon}$. Our bounds also need to be uniform in $\lambda, M$ and $N$
since we will eventually find a large $\Lambda(\varepsilon)$
so that whenever $\lambda > \Lambda(\varepsilon)$, we will find $M$ and $N$, depending on $\lambda$,
such that \eqref{MOE-reduction} holds. 
 Recall that this was the case when we passed from the oscillatory
integral in \eqref{Main-osc-est} to ${\mathcal I}$. Our bounds are allowed to depend
on $x_0, y_0, \omega_0, {\overline{\phi}}$ and $r_0$.

\subsection{Analysis of ${\mathcal J}$: reduction to a quadratic phase}

By our choce of the point $(x,y) \in {\mathbb T}^2$, we have by
Lemma \ref{point-properties} (see \eqref{m>2-again} and \eqref{n>2-again}) that for all $|s|,|t| < r_0/2$,
$$
A/2 \ \le \ |\partial^{m_0}_t \varphi(s,t)| \ \le \  2A, \ \ {\rm and} \ \ 
B/2  \ \le \ |\partial^{n_0}_s \varphi(s,t)| \ \le \ 2B, 
$$
where $A = |\partial^{m_0}_t \varphi(0,0) |$ or $A = |\partial^{1,m_0} \varphi(0,0) (x - x_0)|$.
Similarly for $B$. Note that $A$ and/or $B$ may depend on $x$ or $y$ which in turn depends on 
$\varepsilon$. We are in a position to employ the analysis in Section \ref{prelude-I}
and in particular Section \ref{00} since the required derivative bounds \eqref{derivative-hyp}
and \eqref{derivative-hyp-II} correspond to the above derivative bounds on $\varphi$. Unfortunately
the bounds derived in Section \ref{prelude-I} and Section \ref{00} depend on $A$ and $B$ which
may depend on $\varepsilon$ and 
this is not good for us.

Instead
we will  go a bit further than the argument developed in Section \ref{00} and show that
\begin{equation}\label{cal-J}
{\mathcal J} \ = \
\int\!\!\!\int_{R_{r_0}} 
e^{2\pi i \lambda  \varphi(s,t) + \rho s + \eta t} \ \frac{ds}{s} \frac{dt}{t} \ + \ O_{r_0,\omega_0}(1)
\end{equation}
where 
$$
R_{r_0} \ = \ \bigl\{(s,t): |s|,|t| < r_0/2, \ \, \max_{2\le m \le m_0} \varepsilon_m |t|^{m}, \,
\max_{2\le n \le n_0} \sigma_n |s|^n \,  \le |s\, t| \,  \bigr\}.
$$
Again the
important point here is that the bound $O(1)$  in \eqref{cal-J} can be taken to be independent of $\varepsilon$
(and therefore independent of $x,y, \lambda, M$ and $N$).

Here $\varepsilon_m := |\partial^m_t \varphi(0,0)|$ and $\sigma_n := |\partial^n_s \varphi(0,0)|$,
Note that 
$\varepsilon_m = |\partial^m_t \psi(x,y,\omega_0)|$ and 
$\sigma_n = |\partial^n_s \psi(x,y,\omega_0)|$ depend on $\varepsilon$
since $(x,y)$ depends on $\varepsilon$. In particular $\varepsilon_{m_0} \sim A$
and $\sigma_{n_0} \sim B$ are both nonzero. Also
$0<\varepsilon_2 = |\psi_{tt}(x,y,\omega)|\le \varepsilon$ and $0 < \sigma_2 = |\psi_{ss}(x,y,\omega_0)|$
by our choice of the point $(x,y)$ in Lemma \ref{point-properties}. These properties will be very important
towards the end of the analysis.

We will achieve \eqref{cal-J} in a couple of steps. First
we will obtain a uniform bound for part of the oscillatory integral ${\mathcal J}$; namely, we consider the part
where we integrate over the region 
$R_{r_0}^{*}  =  \{(s,t): \, |s|, |t| \le r_0/2, \  |s\, t| \ \le \ \max_{2\le m \le m_0} \varepsilon_m |t|^m \}$.

With the notation in Section \ref{prelude-I}, we write
$$
 \int\!\!\!\int_{R_{r_0}^{*}}  e^{i [\lambda \varphi(s,t) + \rho s + \eta t]} \frac{ds}{s} \frac{dt}{t} \ = \
\sum_{(p,q)\in {\mathcal R}_{0}^{*}} I_{p,q}
$$
where now 
$$
{\mathcal R}_{0}^{*} = \bigl\{(p,q): \, 2^{-p}, 2^{-q} \le r_0/2, \  \ 2^{-p-q} \ \le \ \max_{2\le m \le m_0 -1}  \varepsilon_m \, 
2^{-qm} \bigr\}.
$$
We
split ${\mathcal R}_{0}^{*}$ into at most $m_0$ subsets depending on which of the terms 
$\varepsilon_m \, 2^{-q m}$ is maximal among $2\le m \le m_0$. We will concentrate on
$$
\sum_{(p,q)\in {\mathcal R}_{0}^m} I_{p,q} \ = \ \sum_{(p,q)\in{\mathcal R}_{0}^m}
 \int\!\!\!\int e^{i [\lambda \varphi(s,t) + \rho s + \eta t]}  \zeta(2^p s) \zeta(2^q t) \, \frac{ds}{s} \frac{dt}{t} 
$$
where ${\mathcal R}_{0}^m$ consists those $(p,q)$ in ${\mathcal R}_{0}^{*}$ where
$\varepsilon_m \, 2^{-q m}$ is maximal. We set $N = p + q(m - 1)$
and define $L$ so that $2^L = \varepsilon_m^{-1}$. Hence the pairs $(p,q)$ in
${\mathcal R}_{0,1}^m$ satsify $N \ge L$.

We compare $I_{p,q}$ with
$$
II_{p,q} \ = \  \ c \, \int\!\!\!\int e^{i \lambda [\varphi_1(s) + \varphi_2(t)] + \rho s + \eta t} \zeta(2^p s) \zeta(2^q t) \,
 \frac{ds}{s} \frac{dt}{t}
$$
where $c = e^{\lambda \varphi(0,0)}$, $\varphi_1(s) = \varphi(s,0)$ and $\varphi_2(t) = \varphi(0,t)$.
We make the change of variables $s \to 2^{-p} s$ and $t \to 2^{-q} t$ and note that
$$
|D_{p,q}| \ := \ |I_{p,q} - II_{p,q}| \ := \ |\lambda| \,
\bigl|\sum_{k,\ell\ge 1} \frac{1}{k! \ell!} \partial^{k,\ell}\varphi(0,0) (2^{-p}s)^k (2^{-q} t)^{\ell}\bigr|
 \le \ C |\lambda| 2^{-p-q}
$$
\begin{equation}\label{difference}
= \ C |\lambda| 2^{-p-(m-1)q} \varepsilon_m^{-1} \varepsilon_m 2^{- q m} \ = \ 
 C |\lambda| 2^{- (N-L)} \varepsilon_m \, 2^{-m q}.
\end{equation}

{\bf Claim}: \ If $\Phi(t) := \varphi(2^{-p}s, 2^{-q}t) = \psi(x+2^{-p}s, y+2^{-q}t,\omega_0)$, then
there is a $2 \le m' \le m_0$ such that $|\Phi^{(m')}(t)| \gtrsim \,  \varepsilon_m 2^{-qm}$
for all $|t| \sim 1$. 

This is our replacesment for the derivative bound \eqref{derivative-hyp}.

To prove the claim, we first set 
$$
H(t) \ = \ \varphi(2^{-p} s,2^{-q} t) - \varphi(0,2^{-q} t) \ = \
2^{-p} s \int_0^1 \frac{\partial\varphi}{\partial s}(r 2^{-p} s, 2^{-q} t) \, dr
$$
and note that for any $2 \le m' \le m_0$, 
\begin{equation}\label{H}
|H^{(m')}(t)| \ \le \ C 2^{-p-2q} \ = \ C 2^{-q} 2^{-p-q}  \ \le \ 
C 2^{-q} \varepsilon_m 2^{-q m}.
\end{equation}
Since we can take $2^{-q} < r_0/2$ to be small, we will treat $H(t)$ as an error term.

Next we write $\varphi(0, 2^{-q}t) = g(t) + h(t)$ where 
$$
g(t) = \sum_{\ell=0}^{m_0 -1} \frac{1}{\ell!} \partial^{\ell}_t \varphi(0,0) (2^{-q} t)^{\ell} \ \ {\rm and} \ \
h(t) = (2^{-q}t)^{m_0} \int_0^1 \!\!\cdots\!\!\int_0^1 \partial^{m_0}_t \varphi(0,\sigma(r)\,  2^{-q} t)
d\mu
$$ 
Here $\sigma(r) = r_1 \cdots r_{m_0}$ and $d\mu = r_2  r_3^2 \cdots r_{m_0}^{m_0 -1}$.
By classical Cauchy estimates applied to the analytic function $z \to \partial^{m_0}_t \psi(x+s,z,\omega_0)$, we have 
for any $|s|, |t| < r_0/2$,
$$
\bigl|\partial^{\ell+m_0}\varphi(s,t)| = |\partial^{\ell}_t [\partial^{m_0}_t \psi](x+s, y+t,\omega_0)| \ \le
\ C_{\ell} A \ \lesssim \ A
$$
since $x,y \in B_{r_0/2}(x_0,y_0)$. This follows from Lemma \ref{m,n>2}, see
\eqref{m>2} (also see the footnote there). See also \eqref{m>2-again}. Hence for any $2\le m' \le m_0$, we have
\begin{equation}\label{h}
|h^{(m')}(t)| \ \lesssim \ A 2^{-q m_0} \ \lesssim \  \varepsilon_{m_0} 2^{-q m_0}.
\end{equation}
A simple equivalence of norms argument (applied to the space of polynomials of degree
at most $m_0 - 1$) shows the existence of a positive constant $\delta_{m_0} > 0$ and an
$2 \le m' \le m_0 -1$ such that 
\begin{equation}\label{g}
 |g^{(m')}(t)| \ \ge \ \delta_{m_0} \max_{2\le m'' \le m_0 - 1} |\partial^{m''}_t \varphi(0,0)| 2^{-q m''} 
\end{equation}
for all $|t| \sim 1$. We note that the right hand side of \eqref{g} is equal to $\delta_{m_0} \, \varepsilon_m 2^{-q m}$
if $m \le m_0 - 1$.

To complete the proof of the claim, we split the argument into two cases:

{\it Case 1}: $\varepsilon_{m_0} 2^{-q m_0} \ge \ K^{-1} \delta_{m_0} \varepsilon_m 2^{-q m}$ 
where $4K$ is the implicit constant appearing in the bound \eqref{h}.

In this case, we have by \eqref{m>2} or \eqref{m>2-again}
$$
\bigl|\Phi^{(m_0)}(t)\bigr| \ = \ 2^{-q m_0} |\partial^{m_0}_t 
\psi(x+ 2^{-p} s, y + 2^{-q} t, \omega_0)| \ge A 2^{-q m_0} / 2
$$ 
for all $|t| \sim 1$, completing the proof of the claim in this case since $A \sim \varepsilon_{m_0}$.

{\it Case 2}: $\varepsilon_{m_0} 2^{-q m_0} \le K^{-1} \delta_{m_0} \varepsilon_m 2^{-q m}$. 

In this case we see that $2\le m \le m_0 - 1$. We
use the existence of an $2 \le m' \le m_0 - 1$ such that \eqref{g} holds. Hence by
\eqref{H}, \eqref{h} and \eqref{g}, we have
$$
|\Phi^{(m')}(t)| \  \ge \ |g^{(m')}(t)| \ - \ |h^{(m')}(t)| \  - \  |H^{(m')}(t)| \ \ge
$$
$$
 \delta_{m_0} \varepsilon_m 2^{-q m} \  - \
K \varepsilon_{m_0} 2^{-q m_0} / 4  \ - \ C 2^{-q} \varepsilon_m 2^{-q m} \ \ge \ (\delta_{m_0}/2) \varepsilon_m 2^{-q m}
$$
since $C 2^{-q} \le C r_0$ and we can take $r_0 < C^{-1} \delta_{m_0}/4$. 
This completes the proof of the claim.

A similar but easier argument shows that $|\Psi^{(m')}(t)| \gtrsim \varepsilon_m 2^{-q m}$
for all $|t| \sim 1$, with the same $m'$ as in the Claim. Here $\Psi(t) = \psi(x, y + 2^{-q} t, \omega_0)$.
Hence an application of van der Corput's lemma shows that 
$$
|D_{p,q}| \ = \ |I_{p,q}- II_{p,q}| \ \le \ C  \, (|\lambda| \varepsilon_m 2^{-q m})^{-1/m'}.
$$

Taking a convex combination of this estimate with the estimate \eqref{difference},
we have
$$
|D_{p,q}| 
\ \le \ C \, 2^{-\delta (N-L)}  \min(|\lambda| \varepsilon_m 2^{-q m}, (|\lambda| \varepsilon_m 2^{-q m})^{-1})^{\epsilon}
$$
for some absolute exponents $\delta>0$ and $\epsilon >0$. Hence
$$
| \sum_{(p,q) \in {\mathcal R}_{0,1}^m} D_{p,q} | \lesssim
 \mathop{\sum_{N\ge L}\sum_{p,q:}}_{p - (m-1) q = N} |D_{p,q}| 
\lesssim
\sum_{N\ge L} 2^{-\delta (N-L)}\!\!\!\!\!\!\!\!\sum_{\begin{array}{c}\scriptstyle
         q: \\
         \vspace{-5pt}\scriptstyle p -  (m-1)q = N 
         \end{array}}\!\!\!\!\!\!\!\!\! \min(|\lambda| \varepsilon_m 2^{-q m}, (|\lambda| \varepsilon 2^{-q m} )^{-1})^{\epsilon_0}
$$
which implies that
$$
\int\!\!\!\int_{R_{r_0}^{*}}  e^{i [\lambda \varphi(s,t) + \rho s + \eta t]} \frac{ds}{s} \frac{dt}{t} \ = \
c \int\!\!\!\int_{R_{r_0}^{*}}  e^{i (\lambda [\varphi_1(s) + \varphi_2(t)] + \rho s + \eta t)} \frac{ds}{s} \frac{dt}{t} \ + \
O(1).
$$
The integral on the right can be written as
$$
\int_{|s|<r_0/2} e^{i [\lambda \varphi_1(s) + \rho s]} G(s) \, \frac{ds}{s} \ \ \ {\rm where} \ \ 
G(s) = \int_{R(s)} e^{i [\lambda \varphi_2(t) + \eta t]} \, \frac{dt}{t}
$$
where $R(s) = \{ |t| < r_0/2 : (s,t) \in R_{0}^{*} \}$ and
%Since $\varphi_2(t) = \omega \cdot {\overline{\phi}}(x,y+t)$,
$G(s) = G_{x,y,\lambda, \eta}(s)$ depends 
on the parameters $x,y, \lambda$ and $\eta$ as well as depending on $\omega_0 \in {\mathbb S}^{d-1}$.  
An application of Proposition \ref{osc-est-III} shows
that $|G(s)| \le C$ for all $s$ with $|s|<r_0/2$ for some constant $C$ which may be taken to be independent
of $x,y,\lambda$ and $\eta$ (but may depend on $\omega_0$). 
%However if $\psi$ satisfies the factorisation
%hypothesis (FH) on ${\mathbb S}^{d-1}$, then $B$ can be taken to be independent of $\omega \in {\mathbb S}^{d-1}$
%as well. 
Since $G$ is an even function and clearly, $|G'(s)| \lesssim |s|^{-1}$, another application of
Proposition \ref{osc-est-III} shows that
$$
\int_{|s|<r_0/2} e^{i [\lambda \varphi_1(s) + \rho s]} G(s) \, \frac{ds}{s} \ = \ O_{r_0, \omega_0}(1)
$$
and so 
$$
{\mathcal J} \ = \
\int\!\!\!\int_{R_{r_0, 1}} 
e^{2\pi i \lambda  \varphi(s,t) + \rho s + \eta t} \ \frac{ds}{s} \frac{dt}{t} \ + \ O_{r_0,\omega_0}(1)
$$
where 
$$
R_{r_0, 1} = \{(s,t): |s|,|t| < r_0/2, \ \, \max_{2 \le m \le m_0} \varepsilon_m |t^m|
\le  \ |s\, t| \}.
$$
%Recall that $\varphi(s,t) = \psi(x+s,y+t,\omega_0)$ and so 
%$|\partial^{m_0}_t \varphi (0,0)| = |\partial^{m_0}_t \psi(x,y,\omega_0)|  \sim A$.

In exactly the same way as above but now using the derivative bound \eqref{m>2} (or \eqref{n>2-again}) for 
$\partial^{n_0}_s \varphi (s,t) = \partial^{n_0}_s \psi(x+s,y+t,\omega_0)$ which is valid for
all $(s,t)$ with $|s|, |t| < r_0/2$, we conclude that
$$
\int\!\!\!\int_{R_{r_0, 2}} 
e^{2\pi i \lambda  \varphi(s,t) + \rho s + \eta t} \ \frac{ds}{s} \frac{dt}{t} \ = \ O_{r_0,\omega_0}(1)
$$
where 
$$
R_{r_0,2} \ = \ \{(s,t) : |s|, |t| \le r_0/2, \  \max_{2\le n \le n_0} \sigma_n |s|^n \le |s t| \le \max_{2\le m \le m_0} 
\varepsilon_m |t|^m \}.
$$
Hence we arrive at \eqref{cal-J}.

Next we compare the oscillatory integral in \eqref{cal-J} to 
$$
c  \int\!\!\!\int_{R_{r_0}} 
e^{2\pi i \lambda [{\mathcal A} s + {\mathcal B} t + {\mathcal C} s t] + \rho s + \eta t} \ \frac{ds}{s} \frac{dt}{t} 
$$
where $c = e^{i\lambda \psi(x,y,\omega_0)}, \, {\mathcal A} = \partial_s \varphi(0,0), \ 
{\mathcal B} = \partial_t \varphi(0,0)$ and ${\mathcal C} = \partial^{1,1}\varphi(0,0) =
\psi_{st}(x,y,\omega_0)$. Recall  that 
$$
\varphi(s,t) = \psi(x+s, y+t,\omega_0) + \omega_0 \cdot {\vec{L}}_1 s + \omega_0 \cdot {\vec{L}}_2 t
= \omega_0 \cdot [{\overline{\phi}}(x+s,y+t) + {\vec{L}}_1 s + {\vec{L}}_2 t].
$$ 

Again we write
$$
\int\!\!\!\int_{R_{r_0}} 
e^{2\pi i \lambda  \varphi(s,t) + \rho s + \eta t} \ \frac{ds}{s} \frac{dt}{t} \ = \ \sum_{(p,q) \in {\mathcal R}_{0}} 
I_{p,q}
$$
and
$$
c  \int\!\!\!\int_{R_{r_0}} 
e^{2\pi i \lambda [{\mathcal A} s + {\mathcal B} t + {\mathcal C} s t] + \rho s + \eta t} \ \frac{ds}{s} \frac{dt}{t} \ = \ 
\sum_{(p,q)\in {\mathcal R}_{0}} II_{p,q}
$$
where now $(p,q)\in {\mathcal R}_{0}$ satisfies
$$
2^{-p}, 2^{-q} < r_0/2, \ \ {\rm and} \ \ 
\max_{2 \le m \le m_0} \varepsilon_m 2^{-qm}, \,
\max_{2 \le n \le n_0} \sigma_n 2^{-p n}
\le  \ 2^{-p-q}.
$$

As before we make the change of variables $s \to 2^{-p}s$ and $t \to 2^{-q} t$ in both
integrals $I_{p,q}$ and $II_{p,q}$, consider te difference $D_{p,q} = I_{p,q} - II_{p,q}$,
first observing
$$
\varphi(2^{-p}s, 2^{-q}t) - \varphi(0,0) - \partial_s \varphi(0,0) 2^{-p} s - \partial_t \varphi(0,0) 2^{-q} t -
\frac{1}{2} \partial_{st}\varphi(0,0) (2^{-p}s)(2^{-q}t)
$$
$$
= \sum_{m=2}^{m_0-1} \frac{1}{m!} \varepsilon_m (2^{-q} t)^{m} \ + \ 
 \sum_{n=2}^{n_0-1} \frac{1}{n!} \sigma_n (2^{-p} s)^n  \ + \
\sum_{k,\ell\ge 2} \frac{1}{k!\ell!}
\partial^{k,\ell}\varphi(0,0) (2^{-p}s)^k (2^{-q} t)^{\ell}
$$
$ + \ h_q(t) + f_p(s)$ 
where $|h_{q}(t)| \lesssim \varepsilon_{m_0} 2^{-q m_0}$ and $|f_{p}(s)| \lesssim \sigma_{n_0} 2^{-p n_0}$.
Recall the formula for $h = h_q$, appearing shortly after \eqref{H},
$$
h(t) \ = \ 
 (2^{-q}t)^{m_0} \int_0^1 \!\!\cdots\!\!\int_0^1 \partial^{m_0}_t \varphi(0,\sigma(r)\,  2^{-q} t)
d\mu
$$ 
and so indeed the estimate $|h_q(t)| \lesssim A 2 ^{-q m_0} \lesssim \epsilon_{m_0} 2^{- q m_0}$ follows from \eqref{m>2}. Similarly
for $f_q(s)$. 

Therefore
$$
|D_{p,q}|  \ \lesssim \ 2^{-p-2q} + 2^{-2p - q} + 
 \sum_{m=2}^{m_0} \varepsilon_m 2^{- q m}  \ + \ \sum_{n=2}^{n_0}
 \sigma_n 2^{-p n} 
$$
where we recall $\varepsilon_m = \partial^m_t \varphi(0,0)$ and
$\sigma_n := \partial^n_s \varphi(0.0)$. We
split the sum over $(p,q) \in {\mathcal R}_{0}$  into various subsums,
depending on which of these four terms above is largest. The argument is similar
in all cases and so we shall only concentrate on one subsum; namely
we restrict those $(p,q) \in {\mathcal R}_{0}$ where the maximal term is
$\max_{2\le m \le m_0} \varepsilon_m 2^{-q m}$. As before we divide these
$(p,q)$ depending on which term in this maximum is maximal; say $\varepsilon_m 2^{-q m}$.
We will call this subcollection ${\mathcal R}_{0}^{*}$.
%Finally we split ${\mathcal R}_{0,3}^{*}$ into two parts, depending on whether
%$p\le q$ or $q < p$, and we consider
%$$
%{\mathcal R}_{0,3}^{**} \ := \ 
%\{(p,q) \in {\mathcal R}_{0,3} : 2^{-q} \le 2^{-p} \le \varepsilon_m 2^{-q m} \}.
%$$ 
Therefore for $(p,q) \in {\mathcal R}_{0}^{*}$, we have
\begin{equation}\label{difference-Again}
|D_{p,q}| \ \lesssim \ \varepsilon_m 2^{- q m} \ = \  2^{-(N-L)} \, 2^{-p-q}
\end{equation}
where $N := q(m-1) - p$ and $2^L := \varepsilon_m$. We note that $L\le N$ or $\varepsilon_m \le 2^N$ since 
$\varepsilon_m 2^{-q m} \le 2^{-p-q}$ for $(p,q) \in {\mathcal R}_{0}$. 

We now aim to obtain a correspoinding decay bound for $I_{p,q}$ and $II_{p,q}$
whenever  $(p,q) \in {\mathcal R}_{0}^{*}$. We write the phase function
in $I_{p,q}$ as
$$
\lambda \varphi(2^{-p} s, 2^{-q} t) + \rho 2^{-p} s + \eta 2^{-q} t =  \lambda 2^{-p-q} \Phi(s,t) + E 2^{-p} s + F 2^{-q} t
$$
 where $\Phi(s,t) =$
$$
2^{p+q} \bigl[ \psi(x+2^{-p} s, y+ 2^{-q} t,\omega_0) - \psi(x,y,\omega_0) - \partial_s \psi(x,y,\omega_0)
2^{-p} s - \partial_t \psi(x,y,\omega_0) 2^{-q} t \bigr],
$$
$E = \lambda \partial_s \psi(x,y,\omega_0) + \rho = \lambda {\mathcal A} + \rho$ and $F = \lambda {\mathcal B} + \eta$.

The mixed second derivative of
the function $\Phi(s,t)$ is equal to
$$
\psi_{st} (x, y, \omega_0) +  O(2^{-\min(p,q)}) \ \ {\rm and \ so} \ \
|\Phi_{st}(s,t)| \  \ge \ |\psi_{st}(x,y,\omega_0)| - C r_0
$$ 
since
$2^{-\min(p,q)} \le r_0/2$ for $(p,q) \in {\mathcal R}_0$. Hence by Lemma \ref{point-properties}, we have
\begin{equation}\label{Phi-st}
|\Phi_{st}(s,t)| \ \ge \  D / 4 
\end{equation}
if we take $r_0$ small enough. Recall that we chose the parameter $r_0$ after we located the point
$(x_0,y_0, \omega_0)$ such that $D = |\psi_{st}(x_0,y_0,\omega_0)| > 0$ and $\psi_{tt}(x_0,y_0,\omega_0) = 0$,
see \eqref{FH-fail}. 

Since 
$$
I_{p,q} \ = \ \int\!\!\int e^{i[\lambda 2^{-p-q} \Phi(s,t) +  E 2^{-p} +- F 2^{-q} t]} \,  \zeta(s) \zeta(t) \, 
\frac{ds}{s} \frac{dt}{t},
$$
the derivative bound \eqref{Phi-st} puts us in a position to invoke a multidimensional version of van der Corput's
lemma. Such a variant can be found in \cite{S} (see Proposition 5 on page 342)
and implies
\begin{equation}\label{decay-st}
|I_{p,q}| \ \le \ C (|\lambda| 2^{-p-q})^{-1/2}
\end{equation}
 but the constant $C$ in the  bound here depends on the $C^3$ norm of 
$f(s,t) :=  \Phi(s,t) + (\lambda 2^{-p-q})^{-1} [ E 2^{-p} +  F 2^{-q} t ]$. 
This is in stark contrast to the one dimensional version of van der Corput's lemma
and it begs the question to what extent is there a multidimesional version of van der Corput's lemma
with all the appropriate uniformity. In any case,
an examination of the proof of Proposition 5 in \cite{S} shows in fact the bound
$C$ can be taken to depend only on the $L^{\infty}$ norms of the third order derivatives
of $f(s,t)$ and in particular, not on the first order derivatives. The third order derivatives of $f$ are the
same as those of $\Phi$.

From above, we have $\Phi(s,t) = 2^{p+q} \Psi(s,t)$ where $\Psi(s,t) =$
$$
\sum_{m=2}^{m_0-1} \frac{1}{m!} \varepsilon_m (2^{-q} t)^{m} \ + \ 
 \sum_{n=2}^{n_0-1} \frac{1}{n!} \sigma_n (2^{-p} s)^n  \ + \ O(2^{-p-q}) + O(\varepsilon_{m_0} 2^{-q m_0})
+ O(\sigma_{n_0} 2^{-p n_0}) 
$$
where the three O terms are stable under differentiation; for example, $\partial^{k,\ell} O( 2^{-p-q}) = O(2^{-p-q})$.
Therefore we have the following upper bound on the $C^3$ norm of $\Phi$,
$$
\|\Phi\|_{C^3} \lesssim 2^{p+q} \bigl[ \max_{2\le m \le m_0} \varepsilon_m 2^{-q m} + \max_{2\le n\le n_0} \sigma_n 2^{-p n}
\bigr]  + 
O(1) \ = \ O(1)
$$
since $(p,q) \in {\mathcal R}_0$. This implies that the constant $C$ in \eqref{decay-st} can be taken
to be an absolute constant, although depending on ${\overline{\phi}}$, it can be taken to be independent of $p,q, x, y, \lambda, \rho$ and $\eta$. 

A more elementary argument shows that $|II_{p,q}|\ \le C (|\lambda| 2^{-p-q})^{-1}$.
Taking a convex combination of \eqref{difference-Again} and \eqref{decay-st},
we have
$$
|D_{p,q}| 
\ \le \ C \, 2^{-\delta (N-L)}  \min(|\lambda| 2^{-p-q}, (|\lambda|  2^{-p-q})^{-1})^{\epsilon}
$$
for some absoluate exponents $\delta>0$ and $\epsilon >0$. Hence
$$
| \sum_{(p,q) \in {\mathcal R}_{0}} D_{p,q} | \lesssim
 \mathop{\sum_{N\ge L}\sum_{p,q:}}_{(m-1) q - p= N} |D_{p,q}| 
\lesssim
\sum_{N\ge L} 2^{-\delta (N-L)}\!\!\!\!\!\!\!\!\sum_{\begin{array}{c}\scriptstyle
         q: \\
         \vspace{-5pt}\scriptstyle (m-1)q - p= N 
         \end{array}}\!\!\!\!\!\!\!\!\! \min(|\lambda| 2^{-p-q}, (|\lambda| 2^{-p-q} )^{-1})^{\epsilon_0}
$$
which implies that
$$
{\mathcal J} \ = 
c  \int\!\!\!\int_{R_{r_0}} 
e^{2\pi i \lambda [{\mathcal A} s + {\mathcal B} t + {\mathcal C} s t] + \rho s + \eta t} \ \frac{ds}{s} \frac{dt}{t} \ + \
O_{r_0, \omega_0}(1).
$$
We have successfully reduced the analysis of ${\mathcal J}$  to an oscillatory integral with a quadratic phase.
We label the  oscillatory integral above with the quadratic phase as ${\mathcal J}'$. 

\subsection{Analysis of ${\mathcal J}'$: reduction to the heart of the matter}

Recall that our immediate main task \eqref{Main-osc-est} follows from  \eqref{MOE-reduction}, 
the logarithmic  bound
from below on the oscillatory integral ${\mathcal I}= I + II + III + IV$
where each of these terms is of the form ${\mathcal J}$ with $\rho = \pm M$ and $\eta = \pm N$. 
We rewrite the oscillatory integral ${\mathcal J}'$ above as 
$$
{\mathcal J}' \ = \ 
\int\!\!\!\int_{R_{r_0}} 
e^{2\pi i ( \lambda  {\mathcal C} s t +  [\lambda {\mathcal A} \pm M] s + [\lambda {\mathcal B} \pm N] t) } \ \frac{ds}{s} \frac{dt}{t}
$$
and rename the coefficients $\lambda' := \lambda {\mathcal C},  \, {\mathcal M} := \lambda {\mathcal A} \pm  M$
and ${\mathcal N} := \lambda {\mathcal B} \pm  N$. Since ${\mathcal C} = \psi_{st}(x,y,\omega_0)$, we have
$|\lambda'| \sim |\lambda|$. In fact by Lemma \ref{point-properties}, we have
$(D/2) |\lambda| \le |\lambda'| \le 2D |\lambda|$ . Furthermore
$$
{\mathcal M} \ = \ \lambda [\partial_s \psi(x,y,\omega_0) + \omega_0 \cdot {\vec{L}}_1] \pm M \ =: \ \lambda v_1(x,y,\omega_0) \pm M
$$
where $v_1(x,y,\omega_0)$ is the first component of the vector 
$$
{\vec{v}}(x,y,\omega_0) \ = \ \nabla \psi(x,y,\omega_0) + \omega_0 \cdot {\overline{L}} \ = \
\omega_0 \cdot \bigl[ \nabla {\overline{\phi}}(x,y) + {\overline{L}} \bigr],
$$
both of whose components are nonzero, see  Lemma \ref{point-properties}. Similarly for ${\mathcal N}$. 
 
To understand ${\mathcal J}'$ more precisely, we will decompose the region $R_{r_0}$ into various
subregions $R_{m,n}$ determined by where the maxima 
$$
\max_{2\le m'  \le m_0} \varepsilon_{m'} |t|^{m'} \ = \ \varepsilon_m |t|^m \ \ \ {\rm and} \ \ \ 
\max_{2\le n'  \le n_0} \sigma_{n'} |s|^{n'} \ = \ \sigma_n |s|^n
$$
occurs. We do this for every $2\le m \le m_0$ with $\varepsilon_m > 0$ and $2\le n \le n_0$ with
$\sigma_n > 0$. 
Hence
$$
R_{m,n} \ = \ \{(s,t) \in R_{r_0} : \, d_{1,m} \le |t| \le d_{2,m} \ \ {\rm and} \ \ c_{1,n} \le |s| \le c_{2,n} \}
$$
for some choice of exponents $d_{1,m}, d_{2,m}, c_{1,n}$ and $c_{2,n}$. In fact $d_{1,2} = c_{1,2}  = 0$
and $d_{2,m_0} = c_{2, n_0} = 1$. Otherwise when $\varepsilon_m >0$,
$$
d_{1,m} \ := \ \max_{m' < m} \, \Bigl( \frac{\varepsilon_{m'}}{\varepsilon_m}\Bigr)^{\frac{1}{m-m'}} \ \ {\rm and} \ \
d_{2,m} \ := \ \min_{m < m'} \, \Bigl( \frac{\varepsilon_{m}}{\varepsilon_{m'}}\Bigr)^{\frac{1}{m'-m}}
$$
and when $\sigma_n > 0$,
$$
c_{1,n} \ := \ \max_{n' < n} \, \Bigl( \frac{\sigma_{n'}}{\sigma_n}\Bigr)^{\frac{1}{n-n'}} \ \ {\rm and} \ \
c_{2,n} \ := \ \min_{n < n'} \, \Bigl( \frac{\sigma_{n}}{\sigma_{n'}}\Bigr)^{\frac{1}{n'-n}}.
$$
Importantly we note that when $m\ge 3$, we have $d_{1,m} > 0$ since $\varepsilon_2 >0$
and also when $n\ge 3$, we have $c_{1,n} > 0$ since $\sigma_2 >0$. This came from our 
choice of point $(x,y)$ from Lemma \ref{point-properties}. This bit of information is important enough to record;
\begin{equation}\label{important-info}
d_{1,m} \ \ {\rm and} \ \ c_{1,n} \ > 0 \ \ \ {\rm when} \ \ \ m, n \ge 3.
\end{equation}

Hence we can decompose ${\mathcal J}'$ into a sum of oscillatory integrals ${\mathcal J}_{m,n}$ where
$$
{\mathcal J}_{m,n} \ := \ 
\int\!\!\!\int_{R_{m,n}}
e^{2\pi i ( \lambda' s t +  {\mathcal M} s +  {\mathcal N}  t) } \ \frac{ds}{s} \frac{dt}{t}
$$
We split ${\mathcal J}_{m,n} = {\mathcal J}_{m,n}^1 + {\mathcal J}_{m,n}^2$ where for ${\mathcal J}_{m,n}^1$
the region of integration $R_{m,n}$ is restricted to $R_{m,n}^1 = \{(s,t) \in R_{m,n} : |t| \le |{\mathcal N}|^{-1} \}$. We
claim that
\begin{equation}\label{m,n-1}
{\mathcal J}_{m,n}^2 \ = \ 
\int\!\!\!\int_{R_{m,n}^2}
e^{2\pi i ( \lambda' s t +  {\mathcal M} s +  {\mathcal N}  t) } \ \frac{ds}{s} \frac{dt}{t} \ = \ O(1)
\end{equation}
and to achieve this we split ${\mathcal J}_{m,n}^2 = {\mathcal J}_{m,n}^{2,-} + {\mathcal J}_{m,n}^{2,+}$ 
further where
$R_{m,n}^2 = R_{m,n}^{2,-} \cup R_{m,n}^{2,+}$ and 
$R_{m,n}^{2,-} = \{(s,t)\in R_{m,n}^2 : |s| \le C |{\mathcal N}/\lambda'| \}$ for some large
constant $C$. We write
$$
{\mathcal J}_{m,n}^{2,-} \ = \ \int_{|t|\ge |{\mathcal N}|^{-1}} e^{2\pi i {\mathcal N} t} F(t) \frac{dt}{t} \ \ {\rm where} \ \
F(t) = \int_{s\in R_{m,n}^{2,-}(t)} e^{2\pi i [\lambda' t + {\mathcal M}]s} \frac{ds}{s}
$$
and $R_{m,n}^{2,-}(t) = \{s : (s,t) \in R_{m,n}^{2,-}\}$. Note that $F(t) = O(1)$ and
$$
F'(t) \ = \ 2 \pi i \lambda' \int_{s\in R_{m,n}^{2,-}(t)} e^{2\pi i [\lambda' t + {\mathcal M}] s} ds \ + \ 
O(|t|^{-1}).
$$
Therefore integration by parts shows that
$$
{\mathcal J}_{m,n}^{2,-} \ = \ - \frac{\lambda'}{{\mathcal N}} \int_{|t|\ge |{\mathcal N}|^{-1}} e^{2\pi i {\mathcal N} t}
\Bigl[\int_{s\in R_{m,n}^{2,-}(t)} e^{2\pi i [\lambda' t + {\mathcal M}]s} ds \Bigr] \frac{dt}{t} \ + \ O(1)
$$
$$
- \frac{\lambda'}{{\mathcal N}} \int_{|s|\le C |{\mathcal N}/\lambda'|} e^{2\pi i {\mathcal M} s}
\Bigl[\int_{t\in R_{m,n}^{2,-}(s)} e^{2\pi i [\lambda' s + {\mathcal N}]t} \frac{dt}{t}\Bigr] ds \ + \ O(1) \ = \ O(1)
$$
since the last inner integral is easily seen to be $O(1)$. Next we show that ${\mathcal J}_{m,n}^{2,+} = O(1)$. 
We write
$$
{\mathcal J}_{m,n}^{2,+} \ = \ \int_{|s|\ge C |{\mathcal N}/\lambda'|} e^{2\pi i {\mathcal M} s} G(s) \frac{ds}{s} \ \ {\rm where} \ \
G(s) = \int_{t\in R_{m,n}^{2,+}(s)} e^{2\pi i [\lambda' s + {\mathcal N}]t}\,  \frac{dt}{t}
$$
and $R_{m,n}^{2,+}(s) = \{t : (s,t) \in R_{m,n}^{2,+}\}$. Note that when $|s| \ge C |{\mathcal N}/\lambda'|$,
we have $|\lambda' s + {\mathcal N}| \sim |\lambda' s|$ and a simple integration by parts argument shows that
$|G(s)| \lesssim |{\mathcal N}/\lambda' s|$, implying that 
$$
|{\mathcal J}_{m,n}^{2,+}| \ \lesssim \ |{\mathcal N}/\lambda'| \, \int_{|s|\ge C |{\mathcal N}/\lambda'|} \frac{ds}{s^2} \ = \ 
O(1)
$$  
and so ${\mathcal J}_{m,n}^2 = O(1)$ as claimed. 

Since 
$$
{\mathcal J}_{m,n}^1 \ = \ \int_{|t| \le |{\mathcal N}|^{-1}} e^{2\pi i {\mathcal N} t} \, \Bigl[
\int_{s \in R_{m,n}^1(t)} e^{2\pi i [\lambda' t + {\mathcal M}]s} \frac{ds}{s}\, \Bigr] \frac{dt}{t},
$$
we see that
$$
{\mathcal J}_{m,n}^1 \ = \ \int\!\!\int_{R_{m,n}^1} e^{2\pi i [\lambda' st + {\mathcal M} s]} \, \frac{ds}{s} \frac{dt}{t} \ + \ O(1)
$$
since $|e^{2\pi i {\mathcal N} t} - 1 | \le 2 \pi |{\mathcal N} t|$. As we did with $R_{m,n}^2$,
we decompose $R_{m,n}^1 = R_{m,n}^{1,-} \cup R_{m,n}^{1,+}$ where
$R_{m.n}^{1,-} \{(s,t)\in R_{m,n}^1 : |s| \le |{\mathcal M}|^{-1}\}$. A similar but more straightforward
argument as above shows that ${\mathcal J}_{m,n}^{1,+} = O(1)$ and we are left with ${\mathcal J}_{m,n}^{1,-}$.

Again since
$$
{\mathcal J}_{m,n}^{1,-} \ = \ \int_{|s| \le |{\mathcal M}|^{-1}} e^{2\pi i {\mathcal M} s} \, \Bigl[
\int_{t \in R_{m,n}^{1,-}(s)} e^{2\pi i \lambda' s t} \frac{dt}{t}\, \Bigr] \frac{ds}{s},
$$
we see that
$$
{\mathcal J}_{m,n}^{1,-}  =  \int\!\!\int_{R_{m,n}^{1,-}} e^{2\pi i \lambda' st} \, \frac{ds}{s} \frac{dt}{t}  +  O(1)  = 
\int\!\!\int_{P_{m,n}^{1,-}} \frac{\sin(\lambda' s t)}{s t} \, ds dt + O(1)
$$
where $P_{m,n} = \{(s,t) \in R_{m,n}^{1,-} : s, t\ge 0 \}$ is the positive part of $R_{m,n}^{1,-}$.

Summarising, we see that
\begin{equation}\label{P-4}
{\mathcal J}' = \sum_{m=2}^{m_0}\sum_{n=2}^{n_0} {\mathcal J}_{m,n}  = 
\sum_{m=2}^{m_0}\sum_{n=2}^{n_0} \int\!\!\!\int_{P_{m,n}^{1,-}} \frac{\sin(\lambda' s t)}{s t} \, ds dt + O(1)
\end{equation}
but let us not forget that each integral above is a sum of 4 integrals, one for each choice of $\pm$
in the coefficients ${\mathcal M}$ and ${\mathcal N}$ which now only appear in the definition
of the region of integration $P_{m,n}^{1,-}$. 

\subsection{Getting to the heart of the matter}\label{heart}

Without of loss of generality, let us assume that both $\lambda, \lambda' >0$.
Recall that $\lambda \sim \lambda'$. 

Given $\lambda >0$ (which we choose to be large later, depending on $\varepsilon$) and our point choice $(x,y)$,
we now choose $M$ and $N$ so that $M - \lambda {\mathcal A} = O(1)$ and $N - \lambda {\mathcal B} = O(1)$.
Recall that $({\mathcal A}, {\mathcal B})  = {\vec{v}}(x,y,\omega_0)$ where both components
$$
{\mathcal A} \ = \ \partial_s \psi(x,y,\omega_0) + \omega_0
\cdot {\vec{L}}_1 \ \ {\rm and} \ \
{\mathcal B} \  = \ \partial_t \psi(x,y,\omega_0) + \omega_0
\cdot {\vec{L}}_2
$$
are nonzero, so we can certainly make such choices for $M$ and $N$.
We denote by ${\mathcal M}^{*} = \lambda {\mathcal A} + M$ and ${\mathcal N}^{*} =
\lambda {\mathcal B} + N$ so that both $|{\mathcal M}^{*}|, |{\mathcal N}^{*}| \sim_{\varepsilon} \lambda$. 
We will reserve ${\mathcal M}$ and ${\mathcal N}$ to denote $\lambda {\mathcal A} - M$
and $\lambda {\mathcal B} - N$, respectively, so that both ${\mathcal M}, {\mathcal N} = O(1)$. 
In fact, $|{\mathcal M}|, |{\mathcal N}| \le 1$. 

Let us now examine the integral
$$
{\mathcal S} \ := \ \int\!\!\!\int_{P_{m,n}^{1,-}} \frac{\sin(\lambda' s t)}{s t} \, ds dt 
$$
in the case where where we take ${\mathcal N}^{*}$ in the definition of $P_{m,n}^{1,-}$
and the other coefficient can either be ${\mathcal M}$ or ${\mathcal M}^{*}$. In this case
$$
{\mathcal S} \ = \ 
\int_{\begin{array}{c}\scriptstyle
         d_{1,m} \le t \le d_{2,m} \\
         \vspace{-5pt}\scriptstyle 0 \le t \le |{\mathcal N}^{*}|^{-1}
         \end{array}} \!\!\!\left[ \int_{s\in P_{m,n}^{1,-}(t)} \frac{\sin(\lambda' s t)}{s t} ds \right] \, 
\frac{dt}{t}.
$$
When $3 \le m \le m_0$, we have $d_{1,m} >  0$ from \eqref{important-info} and since 
$|{\mathcal N}^{*}| \sim_{\varepsilon} \lambda$,
we see that the interval of the $t$ integration is empty if we choose $\lambda > \Lambda(\varepsilon)$ and
$\Lambda(\varepsilon) > 0$ large enough. Hence ${\mathcal S} = 0$ in this case.

Next we consider the case $m=2$ so that $d_{1,2} = 0$ and hence (for large $\lambda$),
$$
{\mathcal S} \ = \ 
\int_0^{|{\mathcal N}^{*}|^{-1}}
         \left[ \int_{s\in P_{m,n}^{1,-}(t)} \frac{\sin(\lambda' s t)}{s} ds \right] \, 
\frac{dt}{t}
$$
but recall that for $s \in P_{m,n}^{1,-}$, we have the restriction $c_{1,n} \le s \le c_{2,n}$ as well as
$$
\varepsilon_2 t \ \le \ s \ \le \bigl(\frac{t}{\sigma_n}\bigr)^{1/(n-1)}.
$$
But this is an empty interval of integration in $s$ for every $0<t< |{\mathcal N}^{*}|^{-1}$ 
when $|{\mathcal N}^{*}|^{-1} \le c_{1,n}^{n-1} \sigma_n$ 
which will hold for large $\lambda$ whenever $n\ge 3$ since $c_{1,n} > 0$ by \eqref{important-info}
(recall that the regions $R_{m,n}$ only arise when both $\varepsilon_m, \sigma_n > 0$).
Hence ${\mathcal S} = 0$ in this case as well.

Finally we consider the case $m=n=2$. In this case, we have (again for large $\lambda$)
$$
{\mathcal S} = \int_0^{|{\mathcal N}^{*}|^{-1}} \frac{1}{t} \left[
\int_{\varepsilon_2 t}^{\sigma_2^{-1} t} \frac{\sin(\lambda' s t)}{s} ds \right] dt  = 
 \int_0^{|{\mathcal N}^{*}|^{-1}} \frac{1}{t} \left[
\int_{\varepsilon_2 \lambda' t^2}^{\sigma_2^{-1} \lambda' t^2} \frac{\sin(s)}{s} ds \right] dt 
$$
$$
= \int_0^{\varepsilon_2 \lambda'|{\mathcal N}^{*}|^{-2}} \frac{\sin(s)}{s} \log(B/A) \, ds
+  \int^{\sigma_2^{-1} \lambda' |{\mathcal N}^{*}|^{-2}}_{\varepsilon_2 \lambda'|{\mathcal N}^{*}|^{-2}} \frac{\sin(s)}{s} 
\log(C/A) \, ds
$$
where 
$$
A = \sqrt{\sigma_2 {\lambda'}^{-1} s}, \  B = \sqrt{(\varepsilon_2 \lambda')^{-1} |{\mathcal N}^{*}|^2 s}, \ \
{\rm and} \ \ 
C = |{\mathcal N}^{*}|^{-1}.
$$
Since $\log(B/A) = \log( \eta |{\mathcal N}^{*}|)$ where $\eta = 1/\sqrt{\varepsilon_2 \sigma_2}$
and $|{\mathcal N}^{*}| \sim_{\varepsilon} \lambda \sim \lambda'$, 
the first term is $O_{\varepsilon}(\log \lambda/\lambda)$. A similar estimate holds for the second term
and so in the case $m=n=2$, we have ${\mathcal S} = O_{\varepsilon}(\log \lambda/\lambda) =  O(1)$
if $\lambda > \Lambda(\varepsilon)$ and we choose $\Lambda(\epsilon) > 0$ large enough.

In a similar way, if $\lambda > \Lambda(\varepsilon)$ and  we choose $\Lambda(\varepsilon)> 0$ large enough,
the integral
$$
{\mathcal S} \ := \ \int\!\!\!\int_{P_{m,n}^{1,-}} \frac{\sin(\lambda' s t)}{s t} \, ds dt  \ = \ O(1)
$$
when we take ${\mathcal M}^{*}$ in the definition of $P_{m,n}^{1,-}$
and the other coefficient can either be ${\mathcal N}$ or ${\mathcal N}^{*}$.

This leaves us with examining ${\mathcal S}$ when we take both ${\mathcal M}$ and ${\mathcal N}$
(where we have $|{\mathcal M}|, |{\mathcal N}| \le 1$) as the two coefficients defining ${\mathcal S}$. 
The restrictions $0\le s \le |{\mathcal M}|^{-1}$ and $0\le t \le |{\mathcal N}|^{-1}$ 
for $(s,t) \in P_{m,n}^{1,-}$ are no longer restrictions since $1\le |{\mathcal M}|^{-1}, |{\mathcal N}|^{-1}$. 
Hence in this case, when both $\varepsilon_m, \sigma_n >0$, we see that $(s,t) \in P_{m,n}^{1,-}$
are precisely those nonnegative numbers such that $s,t < r_0/2$ and
$$
c_{1,n}\le s \le c_{2,m}, \ d_{1,m} \le t \le d_{2,m} \ \ {\rm and} \ \
\varepsilon_m t^{m-1} \le s \le (\sigma_n^{-1} t)^{1/(n-1)}.
$$ 

We write 
$$
{\mathcal S} \ = \ \int_{d_{1,m}}^{d_{2,m}} \frac{1}{t} \, \Bigl[ \int_{I_{m,n}(t)} \frac{\sin(\lambda' s t)}{s} \, ds\Bigr] \, dt
$$
where $I_{m,n}(t) = [c_{1,n}, \, c_{2,n}] \cap [\varepsilon_m t^{m-1}, \, (\sigma_n^{-1} t)^{1/(n-1)}]$.
We claim that
\begin{equation}\label{main-sin-I}
{\mathcal S} \ = \ O_{\varepsilon}(1/\lambda) \ \ {\rm whenever} \ \ m\ge 3 \ {\rm or} \ n\ge 3
\end{equation}
and when $m=n=2$, 
\begin{equation}\label{main-sin-II}
{\mathcal S} \ = \ O_{\varepsilon}(1/\lambda) \ + \ \frac{\pi}{4} \log1/(\varepsilon_2 \sigma_2).
\end{equation}
If so, we would be able to conclude that
$$
{\mathcal I} \ = \ O_{\varepsilon}(\log \lambda/\lambda) \ + \ \frac{\pi}{4} \log (\varepsilon_2^{-1})
$$
since $0< \sigma_2 \lesssim 1$. Therefore, since $0< \varepsilon_2 \le \varepsilon$, we can find a
$\Lambda(\varepsilon)$ such that whenever $\lambda > \Lambda(\varepsilon)$, $|{\mathcal I}| \gtrsim
\log(\varepsilon^{-1})$, establishing \eqref{MOE-reduction} and hence \eqref{Main-osc-est}. We now prove
\eqref{main-sin-I} and \eqref{main-sin-II}.

For a fixed $t \in [d_{1,m}, d_{2,m}]$, we make the change of variables $s \to (\lambda' t) s$ and write
$$
{\mathcal S} \ = \ \int_{d_{1,m}}^{d_{2,m}} \frac{1}{t} \, \Bigl[ \int_{{\mathcal I}_{m,n}(t)} 
\frac{\sin(s)}{s} \, ds\Bigr] \, dt
$$
where now 
${\mathcal I}_{m,n}= [\varepsilon_m \lambda' \, t^m, \, \sigma_n^{-1/(n-1)} \lambda'\,  t^{n/(n-1)}] \cap 
[c_{1,n} \lambda' \, t  , \,  c_{2,n} \lambda' \, t]$.

When $m\ge 3$, we have $d_{1,m}>0$ by \eqref{important-info} and so when we
interchange the order of integration, ${\mathcal S}$ has the form
$$
{\mathcal S} \ = \ \int_{C_{\varepsilon}\lambda'}^{D_{\varepsilon}\lambda'} \frac{\sin(s)}{s} \Bigl[ \int_{d_{1,m}}^g \frac{1}{t} \, dt \Bigr] \, ds
$$
where $g = O(d_{2,n})$ and $C_{\varepsilon} >0$. 
From this one can deduce that \eqref{main-sin-I} holds in this case. 

Suppose now that $m=2$ and $n\ge 3$. Hence $d_{1,2} = 0$ and $c_{1,n} > 0$ by \eqref{important-info}
and so
$$
{\mathcal S} \ = \ \int_{0}^{d_{2,2}} \frac{1}{t} \, \Bigl[ \int_{{\mathcal I}_{2,n}(t)} 
\frac{\sin(s)}{s} \, ds\Bigr] \, dt
$$
but if ${\mathcal I}_{2,n}(t)$ is nonempty, it must be the case that $\sigma_n c_{1,n}^{n-1} \le t$.
This puts us in the same position as before but with $d_{1,m}$ replaced with 
$\sigma_n c_{1,n}^{n-1} >0$ and so, as above, we see that \eqref{main-sin-I} holds in this case as well.
Hence this establishes \eqref{main-sin-I} whenever $m\ge 3$ or $n\ge 3$, as claimed.

Finally we turn to the case $m=n=2$ in which case $c_{1,2} = d_{1,2} = 0$. The same argument as above
shows that
$$
 \int_{\sigma_2 c_{2,2}}^{d_{2,2}} \frac{1}{t} \, \Bigl[ \int_{{\mathcal I}_{2,n}(t)} 
\frac{\sin(s)}{s} \, ds\Bigr] \, dt \ = \ O_{\varepsilon}(1/\lambda)
$$
and so we may assume $t \le \sigma_2 c_{2,2}$ in which case
${\mathcal I}_{2,2}(t) = [ \varepsilon_2 \lambda' t^2, \, \sigma_2^{-1} \lambda' t^2]$.
We are left with examining
$$
{\mathcal S}' \ = \
\int_0^{\min(d_{2,2}, \sigma_2 c_{2.2})}  \frac{1}{t} \, \Bigl[ \int_{\varepsilon_2 \lambda' t^2}^{\sigma_2^{-1} \lambda' t^2}
\frac{\sin(s)}{s} \, ds\Bigr] \, dt.
$$
Let us call the minimum in the outer limit of integration ${\mathfrak m}$. 
When we interchange the order of integration, ${\mathcal S}'$ splits into a sum $I + II =$
$$
\int_{\varepsilon_2 \lambda' {\mathfrak m}^2}^{\sigma_2^{-1} \lambda' {\mathfrak m}^2}
\frac{\sin(s)}{s} \, \Bigl[ \int_{\sqrt{\sigma_2 {\lambda'}^{-1} s}}^{{\mathfrak m}} \frac{1}{t} \, dt \Bigr] \, ds \ + \
\int_0^{\varepsilon_2 \lambda' {\mathfrak m}^2}
\frac{\sin(s)}{s} \, \Bigl[ \int_{\sqrt{\sigma_2 {\lambda'}^{-1} s}}^{  \sqrt{(\varepsilon_2 \lambda')^{-1} s}    } \frac{1}{t} \, dt \Bigr] \, ds
$$
of two integrals. As before $I = O_{\varepsilon}(1/\lambda)$ but now
$$
II \ = \ \frac{1}{2} \log(1/\varepsilon_2 \sigma_2) \, \int_0^{\varepsilon_2 \lambda' {\mathfrak m}^2}
\frac{\sin(s)}{s} \, ds \ = \ \frac{\pi}{4} \log(1/\varepsilon_2 \sigma_2) + O_{\varepsilon}(1/\lambda),
$$
establishing \eqref{main-sin-II}.

\subsection{The final step} We now have established \eqref{Main-osc-est}. In fact we have proved that
there exists a $\Lambda(\varepsilon)$ such that whenever $\lambda > \Lambda(\varepsilon)$, we can
find $M = M(\lambda)$ and $N = N(\lambda)$ so that
$$
 \Bigl| \int\!\!\!\int_{{\mathbb T}^2} e^{2\pi i \lambda  \omega_0 \cdot [{\overline{\phi}}(x+s, y+t) + {\overline{L}} \cdot (s,t)]}
\sin(Ms) \sin(Nt) \, \frac{ds}{s} \frac{dt}{t}\Bigr|  \ \gtrsim \
\log(1/\varepsilon)
$$
where $(x,y) = (x_{\varepsilon}, y_{\varepsilon})$ is the point we found from Lemma \ref{point-properties}.

Next, we appeal to the multidimensional Dirichlet principle which gives a sequence of
positive integers $q\to \infty$ and lattice points ${\overline{n}} = {\overline{n}}(q) \in {\mathbb Z}^d$ such that
\begin{equation}\label{multi-dimensional}
\| {\overline{n}} - q \, \omega_0 \| \ \le \ q^{-\epsilon_d}
\end{equation}
for some $\epsilon_d > 0$ depending only on $d$. We write ${\overline{n}} = \|{\overline{n}}\| \, \omega$
in polar coordinates
for some $\omega \in {\mathbb S}^{d-1}$. From \eqref{multi-dimensional}, we see that
$\omega \to \omega_0$  and $\|{\overline{n}}\|/q \to 1$ as $q\to\infty$. By the Lebesgue dominated convergence
theorem, fixing $M$ and $N$ and using \eqref{multi-dimensional}, we have
$$ 
\int\!\!\!\int_{{\mathbb T}^2} e^{2\pi i {\overline{n}} \cdot [{\overline{\phi}}(x+s, y+t) + {\overline{L}} \cdot (s,t)]}
\sin(Ms) \sin(Nt) \, \frac{ds}{s} \frac{dt}{t}  \ - \
$$
$$
 \int\!\!\!\int_{{\mathbb T}^2} e^{2\pi i q \omega_0 \cdot [{\overline{\phi}}(x+s, y+t) + {\overline{L}} \cdot (s,t)]}
\sin(Ms) \sin(Nt) \, \frac{ds}{s} \frac{dt}{t} \ \to \ 0 \ \ {\rm as} \ \ q \to \infty.
$$
Hence by taking $q$ large enough (and in particular $q > \Lambda(\varepsilon)$), we can then find
a lattice point ${\overline{n}} \in {\mathbb Z}^d$ and a pair $(M,N)$ 
such that
$$
\Bigl| \int\!\!\!\int_{{\mathbb T}^2} e^{2\pi i {\overline{n}} \cdot [{\overline{\phi}}(x+s, y+t) + {\overline{L}} \cdot (s,t)]}
\sin(Ms) \sin(Nt) \, \frac{ds}{s} \frac{dt}{t}\Bigr|  \ \gtrsim \
\log(1/\varepsilon),
$$
completing the proof of Proposition \ref{logvar} and hence Theorem \ref{main}.

\section{Proof of Theorem \ref{square-sum}}\label{Square-Sum}

In this section we consider mappings $\Phi = \Phi_{\phi, L} : {\mathbb T}^2 \to {\mathbb T}$
parameterised by  $g(s,t) = \phi(s,t) + L_1 s + L_2 t$ where $\phi$ is a real-analytic,
periodic function of two variables and $L = (L_1, L_2)$ is a pair of integers. We will give
a preliminary examination when the property $(\Phi)_{sq}$
holds for $\Phi$; that is, when does $f \circ \Phi $ have a uniformly convergent
Fourier series with respect to square sums for any $f \in A({\mathbb T})$? 
This boils down to showing whether or not 
the norms $\|e^{2\pi i n  g} \|_{U_{sq}({\mathbb T}^2)}$
are uniformly bounded in $n \in {\mathbb Z}$.

Recall that
$$
\|e^{2\pi i n g} \|_{U_{sq}({\mathbb T}^2)} \ = \
\sup_{M} \| S_{M,M} (e^{2\pi in g})\|_{L^{\infty}({\mathbb T})}
$$
$$
= \
\sup_{M,x,y} \ \Bigl| \int\!\!\!\int_{{\mathbb T}^2} e^{2\pi i n 
[\phi(x + s, y + t) + L_1 s + L_2t]} D_M(s) D_M(t) \, ds dt\Bigr|
$$
where $D_M$ denotes the Dirichlet kernel of order $M$. 

Two applications of Proposition \ref{osc-est-III} show that uniform bounds
for the above oscillatory integral are reduced to determining whether or not 
$$
{\mathcal D} \ := \ 
 \int\!\!\!\int_{|s|,|t| \le 1/2} e^{2\pi i n 
[\phi(x + s, y + t) + L_1 s + L_2t]} \frac{\sin(M s)}{s} \frac{\sin(M t)}{t} \, ds  dt
$$
is unformly bounded in the parameters $n, x,y$ and $M$.

\subsection{Proof of Theorem \ref{square-sum} -- existence of uniform bounds} Here we show that there exists
infinitely many pairs $(L_1, L_2)$ such that ${\mathcal D} = O(1)$. This follows from 
the following proposition.
%Writing 
%$\sin(u) = (e^{i u} - e^{-iu}/2i$, we see that ${\mathcal D}'$ splits into four integrals of
%the form
%$$
%{\mathcal D} \ = \
 %\int\!\!\!\int_{|s|,|t| \le 1/2} e^{2\pi i (
%n \phi(x + s, y + t) +[ n  L_1 \pm M]s + [ n L_2 \pm M] t)} \,  \frac{ds}{s} \frac{dt}{t}.
%$$

%The following proposition gives us a proof of Theorem \ref{square-sum}.

\begin{proposition}\label{square-prop} For any pair $(L_1, L_2)$ of integers satisfying
$L_2 > 10^3 L_1$ and $L_1 \ge 10^2 \|\phi\|_{C^2}$, we have ${\mathcal D} = O(1)$.
%$$
%L_1 \ \ge \ 10^2 \, \max\bigl(\|\|partial_s \phi\|_{L^{\infty}}, \|\partial_t \phi\|_{L^{\infty}}, 
\end{proposition}

\begin{proof} Set $A := n (\partial_t \phi(x,y) + L_2) \pm M$ and $B := n (\partial_s \phi(x,y) +  L_1) \pm M$.

{\bf Claim}: \ One of the following two statements hold. Either
$$
(a) \ \ n, |M| \ \lesssim \ |A|, \ \ \ {\rm and} \ \ \ 10 n \|\partial^2_t \phi\|_{L^{\infty}} \ \le \ |A| 
$$
or
$$
(b) \ \ n, |M| \ \lesssim \ |B|, \ \ \ {\rm and} \ \ \ 10 n \|\partial^2_s \phi\|_{L^{\infty}} \ \le \ |B| .
$$

We consider three cases:

{\underline{{\it Case 1}}}: \ Either 
$\bigl| M + n (\partial_t \phi(x,y) + L_2) \bigr| \ge 10 n |\partial_t \phi(x,y) + L_2 |$ \ or
$$
\bigl| M  -  n (\partial_t \phi(x,y) + L_2) \bigr| \  \ge \ 10 n |\partial_t \phi(x,y) + L_2 |.
$$
Either situation implies that $|A| \sim |M|$. Furthermore
$$
n \|\partial^2_t \phi\|_{L^{\infty}} \ \le \ 10^{-5} L_2 n \ \le \ 10^{-3} n |\partial_t (x,y) + L_2 | \ \le \ |A|/10
$$
and this leads to (a). 

{\underline{{\it Case 2}}}: \ Either
$\bigl| M + n (\partial_t \phi(x,y) + L_2) \bigr| \le 10^{-1} n |\partial_t \phi(x,y) + L_2 |$ \ or
$$
\bigl| M  -  n (\partial_t \phi(x,y) + L_2) \bigr| \  \le \ 10^{-1}  n |\partial_t \phi(x,y) + L_2 |.
$$
Either situation implies that $|M| \sim n |\partial_t \phi(x,y) + L_2| \ge 10^2 |\partial_s \phi(x,y) + L_1|$. Hence
$$
|B| \ \sim \ |M| \ \ge \ 10^{-1} n L_2 \ \ge 10^2 n \|\partial^2_s\phi\|_{L^{\infty}}
$$
and this leads to (b). 

{\underline{{\it Case 3}}}: 
$$
10^{-1} n \, |\partial_t \phi(x,y) + L_2 | \ \le \ |M \pm n (\partial_t \phi(x,y) + L_2) | \ \le \ 10 n |\partial_t \phi(x,y) + L_2 |
$$
In this case, we have
$$
|A| \ \sim \ n |\partial_t \phi(x,y) +  L_2 | \ \ge \ 10^{-1} n L_2 \ \ge \ 10^2 n \|\partial^2_t \phi\|_{L^{\infty}}
$$
and so $|M| =$
$$
\bigl| M \pm n (\partial_t \phi(x,y) + L_2) \mp n (\partial_t \phi(x,y) + L_2)\bigr| \le |A| + n |\partial_t \phi(x,y) + L_2|
\lesssim |A|.
$$
This leads to (a), establishing the claim in all three cases.

Without loss of generality suppose that (b) holds. We split ${\mathcal D} = {\mathcal D}_1 + {\mathcal D}_2$
where
$$
{\mathcal D}_1 \ := \ 
 \int\!\!\!\int_{|B|^{-1} \le |s|} e^{2\pi i n [
\phi(x + s, y + t) +   L_1 s + L_2  t]} \,  \sin(Ms) \sin(M t) \, \frac{ds}{s} \frac{dt}{t}.
$$
Using $\sin(u) = (e^{i u} - e^{-iu})/2i$, we 
write ${\mathcal D}_1 = I^{+} + I^{-}$ where 
$$
I^{\pm} \ = \
\int_{\frac{1}{|B|} \le |s| \le 1/2} e^{ i f_{\pm}(s)} \frac{1}{s} \left[
\int_{|t|\le 1/2} e^{2\pi  i n[\phi(x+s, y+t) - \phi(x+s,y) + L_2 t]} \sin(M t) \frac{dt}{t}\right] \, ds
$$
where $f_{\pm}(s) = n (\phi(x+s, y) + L_1 s) \pm M s$. Let us denote by $G(s)$ the inner integral
of $I^{\pm}$ above. By Proposition \ref{osc-est-III}, $G(s) = O(1)$. 

 Integrating by parts in the $s$ integral, we have
$I^{\pm} = BT + II^{\pm} + III^{\pm}$ where $BT$ denote the boundary terms,
$$
II^{\pm} \ = \ 
i \int_{\frac{1}{|B|} \le |s| \le 1/2} e^{ i f_{\pm}(s)} \Bigl( \frac{1}{s {f_{\pm}}'(s)}\Bigr)'  G(s) \, ds
$$
and
$$
III^{\pm} \ = \
i \int_{\frac{1}{|B|} \le |s| \le 1/2} e^{ i f_{\pm}(s)} \frac{1}{s {f_{\pm}}'(s)} G'(s) \, ds
$$

Note that
$$
{f_{\pm}}'(s) \ = \  n (\partial_s \phi(x+s, y) + L_1) \pm M \ = \  B + n s\int_0^1 \partial^2_s \phi(x+ r s, y) dr
$$ 
so that
$$
\frac{1}{s {f_{\pm}}'(s)} \ = \ \frac{1}{s B} \ + \ \frac{1}{s} \frac{B - {f_{\pm}}'(s)}{B {f_{\pm}}'(s)} \ = \ \frac{1}{sB} + 
O\Bigl(\frac{n \|\partial^2_s \phi\|_{L^{\infty}}}{B^2}\Bigr).
$$

From this we see that $BT = O(1)$. Furthermore
$$
{f_{\pm}}''(s) \ = \ n \, \int_0^1 \partial^2_s \phi(x+ rs, y) dr \ + \ n s \, \int_0^1 \partial^3_s \phi(x+ r s, y) r \, dr
$$
and so
$$
\Bigl(\frac{1}{s {f_{\pm}}'(s)} \Bigr)'  =  - \frac{1}{s^2 {f_{\pm}}'(s)} -  \frac{{f_{\pm}}''(s)}{s {f_{\pm}}'(s)^2}  = 
O\Bigl(\frac{1}{|B| s^2} + \frac{n \|\partial^2_s \phi\|_{L^{\infty}}}{B^2 |s|} + 
\frac{n \|\partial^3_s \phi\|_{L^{\infty}}}{B^2} \Bigr).
$$

From this we see that
$$
II^{\pm} \ = \ O(1) \ + \ O(\log |B| / |B|) \ = \ O(1).
$$

Turning to $III^{\pm}$, we note that
$$
\partial_s \phi(x+s, y+t) - \partial_s \phi(x+s,y) \ = \ t \, \int_0^1 \partial^2_{st} \phi(x+s, y + u t) du
$$ 
and so $G'(s) \ =$
$$
2 \pi i n \int_0^1 \left[ \int_{|t|\le 1/2} e^{2\pi  i n[\phi(x+s, y+t) - \phi(x+s,y) + L_2 t]} 
\partial^2_{st} \phi(x+s, y + u t)  
\sin(M t) dt\right] du,
$$
showing that $G'(s) = O(n)$. Our formula for 
$(s {f_{\pm}}'(s))^{-1}$ shows that
$$
III^{\pm} \ = \ O(n^2/|B|^2) + \frac{n}{B} 
 \int_0^1
\int_{|t| \le 1/2} e^{2 \pi i n L_2 t} \sin(Mt) \, H(t) \, dt \, du
$$
where
$$
H(t) \ := \
 \int_{\frac{1}{|B|} \le |s| \le 1/2} e^{ 2 \pi i g_{\pm}(s,t)}  
\partial^2_{st} \phi(x+s, y+ u t) \, \frac{ds}{s}
$$ 
and $g_{\pm}(s,t) =  n (\phi(x+s, y+t) + L_1 s) \pm M s$.

We write
$$
\partial^2_{st} \phi(x+s, y + u t) = \partial^2_{st} \phi(x, y + u t) + s \int_0^1
\partial^{2,1}\phi(x + r s, y + u t) dr
$$
and so
$$
H(t) \ = \  \partial^2_{st} \phi(x, y + u t)
\int_{\frac{1}{|B|} \le |s| \le 1/2} e^{ 2 \pi i g_{\pm}(s,t)} \frac{ds}{s} \ + \ O(1) \ = \ O(1)
$$
by Proposition \ref{osc-est-III}. Hence $III^{\pm} = O(1)$ and so $I^{\pm} = O(1)$, showing
${\mathcal D}_1 = O(1)$. 

We are left with
$$
{\mathcal D}_2 \ = \
 \int\!\!\!\int_{|s| \le |B|^{-1}} e^{2\pi i n [
\phi(x + s, y + t) +   L_1 s + L_2  t]} \,  \sin(Ms) \sin(M t) \, \frac{ds}{s} \frac{dt}{t}
$$
which can be written as
$$
\int_{|s| \le \frac{1}{|B|}} e^{2\pi i n L_1 s} \frac{\sin(M s)}{s} \left[
\int_{|t|\le 1/2} e^{2\pi i n (\phi(x+s, y+t) + L_2 t)} \frac{\sin{M t}}{t} dt \right] ds.
$$
The inner integral is uniformly bounded by Proposition \ref{osc-est-III} and so
${\mathcal D}_2 = O(|M|/|B|) = O(1)$ by (b). This completes the proof of the proposition,
showing ${\mathcal D} = O(1)$. 
\end{proof}

\subsection{Examples when $(\Phi)_{sq}$ fails}

From Theorem \ref{main}, we know that $(\Phi)_{sq}$ can only fail for
maps $\Phi = \Phi_{\phi, L}$ whose periodic part $\phi$ does not satisfy
the factorisation hypothesis (FH). 

As mentioned at the outset of this section, the
property $(\Phi_{\phi, L})_{sq}$ fails if we can show that the oscillatory integral
$$
{\mathcal D} \ := \ 
 \int\!\!\!\int_{|s|,|t| \le 1/2} e^{2\pi i n 
[\phi(x + s, y + t) + L_1 s + L_2t]} \frac{\sin(M s)}{s} \frac{\sin(M t)}{t} \, ds  dt
$$
is unbounded  in the parameters $n, x,y$ and $M$. 

Suppose now  (FH) fails for $\phi$. Then there exists a point $(x_0,y_0) \in {\mathbb T}^2$ such that
\eqref{FH-fail} holds; that is, $\phi_{st}(x_0,y_0) \not= 0$ yet $\phi_{tt}(x_0,y_0) = 0$, say. 
In addition we will assume that we can find such a point such that the difference
$\partial_s \phi(x_0, y_0) - \partial_t \phi(x_0,y_0)$ is an integer. 
There are many such examples; for instance $\phi(x,y) = \cos(x) (-1 + \sin(y))$
has these properties for $(x_0, y_0) = (\pi/4, 0)$. 
We  fix the point $(x,y)$ in
${\mathcal D}$ above to be $(x_0, y_0)$.

We choose any pair $L = (L_1, L_2)$ of integers so that
$$
\partial_s \phi(x_0,y_0) + L_1  \ = \ \partial_t \phi(x_0, y_0) +  L_2  \ \not= \ 0.
$$
Finally for $n$ large we choose the integer $M = M(n)$ such that
$$
n (\partial_s \phi(x_0, y_0) + L_1)  -  M \ = \ n (\partial_t \phi(x_0,y_0) +  L_2 ) - M \ = \ O(1).
$$ 

\vskip 15pt

\begin{proposition}\label{example} With the point $(x_0, y_0) \in {\mathbb T}^2$ and pair
of  integers $L =  (L_1, L_2)$ described above, we can find an $N$ such that whenever
$n\ge N$, we have (with the above $M = M(n)$)
\begin{equation}\label{D} 
|{\mathcal D}| \ \gtrsim \ \log n.
\end{equation}
Hence $(\Phi)_{sq}$ fails for $\Phi = \Phi_{\phi, L}$.
\end{proposition}

The proof of Proposition \ref{example} is much easier than the analysis required to show
 the corresponding unboundedness in Theorem \ref{main}
when (FH) fails to hold for $\phi$. Of course there we established unboundedness
for {\it all} pairs of integers $L = (L_1, L_2)$ (although we did have two integers
$M$ and $N$ to choose instead of a single integer $M$ in this case).
We no longer have to perturb the point $(x_0,y_0)$ and carry out such an intricate argument.
Much of the analysis, even though simpler, is the same as in Section \ref{main-fail} and so
we will be brief and only outline the steps.

{\bf Step 1}: \ Let $m_0\ge 3$ be an integer such that $\partial^{m_0}_t \phi (x_0,y_0) \not= 0$
and furthermore we take $m_0$ to be minimal with these properties. Recall that $\partial^2_t \phi(x_0,y_0) = 0$.
Also let $n_0 \ge 2$ be an  integer such that $\partial^{n_0}_s \phi (x_0,y_0) \not= 0$
and furthermore we take $n_0$ to be minimal with these properties.  It  may be the case that
$m_0$ and/or $n_0 = \infty$. We allow this possibility which only will exclude certain successive steps
from occuring.

{\bf Step 2}: \ Reduction to the oscillatory integral
$$
 \int\!\!\!\int_{R} e^{2\pi i n 
[\phi(x + s, y + t) + L_1 s + L_2t]} \frac{\sin(M s)}{s} \frac{\sin(M t)}{t} \, ds  dt
$$
where
$$
R \ = \ \{(s,t) : |t|^{m_0}, \ |s|^{n_0} \ \le \ |s t| \, \}
$$
The reduction is carried out in Section \ref{prelude-I} with $(m_1, n_1) = (1,1)$.
Recall in the general $(m_1, n_1)$ case, we reduced matters to an oscillatory integral
with {\it almost polynomial} phase $P_{x_0,y_0}(x+s, y+t) + \phi(x+s, y) + \phi(x, y+t)$.
When $(m_1, n_1) = (1,1)$, we see that $P_{x_0, y_0} \equiv 0$ and so an application
of Proposition \ref{osc-est-III} allows us to complete the analysis over the
region $\{ |s t| \le |t|^{m_0} \} \cup \{ |s t| \le |s|^{n_0}\}$, reducing matters to the above integral.

Here it
is essential that $m_0 \ge 3$. Otherwise if $m_0 = n_0 = 2$, then $R = \emptyset$.

{\bf Step 3}: \ Reduction to an oscillatory integral with quadratic phase
$$
 \int\!\!\!\int_{R} e^{2\pi i n 
[(\partial_s \phi(x_0, y_0) + L_1) s + (\partial_t \phi(x_0,y_0)  + L_2)t + \partial_{st}\phi(x_0,y_0) st]} \frac{\sin(M s)}{s} \frac{\sin(M t)}{t} \, ds  dt
$$
 A more complicated
reduction is used in the proof of Theorem \ref{main} as carried out in Section \ref{main-fail}.

{\bf Step 4}: \ Reduction to
$$
 \int\!\!\!\int_{R^{+}} \frac{\sin( c n s t)}{s t} \, ds dt
$$
where $c = \partial_{st} \phi(x_0, y_0)$ and $R^{+} = \{s,t\ge 0: t^{m_0 -1} \le s \le t^{n_0/(n_0 - 1)} \}$.

This reduction is similar but much easier than what we did in Section \ref{main-fail};
see in particular the analysis leading up to Section \ref{heart}. It is here where our choice of $L = (L_1, L_2)$
(which depends on $\nabla \phi (x_0, y_0)$) and $M = M(n, L, \nabla \phi(x_0,y_0))$ comes into play.

{\bf Step 5}: \ 
$$
 \int\!\!\!\int_{R^{+}} \frac{\sin( c n s t)}{s t} \, ds dt \ = \ \frac{\pi}{2} \log n^{1 - \frac{1}{n_0} - \frac{1}{m_0}} \ +
\ O(1).
$$
This computation is similar but easier than what we did in Section \ref{heart}. As we said before, we allow
for $m_0$ and/or $n_0 = \infty$. This does not prevent the logarithmic blow up in $n$. 

These steps give a proof  \eqref{D} and therefore  Proposition \ref{example}.

\section{Higher dimensional tori ${\mathbb T}^k$}

Recall that a map $\Phi : {\mathbb T}^k \to {\mathbb T}^d$ is parametrised by $d$ periodic functions
${\overline{\phi}} = (\phi_1, \ldots, \phi_d)$ of $k$ variables ${\vec{t}} = (t_1, \ldots, t_k)$ and
$d$ lattice points ${\overline{L}} = ({\vec{L}}_1, \ldots, {\vec{L}}_d)$ in ${\mathbb Z}^k$. Explicitly,
if $P = (e^{2\pi i t_1}, \ldots, e^{2\pi i t_k})$ is a point in ${\mathbb T}^k$, then
$$
\Phi(P) \ = \ \bigl( e^{2\pi i [\phi_1({\vec{t}}) + {\vec{L}}_1 \cdot {\vec{t}}]}, \ldots,
e^{2\pi i [\phi_d({\vec{t}}) + {\vec{L}}_d \cdot {\vec{t}}]} \bigr).
$$
Freezing one of the $k$ variables, say $t_j$, gives a map $\Phi_j : {\mathbb T}^{k-1} \to {\mathbb T}^d$.

Furthermore if $(\Phi)_{rect}$ holds, then $(\Phi_j)_{rect}$ holds for every $1\le j\le k$. This follows easily
from the definition of {\it unrestricted} rectangular convergence. We illustrate this when $k=2$: if $f\in C({\mathbb T}^2)$,
then $f$ has a uniformly convergent Fourier series with respect to rectangular convergence if
$$
\sup_{x,y} \bigl| \sum_{|k|\le M, |\ell| \le N} {\hat f}(k,\ell) e^{2\pi i (k x + \ell y)} \ - \ f(x,y) \bigr| \ \to \ 0
\ \ {\rm as} \ \ M, N \to \infty.
$$
In particular (letting $N\to \infty$) this says that for any $y \in {\mathbb T}$,
$$
\sup_x \bigl|  \sum_{|k|\le M} {\hat f_y}(k) e^{2\pi i kx } \ - \ f_y(x) \bigr| \ \to \ 0
\ \ {\rm as} \ \ M\to \infty
$$
where $f_y \in C({\mathbb T})$ is defined as $f_y(x) = f(x,y)$. That is, $f_y$ has a uniformly convergent
Fourier series for every $y\in {\mathbb T}$. 

Hence whatever conditions we impose on $\Phi$ to guarantee that $(\Phi)_{rect}$ holds, then necessarily
$(\Phi_j)_{rect}$ holds and in particular any necessary conditions we happen to know regarding $(\Phi_j)_{rect}$
give us necessary conditions on $\Phi$. For instance for mappings $\Phi : {\mathbb T}^3 \to {\mathbb T}^d$
such that $(\Phi)_{rect}$ holds, then
necessarily 
$$
\Phi_1 (e^{2\pi i t}, e^{2\pi i u}) = \Phi(e^{2\pi i s}, e^{2\pi i t}, e^{2\pi i u}) = 
\Phi_2 (e^{2\pi i s}, e^{2\pi i u}) = 
\Phi_3 (e^{2\pi i s}, e^{2\pi i t}) 
$$
satisfy $(\Phi_1)_{rect}, (\Phi_2)_{rect}$  and $(\Phi_3)_{rect}$, respectively. Threfore if
${\overline{\phi}}(s,t,u)$ parametrises the periodic part of $\Phi$,
Theorem \ref{main} implies that for every $u \in {\mathbb T}$, $\psi_u(s,t,\omega) := 
\omega \cdot {\overline{\phi}}(s,t,u)$ satisfies the factorisation hypothesis (FH) on
${\mathbb S}^{d-1}$. Likewise $\psi_s$ and $\psi_t$ necessarily satsifies (FH). These conditions
allow us to control the particular partial derivatives
$$
\frac{\partial^{k+\ell}\psi}{\partial^k s \, \partial^{\ell} t}, \ \ 
\frac{\partial^{k+\ell}\psi}{\partial^k s \,  \partial^{\ell} u}, \ \
\frac{\partial^{k+\ell}\psi}{\partial^k t \, \partial^{\ell} u}
$$
for any $k,\ell \ge 1$
in terms of pure derivatives. This is the content of Lemma \ref{relations}. For our argument it is essential 
that we are able to control all mixed derivatives by pure derivatives. So when $k=3$, it remains
to control the third order mixed derivative $\partial_{stu} \psi(s,t,u,\omega)$. This leads to our 
characterisation of when $(\Phi)_{rect}$ holds when $k=3$.

\begin{theorem}\label{k=3} Let $\Phi : {\mathbb T}^3 \to {\mathbb T}^d$ be a real-analytic
map and suppose ${\overline{\phi}}$ is its periodic part. Set $\psi(s,t,u) = \omega \cdot {\overline{\phi}}(s,t,u)$.
Then $(\Phi)_{rect}$ holds if and only if
$$
\psi_{ss}, \, \psi_{tt} \ | \ \psi_{st}, \ \ \psi_{ss}, \, \psi_{uu} \ | \ \psi_{su}, \ \
\psi_{tt}, \, \psi_{uu} \ | \ \psi_{tu}
$$
and
$$
\psi_{sss}, \,  \psi_{ttt},  \, \psi_{uuu} \ | \ \psi_{stu}
$$
as germs of real-analytic functions of ${\mathbb T}^3 \times {\mathbb S}^{d-1}$.
\end{theorem}

Proceeding iteratively we arrive at a characterisation of real-analytic mappings $\Phi$ from ${\mathbb T}^k$
to ${\mathbb T}^d$ such that $(\Phi)_{rect}$ holds for any $k\ge 1$. The proof is much more technical than
Theorem \ref{main} but essentially all the ideas are present in the $k=2$ case and this is the reason we have
decided to carry out the analysis only in this case. We may present the details in the general case at a later time.

\end{document}